\documentclass{article}
\usepackage{graphicx} %
\usepackage{xparse}

\usepackage{latexsym,amssymb,amsmath,amsfonts,amsthm}
\numberwithin{equation}{section}
\usepackage{enotez,hyperref,float}
\usepackage{forest}
\usepackage{thmtools, thm-restate}
\usepackage{mathtools}

\usetikzlibrary{arrows.meta, positioning}
\usetikzlibrary{decorations.markings}

\usepackage{color}
\usepackage{epsfig}
\usepackage{enumerate}
\usepackage{comment}

\usepackage{svg}

\newtheorem{theorem}{Theorem}

\newtheorem{corollary}[theorem]{Corollary}

\newtheorem{lemma}{Lemma}

\graphicspath{ {./assets/} }

\newtheorem{proposition}{Proposition}
\newtheorem*{proposition*}{Proposition}
\newtheorem{remark}{Remark}

\DeclareMathOperator{\supp}{supp}

\newcommand{\gsimas}[2]{\overset{{#1}\to {#2}}{\gtrsim}}
\newcommand{\lsimas}[2]{\overset{{#1}\to {#2}}{\lesssim}}
\newcommand{\simas}[2]{\overset{{#1}\to {#2}}{\sim}}
\newcommand{\ggas}[2]{\overset{{#1}\to {#2}}{\gg}}
\newcommand{\llas}[2]{\overset{{#1}\to {#2}}{\ll}}
\newcommand{\ordas}[2]{\overset{{#1}\to {#2}}{\asymp}}

\def\beq{ \begin{equation} }
\def\eeq{ \end{equation} }

\def\square{\vcenter{\vbox{\hrule height .4pt
  \hbox{\vrule width .4pt height 5pt \kern 5pt
        \vrule width .4pt} \hrule height .4pt}}}

\def\sqz{\kern-0.2em}

\newcommand{\tabeq}[3]{%
  \begin{aligned}[t]
    \makebox[#1][l]{$#2$} &= #3
  \end{aligned}%
}

\usepackage{graphicx}
\graphicspath{%
    {converted_graphics/}%
    {/}%
}

\newtheorem{assumptionstar}{Assumption}

\usepackage[
backend=biber,
style=numeric,
sorting=none
]{biblatex}
\usepackage{authblk}

\title{The Site Frequency Spectrum in an Exponentially Growing Population with Selection}
\author[1,2,3]{Armaan Ahmed}
\author[4]{Jasmine Foo}
\author[5]{Einar Bjarki Gunnarsson}
\author[6]{Kevin Leder}

\affil[1]{Department of Applied Mathematics \& Statistics, Johns Hopkins University}
\affil[2]{Department of Mathematics, Johns Hopkins University}
\affil[3]{Department of Statistics, University of California, Berkeley}
\affil[4]{Department of Mathematics, University of Minnesota}
\affil[5]{Division of Applied Mathematics, Science Institute, University of Iceland}
\affil[6]{Department of Industrial and Systems Engineering, University of Minnesota}

\usepackage{nicematrix}

\usepackage{lineno}

\providecommand{\classification}[2]
{
  \small	
  \textbf{\textit{#1---}} #2
}

\begin{document}

\maketitle
\begin{abstract}
  We consider a supercritical two-type continuous-time linear birth-death process with mutation and selection, in which wild-type individuals give rise to mutant offspring with a larger net growth rate $\lambda_1$. In this setting, we investigate a component of the site frequency spectrum (SFS), describing the number of driver mutations present at any given frequency in the population.
  We establish asymptotic convergence results at large times and detection sizes. We examine three distinct regions of the SFS: (1) the small frequency region occupied by clones of size $\mathcal{O}(1)$, (2) the intermediate frequency region occupied by clones of size $1\ll j(t)\ll \exp(\lambda_1 t)$, and (3) the large frequency region occupied by clones of size at least $\exp(\lambda_1 t).$
  We obtain strong laws of large numbers for the driver SFS in the small- and large-frequency regions by constructing suitable $L^2$-approximations. In the intermediate frequency region, we identify a cutoff frequency $j(t)$ growing with time at which there are an order-one number of mutant clones. At this frequency, we obtain a Poisson convergence theorem as $t\to\infty$.
  Overall, our results provide quantitative insights into how selection shapes the site frequency spectrum both at small and large frequencies, which can in principle be leveraged to construct estimators of relevant evolutionary parameters. As an example of this, we propose and show asymptotic consistency of a new estimator of the fitness increase. 
\end{abstract}

\classification{Keywords}{Site Frequency Spectrum, Birth-Death Process, Cancer Evolution}

\classification{MSC2020}{Primary: 60J80. Secondary: 92D10, 92D25, 60F15}

\section{Introduction}

Human cells accumulate mutations throughout the lifetime of a person due to various factors such as DNA replication errors, chronic inflammation, and decreased DNA repair efficacy \cite{takeshima2019}. 
So-called `driver' mutations confer a fitness advantage to a cell, enabling clonal expansion and accumulation of additional mutations, leading to premalignant or cancerous cell populations \cite{bielas2006,hanahan2022}. Understanding the genetic tumor heterogeneity remains an important goal as it can provide clinically-relevant information regarding the mechanisms of tumor formation, tumor classification, prediction of therapeutic response, etc. \cite{predicting,salvadores2019,youn2013}

A simple summary metric of the genetic heterogeneity of a population is the site frequency spectrum (SFS), denoted $S_j(t)$, counting the number of distinct mutations found in $j$ members of the population at time $t$. 
The SFS can be used to estimate evolutionary parameters of a phylogeny (mutation, birth, death rates) and to provide evidence for either neutral or selective evolution \cite{tung,gunnarsson2021,tanaka2023,gunnarsson2025}.

In this work, we characterize the behavior of the SFS in the setting of a two-type continuous-time linear birth-death process. The initial population consists of type-0 (wild-type) cells growing at Malthusian rate $\lambda_0>0$ which can give rise to type-1 cells growing at Malthusian rate $\lambda_1>\lambda_0$ through the acquisition of a selectively advantageous driver mutation. In this setting, $S_j(t)$ refers to the number of distinct type-1 clones of size $j$ at time $t$ and can be called the ``driver SFS''. Many others have investigated frequencies of neutral mutations in similar branching process or semi-deterministic Cox process models of evolution \cite{gunnarsson2021,gunnarsson2025,cheek2018,durrett2013,bozic2016,williams2016,OHTSUKI201743,stein2025,bonnet2024}.

The main aim of this paper is to study the SFS in large time and detection size limits. Specifically, we investigate three broad regions of the SFS: the \textit{small}, \textit{intermediate}, and \textit{large} frequency regions. The small frequency region $S_j(t)$ is constituted by clones of constant size $j=\mathcal{O}_{t\to\infty}(1)$. The intermediate frequency regime $S_{j(t)}(t)$ is constituted by clones of size $j(t)=o_{t\to\infty}(\exp(\lambda_1 t))$. The large frequency region $\sum_{k\geq j(t)}S_k(t)$ is constituted by the very largest clones of size $j(t)=\Theta_{t\to\infty}(\exp(\lambda_1 t))$. We obtain limit theorems for the number of clones occupying each of these regions. 

In the setting where an evolving population can acquire driver mutations, there are many studies examining the bulk dynamics of the total driver-mutant population (arrival times of mutants, asymptotic total mutant population size, etc.) \cite{iwasa2006,durrett2010,nicholson2023}. However, the dynamics of \textit{each distinct} clone are what matter in  quantifying $S_j(t)$. In view of this, there are certain important studies. Tavar\'e explored a model in which independent Yule processes are initiated at times distributed according to a homogeneous Poisson point process \cite{tavare1987}. One interpretation is that a critical, degenerate type-0 population produces supercritical type-1 mutant clones. The scaled asymptotic sizes of the clones were obtained and studied as a point process.

Cheek and Antal also provide results concerning the driver SFS and the largest clone under selective evolution in the large time-small mutation/large detection size-small mutation regimes \cite{cheek2018}. In the small frequency region and in their scaling regime, only $\mathcal{O}(1)$ clones contribute to $S_j(t)$ whereas in our large time or detection size regime infinitely many clones contribute. Consequently, they obtain Poisson limit theorems whereas we obtain law of large numbers limit theorems for the SFS. This is discussed in Section \ref{sec:time_freq_scaling} in more detail.

Finally, Tung and Durrett recently examined the large frequency end of the SFS, describing the number of clones that constitute a proportion larger than $f\in (0,1)$ of the total mutant population \cite{tung}. They find that in the large-time limit, the expected number of clones with frequency $>f$ follows a $f^{-\lambda_0/\lambda_1}$ power-law where $\lambda_0$ (or $\lambda_1$) is the net growth rate of the wild-type (or mutant) population. We will also obtain results in this large-frequency end of the driver SFS and compare them to the results of Tung and Durrett in Section \ref{sec:large_frequencies}.

\subsection{Outline}

In this paper we study the behavior of the driver SFS ($S_j(t)$) in a two-type continuous-time linear birth-death process, under the assumption that cells can accumulate at most one driver mutation. We introduce the model in more detail along with notation in Section \ref{sec:model_notation}. Our main results are stated in Section \ref{sec:results}. These results are split into three subsections corresponding to each of the ``small'', ``intermediate'', and ``large''-frequency regions of the SFS. Within each subsection, we state our main results for that region of the SFS, compare them to results from the literature, and provide the basic proof outline. The full proofs are deferred to Section \ref{sec:proofs}. In Section \ref{sec:time_freq_scaling}, we extend analyses from Gunnarsson et al. and Harris (\cite{gunnarsson2025} and \cite{harris1963}, respectively) to obtain almost sure convergence results for $S_j(t)$ in the large time and detection size limits when the fitness advance for each clone is deterministic, and a large-time limit for random fitness advances (Theorem \ref{thm:convergence_results}). In Section \ref{sec:SFSVary}, we show that $j(t)=\exp((\lambda_0\lambda_1/(\lambda_0+\lambda_1))t)$ sets a cutoff frequency at which $S_{j(t)}(t)=\mathcal{O}(1)$ (Theorem \ref{thm:SFSSparsityAS}). We also consider contributions of clones growing faster or slower than this frequency (Proposition \ref{prop:SFSSparsity_quant} and Corollary \ref{cor:SFSSparsity_qual}). In Section \ref{sec:large_frequencies}, we provide convergence results for the number of clones above a certain fraction of the population (Theorem \ref{thm:relative_SFS}). In the small mutation limit, these results show agreement with some previously published results obtained in a semideterministic version of our model (see \cite{durrettgenetics2010,tung} and Proposition \ref{prop:relative_SFS_small_mut}).

\section{Models and notation}\label{sec:model_notation}

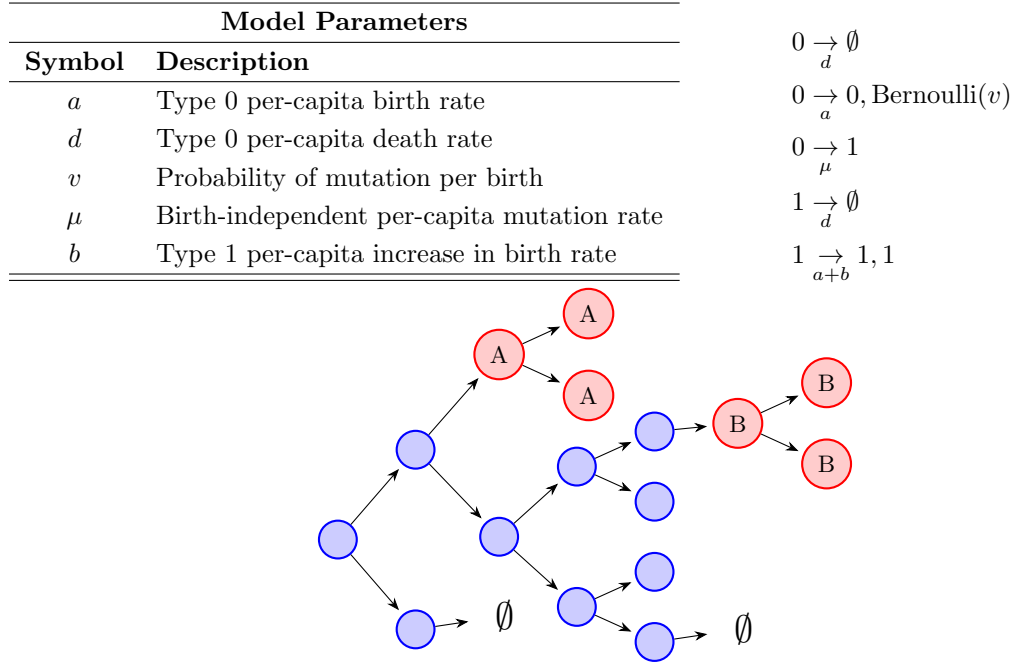
\begin{figure}[H]
    \centering 

    \begin{minipage}{0.45\textwidth}
        \centering
        \renewcommand{\arraystretch}{1.2}
        \begin{tabular}{cl}
            \hline\hline
            \multicolumn{2}{c}{\textbf{Model Parameters}} \\
            \hline
            \textbf{Symbol} & \textbf{Description}      \\
            \hline
            $a$           & Type 0 per-capita birth rate    \\
            $d$           & Type 0 per-capita death rate     \\
            $v$             & Probability of mutation per birth     \\
            $\mu$           & Birth-independent per-capita mutation rate   \\
            $b$             & Type 1 per-capita increase in birth rate   \\
            \hline\hline
        \end{tabular}
    \end{minipage}
    \hfill 
    \begin{minipage}{0.45\textwidth}
        \begin{align*}
            0&\underset{d}{\to}\emptyset\\
            0&\underset{a}{\to}0,\text{Bernoulli}(v)\\
            0&\underset{\mu}{\to}1\\
            1&\underset{d}{\to}\emptyset\\
            1&\underset{a+b}{\to}1,1
        \end{align*}
    \end{minipage}

    \begin{minipage}{\textwidth}
      \centering
      \begin{forest}
  red node/.style={content={#1}, draw=red, fill=red!20},
  for tree={
    grow=east,
    circle,
    draw=blue,
    fill=blue!20,
    thick,
    l sep=5mm,
    s sep=4mm,
    minimum size=0.5cm,
    edge={-{Stealth[]}, shorten >=2pt},
  }
  [
    [
      [, content={\Large$\emptyset$}, draw=none,fill=none]
    ]
    [
      [
        [
          [[, content={\Large$\emptyset$}, draw=none,fill=none]]
          []
        ]
        [
          []
          [[, red node={\small B}, for children={red node={2}}
            [\small B]
            [\small B]]]
        ]
      ]
      [, red node={\small A}, for children={red node={2}}
        [\small A]
        [\small A]]
    ]
  ]
\end{forest}
    \end{minipage}
    \caption{Parameter descriptions (top left) and model scheme (top right). An example graphical representation of the model is shown in the bottom panel. Blue circles represent type-0 cells and red circles represent type-1 cells. $\emptyset$ denotes death events and the differently labeled type-1 cells indicate which clone they belong to. The leaves (cells with no out-going arrows) are the cells currently alive.}
    \label{fig:model}
\end{figure}

Consider a continuous-time Markov branching process with two cell types, denoted type-0 and type-1 \cite{athreya1972}. Each type-0 cell independently divides at exponential rate $a$ and dies at rate $d$,  and the total type-0 population count at time $t$ is denoted $Z_0(t)$. Each type-1 cell  divides at rate $a+b$ and dies at rate $d$, where $b>0$ reflects a selective advantage. The total type-1 population count at time $t$ is denoted $Z_1(t)$. Unless otherwise specified, we begin with the initial condition $(Z_0(0),Z_1(0))=(1,0)$.

Type-0 cells produce type-1 cells through two mechanisms of driver mutation acquisition, (i) {\it replication-dependent:} at each type-0 division exactly one of the two offspring is type-1 with probability $v$, otherwise both are type-0 \footnote{Here, we are making a simplification and ignoring the possibility that both offspring are type-1 cells. This scenario occurs with probability $\mathcal{O}_{v\to 0^+}(v^2)$.}, and (ii) {\it replication-independent:} each type-0 cell mutates to type-1 at rate $\mu$, independently of division \cite{abascal2021}. Thus, the effective birth rate of type-0 cells is $a(1-v)$ and the net growth rate is $\lambda_0 \equiv a(1-v)-d-\mu$. The net growth rate of type-1 cells is $\lambda_1 \equiv a+b-d$. Throughout, we assume supercriticality of the type-$0$ population, namely $\lambda_0>0$. See Figure \ref{fig:model} for a schematic of this model. Observe that clone A is founded by a mutant produced in a replication-dependent fashion and clone B is founded by a mutant produced in a replication-independent fashion. 

In some sections we will assume that driver mutations are replication-dependent ($\mu=0$) and take the small mutation limit $v\to 0^+$. In this case, the notation $\lambda_{0,v}$ will be used to explicitly denote the dependence of $\lambda_0$ on $v$. Denote the probabilities of extinction of a type-0 and of a type-1 clone, each starting from one cell of the respective type, as $p_0=\frac{d+\mu}{a(1-v)}$ and $p_1=\frac{d}{a+b}$, respectively.  These follow from the classical extinction probability formula for supercritical birth-death processes.  For $i\in \{0,1\}$, define $q_i=1-p_i$ as the probability that a clone started by a type-$i$ cell is established, i.e.~survives forever. Where necessary, we will condition on non-extinction, and thus define $\Omega_i^\infty:=\{\exists N\text{ s.t. }t\geq N\text{ implies }Z_i(t)>0\}$ for $i\in \{0,1\}$ as the event that the type-$i$ population does not go extinct. Remark that $ \Omega_0^\infty\subseteq \Omega_1^\infty$ (up to null sets).

Let $S_j(t)$ denote the number of type-1 driver clones of size $j$ at time $t$. $S_j(t)$ is usually referred to as the driver SFS or simply the SFS in this paper. Here, we make the infinite sites assumption, so for example the mutations associated with clones A and B in Figure \ref{fig:model} are distinct\footnote{Before defining $S_j(t)$ more carefully, we clarify the relation between mutations at a given frequency versus clones at a given size. Firstly, a type-1 clone is the tree of descendants rooted at a type-1 founder cell (i.e., a type-1 cell whose parent is a type-0 cell). Associated to each clone is a selective mutation that its founder cell acquired. It may be that different clones have acquired the same selective mutation. If clones $A$ and $B$ in Figure \ref{fig:model} both acquired the same mutation, then at the time of detection clones $A$ and $B$ would both have size $2$ whereas the frequency of the mutation would be $2+2=4$. We avoid this by making the infinite sites assumption, so the frequency of the mutation associated with clone $i$ at time $t$ is the same as the number of cells in clone $i$ at time $t$.}. Define $T_i$ as the initiation time of the $i$th driver clone produced, where $T_i=\infty$ if the $i$th clone is never produced, and define $M_t=\#\{i\in\mathbb{N}:T_i\leq t\}$ as the number of mutant lineages ever produced by time $t$. 
We let $Z_1^{(i)}(s)$ denote the size of the $i$th driver clone $s$ time units after it was produced, and note that 
\[
S_j(t)=\sum_{i:T_i\leq t} \mathbf{1}_{\{Z_{1}^{(i)}(t-T_i)=j\}}.
\]
As our branching process is reducible conditional on the trajectory of the type-0 population, $s\to Z_1^{(i)}(s)$ are i.i.d linear birth-death processes with birth rate $a+b$ and death rate $d$. If $s<0$, $Z_1^{(i)}(s)\equiv 0$.
We will also examine the {\it relative driver SFS}, the number of driver clones above a certain frequency: $R_f(t)=\sum_{j:j>Z_1(t)f}S_j(t)$ for $0<f<1$.

We will investigate features of the driver SFS at a fixed time $t$ and at the random time when the population reaches a fixed detection size $n$, $\tau_n=\inf \{t:Z_0(t)+Z_1(t)=n\}$. On $\Omega_1^\infty$, $\tau_n < \infty$ almost surely. Fixed time results are relevant when considering population growth in a controlled setting. Fixed size results are more relevant in the clinical setting, as a tumor may be readily detected upon reaching a particular size. 

Certain limit theorems are known for $Z_0(t),Z_1(t),\{Z_1^{(i)}(t)\}_{i\in\mathbb{N}}$. For supercritical birth-death processes, it is well-known that $e^{-\lambda_0 t}Z_0(t)$ converges almost surely and in $L^2$ as $t \to \infty$.
The limiting random variable $Y \equiv \lim_{t \to \infty} e^{-\lambda_0 t}Z_0(t)$ is 0 with probability $p_0$, and on the event of survival it is exponential with rate $q_0$ (for example see Theorem 1 in \cite{durrettbranching2015}). Similarly, $Y_1^{(i)}\equiv\lim_{t\to\infty}e^{-\lambda_1 t}Z_1^{(i)}(t)$ exists almost surely and in $L^2$. The $Y_1^{(i)}$ are i.i.d. random variables which are 0 with probability $p_1$ and exponential with rate $q_1$ with probability $q_1$. For the total type-1 population, $W\equiv\lim_{t\to\infty}e^{-\lambda_1 t}Z_1(t)=\sum_{i\in\mathbb{N}}e^{-\lambda_1 T_i}Y_1^{(i)}$ exists almost surely on the event of nonextinction (see Section \ref{app:W_representation}).

{\bf \em Random Fitness Advance Model.}
We will also consider a variant of the above model in which we allow each type-1 clone to have random fitness advances. In particular, consider a random perturbation to the birth rate of the $i$-th type-1 clone, $B_i$, where $B_j\sim \mathbb{Q}$ are i.i.d.~for all $j\in \mathbb{N}$ for some distribution $\mathbb{Q}$, and the $B_j$'s are independent of the type-0 process $Z_0(t)$.  Conditional on $B_i$, the birth rate of the $i$-th type-1 clone is $a+B_i$ while the death rate remains $d$. %
We refer to this modified model as the random fitness advance model, and we denote its SFS by $S_j^B$.
We note that a semi-deterministic version of this model was also studied by the second and fourth authors in various previous works \cite{durrettfoo2010, durrettgenetics2010, foo2014escape}. The following bounded support assumption on $\mathbb{Q}$ is sometimes used in our results for this random fitness model:

\begin{assumptionstar}
    $\mathbb{Q}$ has density $h$ with support $[0,b_{\max}]$ which is left-continuous at $b_{\max}$, $h(b_{\max})>0$, and $h$ is bounded.\label{assumption:random}
\end{assumptionstar}

A summary of all the important notation is provided in Table \ref{table:notation}.
\begin{table}[H]
\renewcommand{\arraystretch}{1.5}
  \centering  
  \begin{tabular}{|c p{6cm} p{6cm}|}
\hline
\textbf{Symbol} & \textbf{Description} & \textbf{Mathematical Description}\\
\hline
$Z_0(t)/Z_1(t)$ & Count of total type-0/type-1 population at time $t$ & \rule[0.5ex]{1cm}{0.4pt}\\
$\Omega_i^\infty$ & Non-extinction event for the type-$i$ population, where $i\in \{0,1\}$, population & $\tabeq{0.9cm}{\Omega_i^\infty}{\{\omega:\exists N\geq 0,\\&t\geq N\Longrightarrow Z_i(t)>0\}}$\\
$T_i$ & Timing of $i$th mutation event & \rule[0.5ex]{1cm}{0.4pt}\\
$M_t$ & Counting process of number of mutation events in $[0,t]$ & $\tabeq{0.9cm}{M_t}{\#\{i\in \mathbb{N}:T_i\leq t\}}$\\
$Z_1^{(i)}(s)$ & Count of cells in the $i$th type-1 clone, $s$ time units after it was produced & $\tabeq{0.9cm}{Z_1(t)}{\sum_{T_i\leq t}Z_1^{(i)}(t-T_i)\\
       &= \sum_{i\leq M_t}Z_1^{(i)}(t-T_i)}$\\
$Y/Y_1^{(i)}/W$ & Scaled limit of the type-0 population/ $i$th type-1 clone / type-1 population &  $\tabeq{0.9cm}{Y}{\lim_{t\to\infty}e^{-\lambda_0 t}Z_0(t)}$\newline$\tabeq{0.9cm}{Y_1^{(i)}}{\lim_{t\to\infty}e^{-\lambda_1 t}Z_1^{(i)}(t)}$\newline$\tabeq{0.9cm}{W}{\lim_{t\to\infty} e^{-\lambda_1 t}Z_1(t)\\&=\sum_{i \in \mathbb{N}}Y_1^{(i)}e^{-\lambda_1 T_i}}$\\
$\tau_n$ & First time the total population reaches size $n$ & $\tabeq{0.9cm}{\tau_n}{\inf\{t\geq 0:Z_0(t)+Z_1(t)=n\}}$\\
$S_j(t)/S_{\geq j}(t)$ & Number of type-$1$ clones composed of $j$/$\geq j$ cells at time $t$ & $\tabeq{0.9cm}{S_j(t)}{\sum_{T_i\leq t}\mathbf{1}_{Z_1^{(i)}(t-T_i)=j}}$\newline$\tabeq{0.9cm}{S_{\geq j}(t)}{\sum_{k\geq j}S_k(t)}$\\
$S_j^B(t)$ & Number of type-$1$ clones composed of $j$ cells at time $t$, where each clone has a random fitness advancement with the law of $B$ & \rule[0.5ex]{1cm}{0.4pt}\\
$R_f(t)$ & Number of clones making up a proportion larger than $f\in (0,1)$ of the type-$1$ population at time $t$&$\tabeq{0.9cm}{R_f(t)}{\sum_{j>fZ_1(t)}S_j(t)}$\\
\hline
\end{tabular}
\caption{Notation used throughout the paper}\label{table:notation}
\end{table}

{\bf \em Asymptotic notation.} Letting $\ell\in \mathbb{R}\cup \{\pm \infty\}$ and $g,h:\mathbb{R}\to \mathbb{R}_{\geq 0}$, we use the following asymptotic notation throughout the paper:

\begin{align*}
    g\simas{t}{\ell}h&\Longleftrightarrow \lim_{t\to\ell}\frac{g(t)}{h(t)}=1\\
    g\lsimas{t}{\ell} h&\Longleftrightarrow \limsup_{t\to\ell}\frac{g(t)}{h(t)}<\infty\\
    g\gsimas{t}{\ell} h&\Longleftrightarrow h\lsimas{t}{\ell} g\\
    g\llas{t}{\ell} h&\Longleftrightarrow \lim_{t\to \ell}\frac{g(t)}{h(t)}=0\\
    g\ggas{t}{\ell} h&\Longleftrightarrow h\llas{t}{\ell} g\\
    g\ordas{t}{\ell} h&\Longleftrightarrow g\lsimas{t}{\ell} h\text{ and }g\gsimas{t}{\ell} h
\end{align*}

Additionally, the Landau-O notation is used to specify in-line error terms. That is, letting $g:\mathbb{R}\to \mathbb{R}$ and $h:\mathbb{R}\to \mathbb{R}_{\geq 0}$ be real-functions, we say:
\begin{align*}
    g=\mathcal{O}_{t\to \ell}(h)&\Longleftrightarrow |g|\lsimas{t}{\ell}h\\
    g=o_{t\to \ell}(h)&\Longleftrightarrow |g|\llas{t}{\ell}h
\end{align*}

Usually, we examine large-time asymptotics ($\ell = \infty$), in which case we drop $t\to\infty$ from the notation as long as the limit is clear.

\section{Results}\label{sec:results}

We provide convergence results for the SFS at different frequency scales $j(t)$. In the small-frequency region, we set $j(t)=j$ and investigate $S_j(t)$. In the intermediate-frequency region, we set $1\ll j(t)\ll \exp(\lambda_1 t)$ and investigate $S_{j(t)}(t)$. Finally, in the large-frequency region, we set $j(t)\asymp \exp(\lambda_1t)$ and investigate $\sum_{k>j(t)}S_k(t)$.

\subsection{Law of large numbers in the small frequency region}\label{sec:time_freq_scaling}

For $j$ constant in time, our main theorem in this section gives an almost sure convergence result for $\exp(-\lambda_0 t)S_j(t)$ as $t\to\infty$.

\begin{restatable}{theorem}{ConvergenceResults} \label{thm:convergence_results}Let $P_j(s|b)=\mathbb{P}(Z_1^{(1)}(s)=j|Z_1^{(1)}(0)=1)$ be the transition kernel from size $1$ to $j$ in time $s$ of a single type-1 clone with fitness enhancement $b$. For non-random $b$, we let $P_j(s)=P_j(s|b)$. Conditional on $\Omega_0^\infty$ and for all $j\geq 1$:
  \begin{align}
    &e^{-\lambda_0 t}S_j(t)\overset{\substack{L^2\\a.s.}}{\to}(av+\mu)Y\int_0^\infty e^{-\lambda_0 s}P_j(s)\,ds=Y\lim_{s\to\infty}e^{-\lambda_0 s}\mathbb{E}[S_j(s)],\label{eq:driver_lln_time}\\
    &n^{-\lambda_0/\lambda_1}S_j(\tau_n)\overset{\substack{a.s.}}{\to}(av+\mu)YW^{-\lambda_0/\lambda_1}\int_0^\infty e^{-\lambda_0 s}P_j(s)\,ds=YW^{-\lambda_0/\lambda_1}\lim_{s\to \infty}e^{-\lambda_0 s}\mathbb{E}[S_j(s)].\label{eq:driver_lln_size}
  \end{align}
  Under the random fitness advance model with Assumption \ref{assumption:random}:
  \begin{align}
    &e^{-\lambda_0 t}S_j^B(t)\overset{\substack{L^2\\a.s.}}{\to}(av+\mu)Y\int_0^\infty e^{-\lambda_0 s}\mathbb{E}[P_j(s|B)]\,ds=Y\lim_{s\to\infty}e^{-\lambda_0 s} \mathbb{E}[S_j^B(s)].\label{eq:driver_lln_time_random}
  \end{align}
\end{restatable}
\begin{proof}
  See Section \ref{sec:SFS_convergence}.
\end{proof}

We extend the analysis from Gunnarsson et al.~\cite{gunnarsson2025} to obtain these almost sure results for a scaled version of the SFS with selection, both at fixed detection time and size.
The fixed-time limits (Eq.~\ref{eq:driver_lln_time} and Eq.~\ref{eq:driver_lln_time_random}) are influenced by the limiting random variable $Y = \lim_{t \to \infty} e^{-\lambda_0t} Z_0(t)$ for the wild-type process that generates the mutant clones.
The fixed-size limit (Eq.~\ref{eq:driver_lln_size}) is further influenced by the limiting random variable $W = \lim_{t \to \infty} e^{-\lambda_1t} Z_1(t)$ for the type-1 process, which controls the time it takes the population to reach size $n$. Remark that in a single-type neutral model, Gunnarsson et al.~\cite{gunnarsson2025} found that the fixed-size limit of the SFS is deterministic, which was attributed to the fact that there is no variability in the population size at time $\tau_n$.
Here, the two-type dynamics of mutation and selection lead to a random limit even in the fixed-size case.

Both limits have a deterministic factor $\lim_{s\to \infty}\exp(-\lambda_0 s)\mathbb{E}[S_j(s)]$. Indeed, this quantity exists and has an interesting tail behavior as $j\to\infty$.

\begin{restatable}{proposition}{SFSMeanAsymp}\label{prop:SFS_mean_asymp}
  The exact representations of the mean of $S_j(t)$ and $S_j^B(t)$ are (for fixed time $t$ and fixed $j$):
  \begin{align}
    \mathbb{E}[S_j(t)]&=\int_0^t (av+\mu)\mathbb{E}[Z_0(s)]P_j(t-s)\,ds, \text{ and}\label{eq:mean_SFS}\\
    \mathbb{E}[S_j^B(t)]&=\int_0^t (av+\mu)\mathbb{E}[Z_0(s)]\mathbb{E}[P_j(t-s|B)]\,ds.\label{eq:mean_SFS_random_fitness}
  \end{align}
  Large-$j$ asymptotics are as follows:
  \begin{align}
    \lim_{t\to\infty}e^{-\lambda_0 t}\mathbb{E}[S_j(t)]&=(av+\mu)\int_0^\infty e^{-\lambda_0 s}P_j(s)\,ds\nonumber\\
    &=(av+\mu)\frac{q_1}{\lambda_1}\frac{\Gamma(j)\Gamma(\lambda_0/\lambda_1+1)}{\Gamma(j+\lambda_0/\lambda_1+1)}\sum_{n=0}^\infty \frac{(\lambda_0/\lambda_1)_n(j)_n}{(j+1+\lambda_0/\lambda_1)_n n!}p_1^n\label{eq:power_series_representation_S_j}\\
    &=(av+\mu)\frac{\Gamma(\lambda_0/\lambda_1+1)}{\lambda_1 j^{\lambda_0/\lambda_1+1}}(1-p_1)^{1-\lambda_0/\lambda_1}+\mathcal{O}_{j\to\infty}\left(\frac{1}{j^{\lambda_0/\lambda_1+2}}\right).\label{eq:large_j_equivalent}
  \end{align}
  $\Gamma$ denotes the gamma function and $(x)_n:=\Gamma(x+n)/\Gamma(x)$. Under the random fitness advance model with Assumption \ref{assumption:random}:
  \begin{align}
    \lim_{t\to\infty}e^{-\lambda_0t}\mathbb{E}[S_j^B(t)]&=(av+\mu)\int_0^\infty e^{-\lambda_0 s}\mathbb{E}[P_j(s|B)]\,ds\nonumber\\
    &\overset{j\to\infty}{\sim} \frac{C}{j^{\frac{\lambda_0}{\lambda_1(b_{\max})}+1}\log j}\label{eq:large_j_equivalent_random}
  \end{align}
  where $C$ is a function of the parameters of the model which can be explicitly determined.
\end{restatable}
\begin{proof}
  See Section \ref{sec:SFS_convergence_mean}
\end{proof}

\begin{figure}[H]
	\centering
	\includegraphics[width=\textwidth]{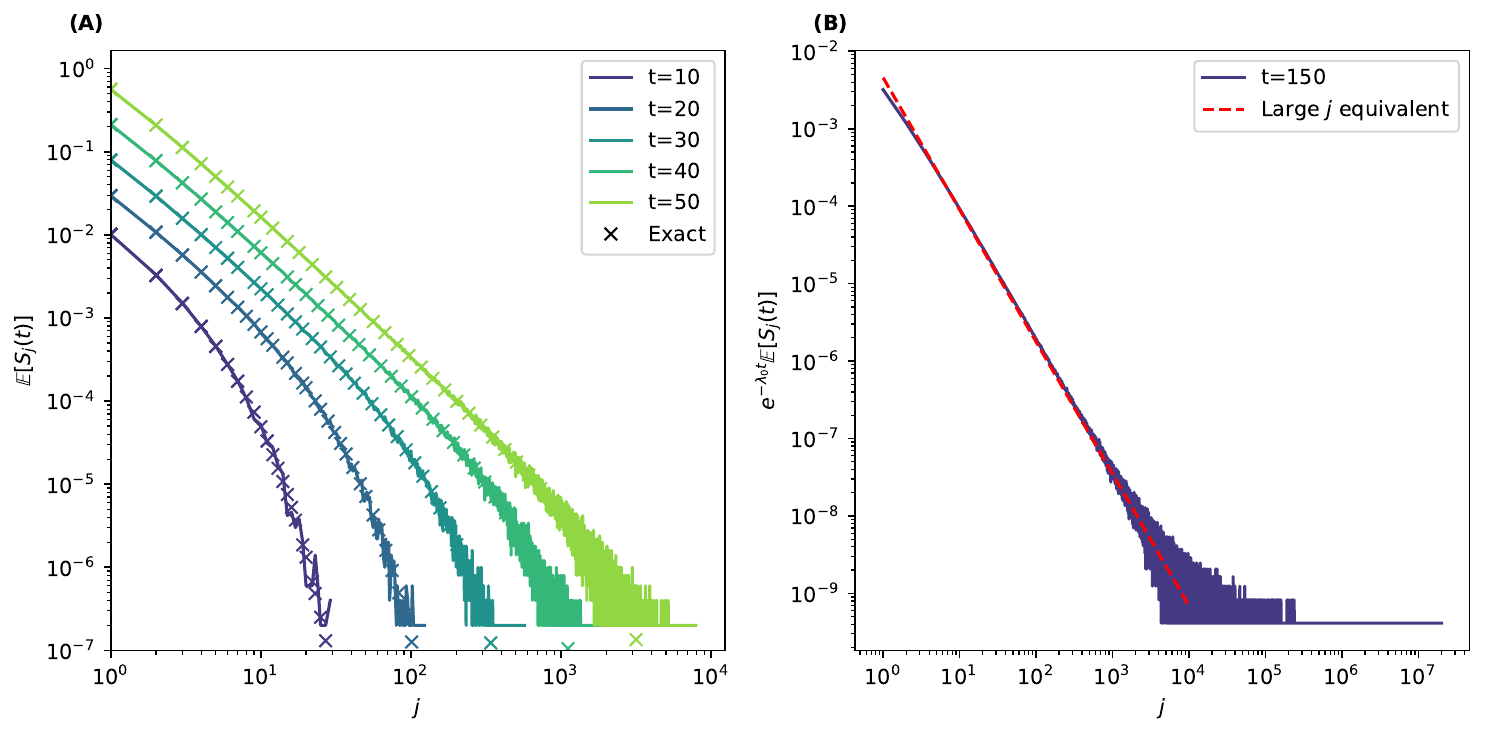}
	\caption{(left) $5\cdot 10^6$ simulations (deterministic mutant fitness model) are performed to obtain empirical means (lines), and compared with exact means (x-markers). (right) $10^3$ simulations are performed to obtain the rescaled mean driver SFS at time $t=150$. The `large $j$ equivalent' line refers to the $j^{-\lambda_0/\lambda_1-1}$ term in Eq. \ref{eq:large_j_equivalent} in Proposition \ref{prop:SFS_mean_asymp}.}\label{fig:driver_muts}
\end{figure}

The expression for the expected value of the SFS is quite intuitive: contributions to the SFS at time $t$ can be discretized into small bins $[s,s+dt]$. A mutant clone is generated at rate $(av+\mu) Z_0(s)$ at time $s$. If it forms and grows to size $j$ in the remaining $t-s$ time units (occurring with probability $P_j(t-s)$), then it contributes towards the SFS, $S_j(t)$. A similar argument works for the random fitness case. The calculation of the asymptotics follows from somewhat long but straightforward computations involving Eqs. \ref{eq:mean_SFS} and \ref{eq:mean_SFS_random_fitness}.

Eq.~\ref{eq:large_j_equivalent} indicates that mutations at large frequencies at large times follow a power law $j^{-\lambda_0/\lambda_1-1}$, as illustrated by the right panel in Figure \ref{fig:driver_muts}.
When $\lambda_0=\lambda_1$, we obtain a $j^{-2}$ power law, characteristic of a single-type model with neutral mutations (see \cite{gunnarsson2021}).
When $\lambda_0 < \lambda_1$, there are more mutations that occupy large frequencies than in the neutral model, corresponding to the heavier-tail $j^{-\lambda_0/\lambda_1-1}$ power law.

Under the random fitness advance model, mutations at large frequencies come from the clones with the largest fitness advance $b_{\max}$, which introduces a log-correction to the power law due to the rarity of clones with near-maximal fitness. We hypothesize that the finite $j$ region of the scaled SFS in this random fitness advance model provides information regarding different parts of the distribution of $B$.

{\bf \em Comparison with Cheek and Antal \cite{cheek2018}.} We provide a brief comment concerning how Theorem \ref{thm:convergence_results} compares with Proposition 4.5 of Cheek and Antal \cite{cheek2018}. If we use a superscript $(n)$ momentarily to denote the parameters of our model under their large detection size-small mutation regime, their proposition says:
\[S_j^{(n)}(\tau_n^{(n)})\Longrightarrow S_j^\ast,\]
weakly where $S_j^\ast\sim \text{Poisson}(\alpha_j)$. $\alpha_j$ is the rate of the Poisson in Proposition 4.5 of \cite{cheek2018}. Hence  $S_j^{(n)}(\tau_n^{(n)})\asymp 1$ in their limiting regime whereas $S_j(\tau_n)\asymp n^{\lambda_0/\lambda_1}$ in our limiting regime, conditional on $\Omega_0^\infty$. By their remark 4.4, $\alpha_j\ordas{j}{\infty}j^{-1-\lambda_0/\lambda_1}$. According to Eq. \ref{eq:large_j_equivalent}, this matches the behavior $\lim_{s\to\infty} e^{-\lambda_0 s}\mathbb{E}[S_j(s)]\ordas{j}{\infty} j^{-1-\lambda_0/\lambda_1}$ for the scaled SFS under our large detection-size/time regime.

{\bf\em Fraction of mutations at various frequencies.} Now, we may combine Theorem \ref{thm:convergence_results} with the representation Eq. \ref{eq:power_series_representation_S_j} in Proposition \ref{prop:SFS_mean_asymp} to prove a convergence result regarding the fraction of mutations present at a given frequency $j$. Let $S_{\geq j}(t)$ denote the number of mutations present in $j$ or more cells at time $t$ for $j\in \mathbb{N}$. $S_{\geq 1}(t)$ thus represents the total number of driver clones present at time $t$. Clearly $S_{\geq j}(t)=\sum_{k=j}^\infty S_k(t)$ at large times follows a $j^{-\lambda_0/\lambda_1}$ law in mean; in the following proposition we characterize its asymptotic behavior:
\begin{restatable}{proposition}{RelativeTypeAS}
  Conditional on $\Omega_0^\infty$, for all $k\geq 1$:
  \begin{align*}
    &e^{-\lambda_0 t}S_{\geq j}(t)\overset{\substack{L^2\\a.s.}}{\to}(av+\mu)Y\int_0^\infty e^{-\lambda_0 s}\sum_{k=j}^\infty P_k(s)\,ds\\
    &n^{-\lambda_0/\lambda_1}S_{\geq j}(\tau_n)\overset{\substack{a.s.}}{\to}(av+\mu)YW^{-\lambda_0/\lambda_1}\int_0^\infty e^{-\lambda_0 s}\sum_{k=j}^\infty P_k(s)\,ds
  \end{align*}

  Consequently, a.s.~on $\Omega_0^\infty$ the fraction of total mutations present at a frequency $j$ is asymptotically:
  \begin{align}
    \lim_{t\to\infty}\frac{S_j(t)}{S_{\geq 1}(t)}=\lim_{n\to\infty}\frac{S_j(\tau_n)}{S_{\geq 1}(\tau_n)}=\frac{\int_0^\infty e^{-\lambda_0 s}P_j(s)\,ds}{\int_0^\infty e^{-\lambda_0 s}\left(1-P_0(s)\right)\,ds}=\frac{\frac{\Gamma(j)\Gamma(\lambda_0/\lambda_1+1)}{\Gamma(j+\lambda_0/\lambda_1+1)}{}_2{F}_1(\lambda_0/\lambda_1,j;j+\lambda_0/\lambda_1+1;p_1)}{\frac{\lambda_1}{\lambda_0}{}_2F_1(1,\lambda_0/\lambda_1;\lambda_0/\lambda_1+1;p_1)},\label{eq:normalized_SFS_limit}
  \end{align}
where ${}_2F_1$ is the hypergeometric function.
\label{prop:fraction_as}
\end{restatable}
\begin{proof}
  See Section \ref{sec:SFS_convergence}.
\end{proof}
The limit is deterministic since the random factors $Y$ and $Y W^{-\lambda_0/\lambda_1}$ cancel, at fixed time and fixed detection size, respectively. The limit depends on the model parameters only through $\lambda_0/\lambda_1$ and $p_1$, which may enable estimation of the fitness effect $b$ from data (see Discussion).

\subsubsection{Proof ideas for Theorem \ref{thm:convergence_results} and related results}

We provide the idea for proving Eq. \ref{eq:driver_lln_time} and Eq. \ref{eq:driver_lln_size}, beginning with the large detection time and in the end comment on proofs for the other results in this section. The idea for this proof is to decompose:
\begin{align*}
  &S_j(t)=S_{j,+}(t)-S_{j,-}(t)\text{, where:}\\
  &S_{j,+}(t)=\sum_{k\in\mathbb{N}}S_{j,+}^k(t)\\
  &S_{j,-}(t)=\sum_{k\in \mathbb{N}}S_{j,-}^k(t).
\end{align*}
$S_{j,+}^k(S_{j,-}^k)$ are both monotone increasing processes counting the total number of times clones have entered into(out of) size $k$ by time $t$, respectively (see details in Section \ref{sec:SFS_convergence_notation}). It suffices to show $\exp(-\lambda_0 t)S_{j,+}(t)$ and $\exp(-\lambda_0 t)S_{j,-}(t)$ converge almost surely in order to prove $\exp(-\lambda_0 t)S_j(t)$ converges. The idea is to approximate $S_{j,\pm}(t)$ by $\widehat{S}_{j,\pm}(t)=Y(av+\mu)\int_0^t\sum_{k\in\mathbb{N}}P_{j,\pm}^k(t-s)\exp(\lambda_0 s)\,ds$, where $P_{j,\pm}^k(s)$ is the probability that a type-1 clone has entered into (or exited from) size j at least $k$ times by time $t$. Then, since $\exp(-\lambda_0 t)\widehat{S}_{j,\pm}(t)$ is known to converge almost surely such that their difference has the same limit as in Eq. \ref{eq:driver_lln_time}, we use the following result:
\begin{restatable}{proposition}{ConvergenceTool}
  Let $(\Omega,\mathcal{F},\mathbb{P})$ be a probability space and consider jointly-measurable, stochastic processes $X_t$ and $Y_t$. We further take $X_t$ to be non-negative and non-decreasing a.s. Suppose there exists $\lambda>0$ such that $e^{-\lambda t}Y_t \to \widetilde{Y}_\infty$ almost surely, where $\widetilde{Y}_\infty$ is a.s. finite and non-negative. If, in addition, the following condition is satisfied:
  \begin{align}
    \int_0^\infty \mathbb{E}[e^{-\lambda t}|X_t-Y_t|]dt<\infty\label{eq:convergence_tool_condition}
  \end{align}
then $e^{-\lambda t}X_t\to \widetilde{Y}_\infty$ a.s.\label{prop:convergence_tool}
\end{restatable}
We show condition \ref{eq:convergence_tool_condition} is achieved with $X_t:=S_{j,\pm}(t),Y_t:=\widehat{S}_{j,\pm}(t)$ and $\lambda:=\lambda_0$. To do so, we prove the following:
\begin{proposition}\label{prop:approximating_lln_pieces}
  We can decompose $\widehat{S}_{j,\pm}(t)=\sum_{k\in\mathbb{N}}\widehat{S}_{j,\pm}^k(t)$  where $\widehat{S}_{j,\pm}(t)=Y(av+\mu)\int_0^t P_{j,\pm}^k(t-s)e^{\lambda_0s}\,ds$ such that there exists a $0<\theta_0<1$ independent of $t,j,b$ and an $\varepsilon\in (0,\lambda_0)$ with:
  \begin{align*}
    \mathbb{E}\left[\left|S_{j,\pm}^k(t)-\widehat{S}_{j,\pm}^k(t)\right|\right]\lsimas{t}{\infty} \theta_0^k e^{\varepsilon t}.
  \end{align*}
  uniformly in $k$.
\end{proposition}
Using triangle inequality then gives condition \ref{eq:convergence_tool_condition}. The geometric decay in $k$ is not surprising because the Markov chain corresponding to the discrete-time embedding of a single type-1 clone is such that state $j\in\mathbb{N}$ is transient (see Lemma \ref{lem:uniform_bound_kth_entrance}).

Eq. \ref{eq:driver_lln_size} follows from the large time law of large numbers result and the following approximation of the detection time $\tau_n$:
\begin{lemma}\label{lem:detection_size_approximation}
  \[|\tau_n-(1/\lambda_1)\log(n/W)|\to 0\]
  almost surely on $\Omega_0^\infty$.
\end{lemma}
\begin{remark}
  \begin{enumerate}
    \item The random fitness advance convergence result (Eq. \ref{eq:driver_lln_time_random}) follows from a similar argument as in the deterministic fitness advance case. Indeed, this is because $\theta_0$ in Proposition \ref{prop:approximating_lln_pieces} is independent of $b$.
    \item Proposition \ref{prop:fraction_as} also follows a similar line of argument as Theorem \ref{thm:convergence_results} but instead redefining $S_{j,+}^k(S_{j,-}^k)$ as counting the total number of times clones have entered into(out of) region $\{j,j+1,\dots\}$ by time $t$.
    \item A corresponding fixed detection-size result for the random fitness advance model would require a non-trivial a.s.~convergence result for $\gamma(t)Z_1^B(t)$ for some $\gamma(t)$, which is outside the scope of this manuscript\footnote{Evidence from a semideterministic approximation of the model \cite{durrettfoo2010} suggests $\gamma(t)=t^{\frac{\lambda_1(b_{\max})}{\lambda_0}}e^{-\lambda_1(b_{\max})t}$, where $\lambda_1(b)=a+b-d$ denotes the type-1 growth rate with fitness advancement $b$.}.
  \end{enumerate}
\end{remark}

\subsection{A cutoff in the intermediate region of the site frequency spectrum}\label{sec:SFSVary}

When we derived the $j^{-\lambda_0/\lambda_1-1}$ power law for mutations appearing at larger, but constant in time, frequencies $j$ in Proposition \ref{prop:SFS_mean_asymp}, we first let $t\to\infty$ and then $j\to\infty$. However, different notions of frequency are possible. For example, each clone grows at rate $Z_1^{(i)}(s)\asymp e^{\lambda_1 s}$, so that to understand their contribution to the SFS, the number of clones of exponential size -- rather than of size order $1$ -- should be investigated. Thus, we now let $j$ depend on time and study the SFS over the joint limit of $(t,j(t))$ as $t\to\infty$. In fact, in Section \ref{sec:large_frequencies}, an important part of our analysis involves studying the limit of a cumulative version of the SFS given by clones that grow at (or faster than) $j(t)\sim xe^{\lambda_1 t}$. When clones are on the order $1\ll j(t)\ll \exp(\lambda_1 t)$, they are of ``intermediate'' size and thus contribute to the intermediate region of the SFS. Intriguingly, we observe segregated behavior of the SFS depending on whether $j(t)$ grows faster or slower than $\exp\left(\lambda_0\lambda_1/(\lambda_0+\lambda_1)t\right)$:

\begin{restatable}{theorem}{SFSSparsityAS}\label{thm:SFSSparsityAS}
  Define $j_k^-(t)=\min\{j:S_j(t)\leq k\}$ and $j_k^+(t)=\max\{j:S_j(t)\geq k\}$. Then:
    \begin{align}
        &\lim_{k\to\infty}\liminf_{t\to\infty}\frac{\log j_k^-(t)}{t}=\lim_{k\to\infty}\limsup_{t\to\infty}\frac{\log j_k^-(t)}{t}\nonumber\\
        =&\lim_{k\to\infty}\liminf_{t\to\infty}\frac{\log j_k^+(t)}{t}=\lim_{k\to\infty}\limsup_{t\to\infty}\frac{\log j_k^+(t)}{t}=\frac{\lambda_0\lambda_1}{\lambda_0+\lambda_1}\label{eq:cutoff_time}
    \end{align}
    and 
    \begin{align}
        &\lim_{k\to\infty}\liminf_{n\to\infty}\frac{\log j_k^-(\tau_n)}{\log n}=\lim_{k\to\infty}\limsup_{n\to\infty}\frac{\log j_k^-(\tau_n)}{\log n}\nonumber\\
        =&\lim_{k\to\infty}\liminf_{n\to\infty}\frac{\log j_k^+(\tau_n)}{\log n}=\lim_{k\to\infty}\limsup_{n\to\infty}\frac{\log j_k^+(\tau_n)}{\log n}=\frac{\lambda_0}{\lambda_0+\lambda_1}\label{eq:cutoff_size}
    \end{align}
    almost surely on $\Omega_0^\infty$.
\end{restatable}
\begin{proof}
  See Section \ref{sec:intermediate_SFS}.
\end{proof}
The quantity $j_k^-(t)$ represents the {\em first} frequency that {\em few} ($\leq k$) clones occupy, and $j_k^+(t)$ represents the {\em last} frequency that {\em many} ($\geq k$) clones occupy. Theorem \ref{thm:SFSSparsityAS} can be reformulated to say that on $\Omega_0^\infty$, there is a sequence of numbers $\varepsilon_k\to 0^+$ and a sequence of times $(t_k)_k$ with $t_k \to \infty$ so that $t\geq t_k$ implies:
\[j_k^-(t),j_k^+(t)\in \left[\exp\left(\left(\frac{\lambda_0\lambda_1}{\lambda_0+\lambda_1}-\varepsilon_k\right)t\right),\exp\left(\left(\frac{\lambda_0\lambda_1}{\lambda_0+\lambda_1}+\varepsilon_k\right)t\right)\right].\]
Intuitively, approximately order-1 clones occupy sizes $j(t)\asymp \exp(t\lambda_0\lambda_1/(\lambda_0+\lambda_1))$. Additionally if $j(t)$ grows slower or faster than this rate, many or asymptotically few clones asymptotically occupy frequency $j(t)$. This can be made rigorous in both a quantitative and qualitative sense. We explore the quantitative sense first in Proposition \ref{prop:SFSSparsity_quant} by providing a scaling so that $S_{j(t)}(t)$ converges (in $L^1$ or in distribution).

\begin{restatable}{proposition}{SFSSparsityQuant}\label{prop:SFSSparsity_quant}
  Suppose that $j:\mathbb{R}_{\geq 0}\to \mathbb{N}$ is non-decreasing. Then, we have the following trichotomy:
    \begin{enumerate}
        \item If $j(t)\gg\exp(t\lambda_0\lambda_1/(\lambda_0+\lambda_1))$, $S_{j(t)}(t)\to 0$ in $L^1$;
        \item If $j(t)\sim C\exp(t\lambda_0\lambda_1/(\lambda_0+\lambda_1))$, then conditional on $\Omega_0^\infty$, $S_{j(t)}(t)$ converges in distribution to a random variable with characteristic function:
        \begin{align}
            \mathbb{E}[e^{YL(e^{i\theta}-1)}|\Omega_0^\infty]=\frac{q_0}{q_0-L\left(e^{i\theta}-1\right)}\label{eq:char_function_poisson}
        \end{align}
        where $L=\lim_{t\to\infty} \mathbb{E}[S_{j(t)}(t)]$ (and $L$ can be explicitly determined). \ref{eq:char_function_poisson} is the characteristic function of a mixture of a Poisson random variable with rate $YL|\Omega_0^\infty$;
        \item If $j(t)\ll\exp(t\lambda_0\lambda_1/(\lambda_0+\lambda_1))$, then:
        \[e^{-\lambda_0 t}j(t)^{(\lambda_1+\lambda_0)/\lambda_1}S_{j(t)}(t)\overset{L^1}{\to} YL_{j(\cdot)}\]
        where $L_{j(\cdot)}=\lim_{t\to\infty}e^{-\lambda_0 t}j(t)^{(\lambda_1+\lambda_0)/\lambda_1}\mathbb{E}\left[S_{j(t)}(t)\right]$ (and $L_{j(\cdot)}$ can be explicitly determined).
    \end{enumerate}
\end{restatable}
\begin{proof}
  See Section \ref{sec:quant_conv}.
\end{proof}

\begin{figure}[H]
	\centering
	\includegraphics[width=\textwidth]{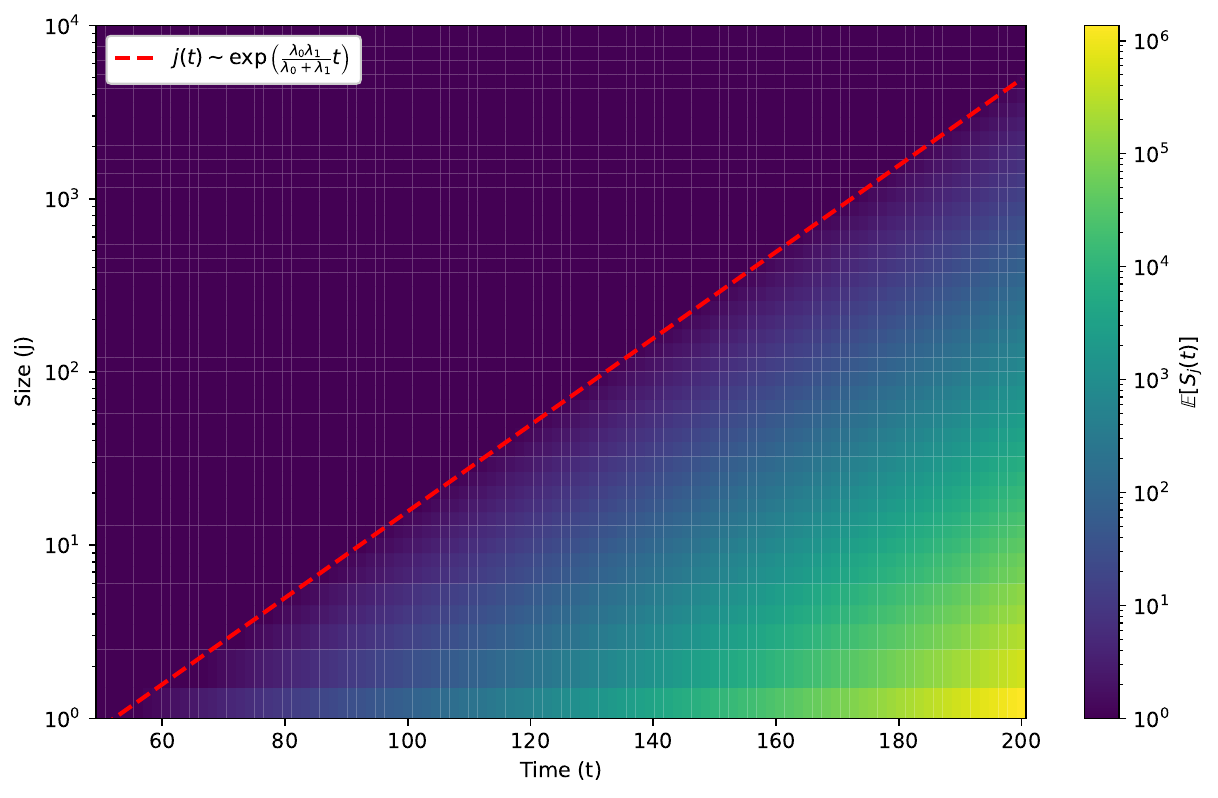}
	\caption{Mean site frequency spectrum heatmap at varying sizes $j$ and times $t$}\label{fig:SFS_heatmap}
\end{figure}

Points 1--3 support the notion that $\exp\left(\lambda_0\lambda_1/(\lambda_0+\lambda_1)t\right)$ sets a threshold for the ``sparsity'' of mutations at a particular frequency $j(t)$. This is also seen via simulations (Figure \ref{fig:SFS_heatmap}). Intuitively, this threshold represents a balance between two competing forces that affect the size of $S_{j(t)}(t)$: production of driver clones and their growth to size $j(t)$ contributes towards $S_{j(t)}(t)$, while a change in the size of a clone with $j(t)$ members detracts from $S_{j(t)}(t)$. When $j(t)\sim C\exp(t\lambda_0\lambda_1/(\lambda_0+\lambda_1))$, these forces counter-balance one another and yield a Poisson limit theorem. We can get a qualitative version of Proposition \ref{prop:SFSSparsity_quant}, which acts as a corollary to Theorem \ref{thm:SFSSparsityAS}.

\begin{restatable}{corollary}{SFSSparsityQual}\label{cor:SFSSparsity_qual}
  For every $\varepsilon>0$ if $j_{\varepsilon}(t)\overset{t\to\infty}{\lesssim} \exp\left(\left(\frac{\lambda_0\lambda_1}{\lambda_0+\lambda_1}-\varepsilon\right)t\right)$ and $J_{\varepsilon}(t)\overset{t\to\infty}{\gtrsim} \exp\left(\left(\frac{\lambda_0\lambda_1}{\lambda_0+\lambda_1}+\varepsilon\right)t\right)$, then:
    \begin{align*}
        \lim_{t\to\infty}S_{j_\varepsilon(t)}(t)&=\infty & \lim_{n\to\infty} S_{j_\varepsilon(\tau_n)}(\tau_n)&=\infty\\
        \limsup_{t\to\infty} S_{J_\varepsilon(t)}(t)&\overset{\varepsilon\to 0^+}{\lesssim}\varepsilon^{-1} & \limsup_{n\to\infty}S_{J_\varepsilon(\tau_n)}(\tau_n)&\overset{\varepsilon\to 0^+}{\lesssim}\varepsilon^{-1}\\
        \liminf_{t\to\infty}S_{J_\varepsilon(t)}(t)&=0.
    \end{align*}
    almost surely on $\Omega_0^\infty$. The asymptotics terms are uniform over the sample points in $\Omega_0^\infty$ (up to null sets).
\end{restatable}
\begin{proof}
  See Section \ref{sec:intermediate_SFS}.
\end{proof}

We remark that on $\Omega_0^\infty$, for $\varepsilon>0$ small enough such that $\lambda_0\lambda_1/(\lambda_0+\lambda_1)+\varepsilon<\lambda_1$, $\limsup_{t\to\infty}S_{J_\varepsilon(t)}(t)\geq 1$ because the production times of clones $T_i$ tend to $\infty$ as $i\to\infty$ and established clones grow at order $\exp(\lambda_1 t)$. So, $S_{J(t)}(t)$ has no almost sure limit. For large $t$, the oscillations of $S_{J_\varepsilon(t)}(t)$ vary within $[0,C\varepsilon^{-1}]$.

\subsubsection{Proof ideas for Theorem \ref{thm:SFSSparsityAS} and related results}\label{proof:intermediate}

An important step in proving Theorem \ref{thm:SFSSparsityAS} is to note that $j_k^-(t)\leq j_k^+(t)+1$ for all $t\geq 0,k\geq 1$. We then obtain the following reductions:

\begin{enumerate}
  \item The large detection size results (Eq. \ref{eq:cutoff_size}) follow from the large time results (Eq. \ref{eq:cutoff_time}) by treating $(\tau_n)_n$ as a discrete subsequence of time and using the fact $\tau_n\simas{n}{\infty}(1/\lambda_1)\log n$ on $\Omega_0^\infty$ (by Lemma \ref{lem:detection_size_approximation})
  \item To prove Eq. \ref{eq:cutoff_time}, it suffices to show that for every $\varepsilon>0$, there is $k$ sufficiently large such that:
  \begin{align}
    \frac{\lambda_0\lambda_1}{\lambda_0+\lambda_1}-\varepsilon\leq \liminf_{t\to\infty}\frac{\log j_k^-(t)}{t},\, \limsup_{t\to\infty}\frac{\log j_k^+(t)}{t}\leq\frac{\lambda_0\lambda_1}{\lambda_0+\lambda_1}+\varepsilon.\label{eq:reduction}
  \end{align}
  a.s. on $\Omega_0^\infty$.
\end{enumerate}

It remains to show \ref{eq:reduction}. For both bounds, we consider functions $j_\varepsilon:\mathbb{R}\to \mathbb{N},J_\varepsilon:\mathbb{R}\to \mathbb{N}$ which are constant over each of the half open intervals $[n,n+1)$ for $n\in \mathbb{N}_{\geq 0}$ and grow at rates:
\begin{align*}
  &j_\varepsilon(t)\ordas{t}{\infty}\exp\left(\left(\lambda_0\lambda_1/(\lambda_0+\lambda_1)-\varepsilon\right)t\right),\\
  &J_\varepsilon(t)\ordas{t}{\infty}\exp\left(\left(\lambda_0\lambda_1/(\lambda_0+\lambda_1)+\varepsilon\right)t\right).
\end{align*}

Defining the following events:
\begin{align*}
  &E^-_n:=\{\exists j\leq j_{\varepsilon}(n),\inf_{t\in [n,n+1)}S_j(t)\leq k\},\\
  &E^+_n=\{\exists j\geq  J_{\varepsilon}(n),\,\sup_{t\in [n,n+1)}S_j(t)\geq k\},
\end{align*}
we must show that for $k$ sufficiently large:
\begin{align}
  &((E_n^-)^c)_n\text{ eventually occurs on }\Omega_0^\infty,\text{ and}\label{eq:eventually_neg}\\
  &((E_n^+)^c)_n\text{ eventually occurs on }\Omega_0^\infty\label{eq:eventually_pos}
\end{align}
\ref{eq:eventually_neg} would show the lower bound of \ref{eq:reduction} and \ref{eq:eventually_pos} would show the upper bound of \ref{eq:reduction}. Thus, we must show that the events $E_n^-$ and $E_n^+$ are asymptotically unlikely as $n\to\infty$ (meaning, the probabilities of these events decay to $0$ sufficiently quickly).

Begin with \ref{eq:eventually_neg}. The idea is to partition $[n,n+1)$ into further half-open intervals $\{[t_i^n,t_{i+1}^n)\}_{i=1}^{M_n}$ such that within each subinterval it is unlikely that even one event were to occur within the process $(Z_0,Z_1)$. Letting the associated events for every $1\leq j\leq j_\varepsilon(n)$ and $[t_i^n,t_{i+1}^n)$ be $E_{n,i,j}^-$ ($E_{n,i,j}^-=\{\inf_{t\in [t_i^n,t_{i+1}^n)}S_j(t)\leq k\}$), we clearly have $E_n^-\subseteq \bigcup_{1\leq j\leq j_\varepsilon(n)}\bigcup_{1\leq i\leq M_n}E_{n,i}^{-}$.

Furthermore, we show that the event $\{S_j(t_i^n)\leq 2k\}$ is asymptotically very unlikely on the event that $B_x:=\{\inf_t \exp(-\lambda_0 t)Z_0(t)\geq x\}$ (which itself is a likely event for $x$ small, see Lemma \ref{lem:large_deviation_bound}).
\begin{lemma}\label{lem:upper_bound_Sj}
  For $j\leq j_\varepsilon(n)$:
  \[\mathbb{P}(S_j(t_i^n)\leq 2k,B_x)\lsimas{n}{\infty}\exp\left(-\kappa_1\exp\left(\varepsilon n\right)\right)\]
\end{lemma}

Thus, it remains to examine:
\[\{S_j(t_i^n)\geq 2k\}\cap E_{n,i}^-=\{S_j(t_i^n)\geq 2k,\,\inf_{t\in [t_i^n,t_{i+1}^n)}S_j(t)\leq k\}.\]
But, this event requires that at least $2k-k=k$ type-1 events occur to the type-1 cells initially present at time $t_i^n$ by time $t_i^{n+1}$. This scenario can be stochastically dominated by the tail of some binomial with small mean by definition of $[t_i^n,t_{i+1}^n)$.

Now we move on to showing \ref{eq:eventually_pos}. We once again break up $E_n^+$ by both $j$ and how large $S_j$ is at the beginning of the interval $[n,n+1)$. We need an opposite bound to Lemma \ref{lem:upper_bound_Sj}

\begin{lemma}\label{lem:lower_bound_Sj}
  \[\mathbb{P}(S_j(s)\geq k+1)\leq C_2 e^{(k+1)\lambda_0 s}j^{-(1+\lambda_0/\lambda_1)(k+1)}.\]
\end{lemma}
Note, in the case $k=0$, this bound follows by using Markov's inequality and the mean of $S_j(s)$ provided in Proposition \ref{prop:SFS_mean_asymp}. As $j\geq J_\varepsilon(n)$, for sufficiently large $k$ the event $\{S_j(n)\geq k+1\}$ is asymptotically unlikely. It remains to examine:
\[\{S_j(n)\leq k,\,\sup_{t\in [n,n+1)}S_j(t)\geq k+1\}.\]
On this event, we can define the stopping time $\zeta_n:=\inf\{t\geq n:S_j(t)\geq k+1\}$, which has the property $S_j(\zeta_n)=k+1$. We examine the local time $L=m\{t:n\leq t\leq n+2,S_j(t)\geq k+1\}$ specifying the measure of time $t$ such that $S_j(t)\geq k+1$ over the interval $[n,n+2)$. If $t_{\zeta_n}$ denotes the first time after $\zeta_n$ in which the lifetime of one of the $k+1$ size-$j$ clones elapses, $L$ satisfies:
\[L\geq \mathbf{1}_{S_j(n)\leq k,\,\sup_{t\in [n,n+1)}S_j(t)\geq k+1}\min(t_{\zeta_n}-\zeta_n,1).\] 
So, taking expectation and using the Markov property yields an upper bound on the probability of $\{S_j(n)\leq k,\,\sup_{t\in [n,n+1)}S_j(t)\geq k+1\}$, which proves the result upon controlling $\mathbb{E}[L]=\int_n^{n+2} \mathbb{P}(S_j(s)\geq k+1)\,ds$ using Lemma \ref{lem:lower_bound_Sj}. In fact, we can explicitly determine how large $k$ should be to obtain sufficient control:
\begin{restatable}{proposition}{IntermediateJKPlus}\label{prop:intermediate_j_k_plus}
    There exists a constant $C_1>0$ such that if $k>\frac{C_1}{\varepsilon}$, then:
    \begin{align*}
        \limsup_{t\to\infty}\frac{\log(j_k^+(t))}{t}\leq \frac{\lambda_0\lambda_1}{\lambda_0+\lambda_1}+\varepsilon
    \end{align*}
    almost surely.
\end{restatable}

\begin{remark}
  \begin{enumerate}
    \item Proposition \ref{prop:SFSSparsity_quant} is comparatively more simple to show. Point 1 follows since $S_j(t)\geq 0$ and by showing that the exact expression for the mean of the SFS (Eq.~\ref{eq:mean_SFS}) tends to $0$ for the specified $j(t)$. Point 2 is proven by examining the characteristic function of $S_{j(t)}(t)$ and using a similar (and in fact simpler) approximation strategy as used to show Theorem \ref{thm:convergence_results}. Point 3 also uses this similar approximation strategy.
  \item Due to Theorem \ref{thm:SFSSparsityAS} and Proposition \ref{prop:SFSSparsity_quant}, we get Corollary \ref{cor:SFSSparsity_qual}. Firstly, the detection size limits follow from the large time limits. The time limit at sizes $j_\varepsilon(t)$ follows since for all $k$ large and $t$ large $j_\varepsilon(t)<j_k^-(t)$. The $\limsup$ limit at sizes $J_\varepsilon(t)$ follow from Proposition \ref{prop:intermediate_j_k_plus}, and the $\liminf$ limit trivially follows since $S_{J_\varepsilon(t)}(t)$ converges to $0$ in $L^1$ (Proposition \ref{prop:SFSSparsity_quant}) and by taking a subsequence of times along which this convergence is almost surely.
  \end{enumerate}
\end{remark}

\subsection{Frequencies of large clones}\label{sec:large_frequencies}
To gain further insights into the behavior of the SFS at the largest frequencies, we now investigate the number of ``large families'' of driver clones using the {\em relative} driver SFS, which gives the number of distinct clones above a certain relative frequency $f\in (0,1)$, $R_f(t)=\sum_{1\leq i\leq M(t)}\mathbf{1}_{Z_1^{(i)}(t-T_i)>Z_1(t)f}=\sum_{j:j>Z_1(t)f}S_j(t)$. This version of the SFS has previously been studied by Tung and Durrett \cite{tung}. We get a convergence result in the directed limit as $t\to\infty$ and then $f\to 0^+$.
\begin{restatable}{theorem}{ASRelativeSFS}\label{thm:relative_SFS}
  The following limits hold a.s. on $\Omega_0^\infty$:
  \begin{align}
    &\lim_{n\to\infty}R_f(\tau_n)=\lim_{t\to\infty} R_f(t)=\sum_{i\in \mathbb{N}}\mathbf{1}_{e^{-\lambda_1 T_i}Y_1^{(i)}>Wf}\label{eq:frequency_limit}\\
    &\lim_{f\to 0^+}\lim_{n\to\infty}f^{\lambda_0/\lambda_1}R_f(\tau_n)=\lim_{f\to 0^+}\lim_{t\to\infty}f^{\lambda_0/\lambda_1}R_f(t)=\frac{(av+\mu) q_1^{1-\lambda_0/\lambda_1}}{\lambda_1}\Gamma(\frac{\lambda_0}{\lambda_1})YW^{-\lambda_0/\lambda_1}.\label{eq:frequency_small_limit}
  \end{align}
  On the other hand, the end of the support of $R_f(t)$ is:
  \begin{align}
    \sup\supp R_f(t)=\max_{i:T_i\leq t}(Z_1^{(i)}(t-T_i))/Z_1(t)\overset{t\to\infty}{\to}\sup_i\left(e^{-\lambda_1 T_i}Y_1^{(i)}\right)/W,\label{eq:support_frequency_limit}
  \end{align}
  a.s. on $\Omega_0^\infty$.
\end{restatable}
\begin{proof}
  See Section \ref{sec:proof_frequencies_large_clones}.
\end{proof}
Note that $e^{-\lambda_1 T_i}Y_1^{(i)}/W$ represents the asymptotic proportion of the population that the $i$th type-1 clone makes up. Additionally, it should be remarked that $R_f(t)\leq \lceil 1/f\rceil$, so that $R_f(t)$ converges boundedly and thus in $L^p$ for $p>0$ as $t\to\infty$.

Conditional on $\Omega_0^\infty$, we can rewrite Eq. \ref{eq:frequency_small_limit} as:
\begin{align}
  \lim_{t\to\infty}R_f(t)\overset{f\to 0^+}{\sim} f^{-\lambda_0/\lambda_1}\frac{(av+\mu) q_1^{1-\lambda_0/\lambda_1}}{\lambda_1}\Gamma(\frac{\lambda_0}{\lambda_1})YW^{-\lambda_0/\lambda_1}\label{eq:relative_large_scaling}
\end{align}
Note that if we substitute $f^{-\lambda_0/\lambda_1} \Gamma(\lambda_0/\lambda_1) $ with $\frac{d}{df} \left( f^{-\lambda_0/\lambda_1} \Gamma(\lambda_0/\lambda_1)\right) = -  f^{-\lambda_0/\lambda_1-1} \Gamma(\lambda_0/\lambda_1+1)$ and ignore the stochastic terms, we recover the $j^{-\lambda_0/\lambda_1-1}$ power law in Eq.~\ref{eq:large_j_equivalent} of Proposition \ref{prop:SFS_mean_asymp} (irrespective of sign). We can thus say that Eq.~\ref{eq:relative_large_scaling} is a cumulative version of the power law result we derived before.

Note, an easy trichotomy follows from \ref{eq:frequency_limit}.
\begin{restatable}{corollary}{ScaleSumSFS}\label{cor:scale_sum_SFS}
    Denote $k(n):\mathbb{N}\to \mathbb{N}$ such that $k(n)\ll n$ as $n\to\infty$, $G^+(t):\mathbb{R}_{\geq 0}\to \mathbb{R}_{\geq 0}$ such that $G^+(t)\gg e^{\lambda_1 t}$ as $t\to\infty$, and $G^-(t):\mathbb{R}_{\geq 0}\to \mathbb{R}_{\geq 0}$ such that $G^-(t)\ll e^{\lambda_1 t}$ as $t\to\infty$. Then conditional on $\Omega_0^\infty$, for every $x\in (0,1)$:

  \begin{align*}
    \sum_{j=n-k(n)}^{n} S_j(\tau_n)&\overset{a.s.}{\to} 0 &\sum_{j> G^+(t)}S_j(t)&\overset{a.s.}{\to}0\\
    \sum_{j=\lceil fn\rceil }^{n} S_j(\tau_n)&\overset{a.s.}{\to} \sum_{i\in \mathbb{N}}\mathbf{1}_{e^{-\lambda_1 T_i}Y_1^{(i)}>Wf}& \sum_{j> xe^{\lambda_1 t}}S_j(t)&\overset{a.s.}{\to}\sum_{i\in \mathbb{N}}\mathbf{1}_{Y_1^{(i)}> xe^{\lambda_1 T_i}}\\
    \sum_{j=k(n)}^{n} S_j(\tau_n)&\overset{a.s.}{\to} \infty & \sum_{j> G^-(t)}S_j(t)&\overset{a.s.}{\to}\infty
  \end{align*}
\end{restatable}
\begin{proof}
  See Section \ref{sec:proof_thm_corollary_relative_SFS}.
\end{proof}

If we look at the large detection-size limit column first, we see that there is an asymptotically vanishing contribution to the cumulative SFS from clones with size near $n$, an order 1 contribution from clones of size $fn$ for $f\in (0,1)$, and an infinite contribution from clones of small size. The middle expression gives a connection between the fixed-size SFS viewed at absolute frequencies $j = \lceil fn \rceil$ and above and the relative SFS observed at a fixed time, showing that they converge to the same limit.
We see similar results for the large-time limit column. 
Finally, compare the above corollary with Proposition \ref{prop:SFSSparsity_quant}. The scale that sets the sparsity of mutations \textit{at a particular frequency} $j(t)$ is $j(t)\ordas{t}{\infty} \exp\left(\left(\lambda_0\lambda_1/(\lambda_0+\lambda_1)\right)t\right)$ and the scale that sets the sparsity of mutations of \textit{at least} a particular frequency $j(t)$ is $j(t)\ordas{t}{\infty} \exp\left(\lambda_1 t\right)$.

{\bf Comparison with the semi-deterministic model}
Durrett has investigated the relative SFS and other summaries of genealogical information in a semi-deterministic approximation to our branching process model. The two-type semi-deterministic model can be described as follows. The type-$0$ population grows as $Z_0^\ast(t)=Ye^{\lambda_{0,0} t}$, and the mutation times $\{T_i\}$ are generated as an inhomogeneous Poisson point process with intensity $\Lambda(dt)=vZ_0^\ast(t)\,dt$ conditional on $Y$\footnote{Recall that $\lambda_{0,v}=a(1-v)-d$ denotes the growth rate with no replication-independent mutations and with the replication-dependent mutation probability explicitly shown. Note that $\lambda_{0,0}$ is used as their semi-deterministic model does not incorporate mutations into the type-0 growth rate}. At each mutation time, an independent type-1 clone is started.

Tung and Durrett analyze the distribution of $\lim_{t\to\infty}R_f(t
)$ (Eq. \ref{eq:frequency_limit}) using results from Pitman and Yor concerning the distribution of normalized points in a Poisson point process over the positive real line \cite{tung,pitman1997}. In particular, they also find that the tail follows a $f^{-\lambda_{0,0}/\lambda_1}$ shape in expectation:
\begin{align}
  \mathbb{E}[\lim_{t\to\infty}R_f(t)]\overset{S.D.}{=}\frac{\sin\left(\pi\frac{\lambda_{0,0}}{\lambda_1}\right)}{\pi\frac{\lambda_{0,0}}{\lambda_1}}\left(\frac{1}{f}-1\right)^{\lambda_{0,0}/\lambda_1},\label{eq:tung_durrett}
\end{align}
where $\overset{S.D.}{=}$ denotes equality in the semideterministic model. In Eq. 9a of \cite{durrettgenetics2010}, it is further found that:
\begin{align}
  \sup\text{supp}(R_f(t))\overset{S.D}{\Longrightarrow} V,
\end{align}
weakly as $t\to\infty$ where $V$ is the largest point in a Poisson-Dirichlet distribution with parameters $(\alpha,0)$, with $\alpha:=\lambda_{0,0}/\lambda_1$. Interestingly, we recover these results in \textit{our} model in the small mutation limit.

\begin{figure}[H]
	\centering
	\includegraphics[width=\textwidth]{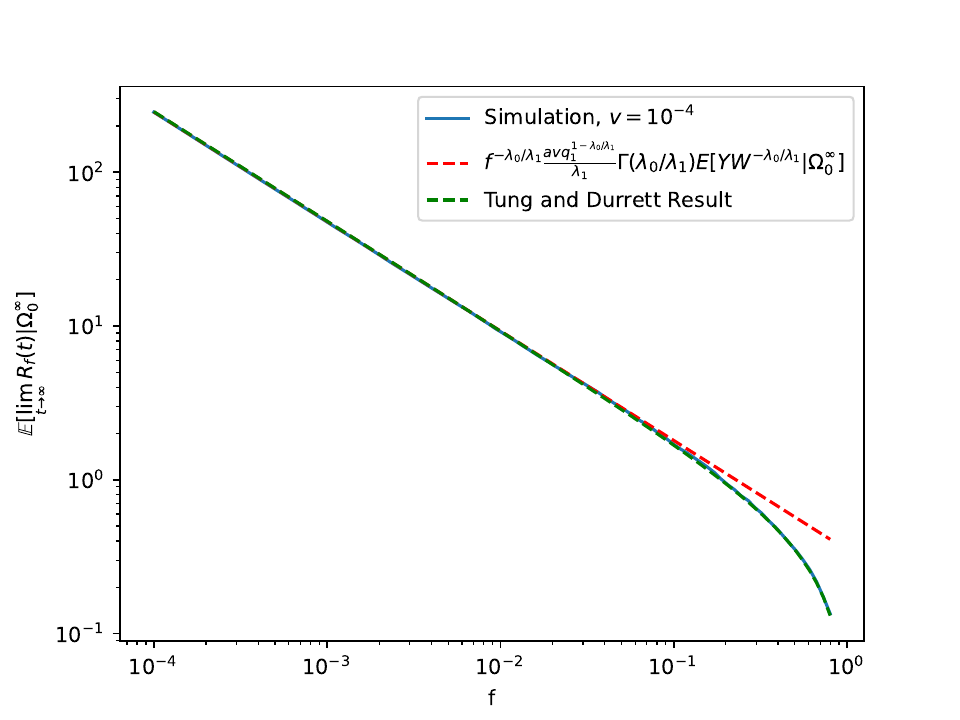}
	\caption{Large time relative plots of $R_f(t)$. }\label{fig:relative_sfs}
\end{figure}
\begin{proposition}\label{prop:relative_SFS_small_mut}
  Suppose mutations are only attached to births ($\mu=0$), then:
  \begin{align}
    &\lim_{f\to0^+}\lim_{t\to\infty}f^{\lambda_0/\lambda_1}R_{f}(t)|\Omega_0^\infty\overset{v\to 0^+}{\Longrightarrow}R^\ast\label{eq:small_freq_small_mutation}\\
    &\lim_{t\to\infty}\sup\supp(R_f(t))|\Omega_0^\infty\overset{v\to0^+}{\Longrightarrow}V.\label{eq:large_freq_small_mutation}
  \end{align}
  $R^\ast$ has Laplace transform:
  \[\mathcal{L}(x)=E_\alpha(-Cx),\]
  where $E_\alpha$ denotes the Mittag-Leffler function and $C=\Gamma(\alpha)\sin(\pi\alpha)/\pi$.
\end{proposition}
\begin{proof}
  See Section \ref{sec:small_mutation}.
\end{proof}
In particular note that:
\[\mathbb{E}[R^\ast]=\frac{\sin(\pi\alpha)}{\pi\alpha}.\]
The intuition behind this correspondence lies in a deeper connection between our two-type branching process model in the small mutation limit and the semi-deterministic model Tung and Durrett study (for example, for the case $\mu=0$, see Cheek and Antal \cite{cheek2018} Proposition 5.1; also see the Discussion for more commentary on this).

By Figure \ref{fig:relative_sfs}, we observe quite good agreement between simulations and Tung and Durrett's result (Eq. \ref{eq:tung_durrett}) when mutation rates are small. Our leading order asymptotic (Eq. \ref{eq:relative_large_scaling}) is good for small $f$ but is an overestimate of both the simulation and Tung and Durrett result for large $f$.

The CDF of $V$ is obtained in Proposition 20 of Pitman and Yor \cite{pitman1997}. We note that connections between the Poisson-Dirichlet distribution and the point process of non-extinct scaled clone sizes $\{e^{-\lambda_1 T_i}Y_1^{(i)}\}_i/\{0\}$ seemed to have first appeared in Tavar\'e \cite{tavare1987}.

\begin{remark}
    In the propositions where we restrict mutation accumulation to occur during births and look in the small mutation limit (Proposition \ref{prop:relative_SFS_small_mut}), it is likely possible to include replication-independent mutations and specify the small mutation limit to be $v,\mu\to 0^+$ such that $v/\mu\to \theta$ where $\theta\in (0,\infty)$.
\end{remark}

\subsubsection{Proof idea for Theorem \ref{thm:relative_SFS} and related results}

Firstly, Eq. \ref{eq:frequency_limit} is straightforward to prove since on $\Omega_0^\infty$ $\mathbf{1}_{Z_1^{(i)}(t-T_i)>Z_1(t)f}\to \mathbf{1}_{e^{-\lambda_1 T_i}Y_1^{(i)}>Wf}$ so that it only remains to justify an interchange of a limit and summation. This is proved by showing that clones initiated later form a negligible proportion of the total type-1 population uniformly in time.

\begin{restatable}{lemma}{UniformCloneControl}\label{lem:Y_small}
     Let $Y^{(i)}(t)=e^{-\lambda_1(t-T_i)}Z_1^{(i)}(t-T_i)$ and $Y_*^{(i)}=\sup_{T_i\leq t<\infty} Y^{(i)}(t)$. Then, $\lim_{i\to\infty} e^{-\lambda_1 T_i}Y_*^{(i)}=0$ a.s. If $i\geq \lim_{t\to\infty}M_t$, set for notation $Y^{(i)}(t)\equiv 0$.
\end{restatable}
To investigate $\lim_{t\to\infty}R_f(t)$ in the $f\to 0^+$ limit, we first examine a simpler quantity counting the number of ``large families'' of mutations.  Specifically, define $\widetilde{R}_x(t)=\sum_{j> x\exp(\lambda_1 t)}S_j(t)$ for $x>0$ as a cumulative SFS counting clones of size $> x\exp(\lambda_1 t)$ at time $t$. This type of quantity has previously been defined in a different setting by Bonnet and Leman \cite{bonnet2024}.

From Corollary \ref{cor:scale_sum_SFS}, $\widetilde{R}_x(\infty):=\lim_{t\to\infty}\widetilde{R}_x(t)=\sum_{T_i}\mathbf{1}_{e^{-\lambda_1 T_i}Y_1^{(i)}>x}$. Hence, $\widetilde{R}_{fW}(\infty)=R_f(\infty)$. Once it is shown that on $\Omega_0^\infty$:
\begin{align}
  \lim_{x\to 0^+}x^{\lambda_0/\lambda_1}\widetilde{R}_x(\infty)=\frac{(av+\mu)q_1^{1-\lambda_0/\lambda_1}}{\lambda_1}\Gamma(\frac{\lambda_0}{\lambda_1})Y,\label{eq:x_limit_semi}
\end{align}
then taking $x_f=Wf$ and sending $f\to 0^+$ would prove Eq. \ref{eq:frequency_small_limit}. By performing the change of variables $x(r)=\exp(-r)$ and sending $r\to\infty$, $\widetilde{R}_{x(r)}(\infty)$ is increasing in $r$ and, purportedly, growing at order $\exp(r\lambda_0/\lambda_1)$. Therefore, we may apply a similar proof strategy as for Theorem \ref{thm:convergence_results}. Namely, we find a process $\overline{R}_{x(r)}$ where $\exp(-r\lambda_0/\lambda_1)\overline{R}_{x(r)}$ converges a.s. on $\Omega_0^\infty$ to the same limit as in \ref{eq:x_limit_semi} and show condition Eq. \ref{eq:convergence_tool_condition} of Proposition \ref{prop:convergence_tool} holds to conclude Eq. \ref{eq:x_limit_semi}. We can then conclude:
\begin{restatable}{theorem}{ASRelativeSFS}
   On $\Omega_0^\infty$, we obtain the following small frequency scaling a.s.:
  \[\lim_{x\to 0^+}\lim_{t\to\infty}x^{\lambda_0/\lambda_1}\widetilde{R}_x(t)=\lim_{f\to 0^+}\lim_{t\to\infty}f^{\lambda_0/\lambda_1}W^{\lambda_0/\lambda_1}R_f(t)=\frac{(av+\mu) q_1^{1-\lambda_0/\lambda_1}}{\lambda_1}\Gamma(\frac{\lambda_0}{\lambda_1})Y.\]
  \label{prop:ASRelativeSFS}
\end{restatable}

\begin{remark}
  \begin{enumerate}
    \item Corollary \ref{cor:scale_sum_SFS} results from a slight strengthening of Eq. \ref{eq:frequency_limit} in which the frequency $f$ can be random, which is stated in Corollary \ref{cor:Rf_conv_uniform}.
    \item To show Proposition \ref{prop:relative_SFS_small_mut}, we use that when mutations are only associated with births then mutation times occur as a Cox process with intensity $avZ_0(s)\,ds$. This simplifies subsequent calculations. For Eq. \ref{eq:small_freq_small_mutation}, we first show weak convergence of $(Y,v^{-\lambda_1/\lambda_0}W)|\Omega_0^\infty$ as $v\to 0^+$ by showing convergence of Laplace transforms (Proposition \ref{prop:joint_weak_conv}). We conclude by applying continuous mapping. For Eq. \ref{eq:large_freq_small_mutation}, we show weak convergence of $(v^{-\lambda_1/\lambda_0}W,\{v^{-\lambda_1/\lambda_0}\exp(-\lambda_1 T_i)Y_1^{(i)}\}_i\backslash\{0\})$ as $v\to 0^+$ by once again examining its Laplace function. The point process represents a $v^{-\lambda_1/\lambda_0}$-scaling of the contributions of each established clone to the total type-1 population. We then apply continuous mapping with the function $(s,\mu)\to \sup\text{supp}(\mu)/s$, where $\mu$ is a locally-bounded counting measure on $(0,\infty]$.
  \end{enumerate}
\end{remark}

\section{Discussion}

In this work, we have characterized the driver SFS in the small, intermediate, and large frequency regions in various limiting regimes.

\textbf{Comparing how clones of different sizes contribute to the SFS.} The connection between $S_j(t)$ and $R_f(t)$ is that they both measure the number of clones existing at a certain frequency $j$ or $\geq f$, respectively; the difference being the notion of frequency used -- $j$ is in units of number of cells whereas $f$ is in units of the fraction of the type-1 population. In other words, $S_j(t)$ considers the contribution of clones of size $\mathcal{O}(1)$ and $R_f(t)$ considers the contribution of clones of size $\mathcal{O}(\exp(\lambda_1 t))$ 

Taking the $t\to\infty$ limit of $S_j(t)$ and $R_f(t)$ (appropriately scaled) reveals that both have power-law tails $j^{-\lambda_0/\lambda_1-1}$ and $f^{-\lambda_0/\lambda_1}$ as $j\to\infty$ and $f\to0^+$ respectively. There is an intuitive heuristic for why we observe these tails. Suppose that $\gamma_j$ is the initiation time of a clone that grows to size $j$ at the detection time $t$. This clone has time $t-\gamma_j$ to grow to size $j$, so solving $\exp(\lambda_1 (t-\gamma_j))\approx j$ we obtain $\gamma_j\approx t-(1/\lambda_1)\log j$. If a clone is of size $j$, it will take time $\approx 1/(j(a+b+d))$ for the clone to change its size. Thus, the number of clones of size $j$ at time $t$ is:
\[S_j(t)\approx \int_{\gamma_j}^{\gamma_j+1/(j(a+b+d))}(av+\mu)q_1e^{\lambda_0 s}\,ds\approx Ce^{\lambda_0 t}j^{-1-\lambda_0/\lambda_1}.\]

Note then that clones that contribute to the SFS at absolute frequency $j$ occur at $\mathcal{O}(1)$ time units before $t$. For the relative SFS, we instead consider $\gamma_f$ as the initiation time of a clone that grows to size $ fZ_1(t)\approx fWe^{\lambda_1 t}$. Thus, solving as above, we obtain that $\gamma_f\approx  -\frac{1}{\lambda_1}\log fW$. Clones that form earlier also contribute so that:
\[R_f(t)\approx \int_{0}^{\gamma_f}(av+\mu)q_1 e^{\lambda_0 s}\,ds\approx CW^{-\lambda_0/\lambda_1}f^{-\lambda_0/\lambda_1}.\]
The clones that contribute to the relative SFS at frequency $\geq f$ occur within a $\mathcal{O}(1)$ window after the start of the process at time $0$.

Finally, when counting clones of size on order $1\ll g(t)\ll \exp(t\lambda_0\lambda_1/(\lambda_0+\lambda_1))$, Point 3 of Proposition \ref{prop:SFSSparsity_quant} and Corollary \ref{cor:sparsity_part1_2_3}, show that $\exp(-\lambda_0t)g(t)^{1+\lambda_0/\lambda_1}S_{\lceil xg(t)\rceil}(t)\to x^{-1-\lambda_0/\lambda_1}YL_{\lceil xj(\cdot)\rceil}=YCx^{-1-\lambda_0/\lambda_1}$ for $x\in (0,\infty)$, with:
\[C=\frac{\lambda_1(av+\mu)}{(a+b)^2}\left(\frac{a+b}{a+b-d}\right)^{\frac{\lambda_1+\lambda_0}{\lambda_1}}\Gamma(1+\lambda_0/\lambda_1).\]
Thus, we recover power law tails when taking $x\to0^+$ and $x\to\infty$. So, the power-law tails are in some sense universal irrespective of whether we consider frequencies of $\mathcal{O}(1)$, intermediate, or large clones.

\textbf{Estimating parameters from the SFS.} We found that in principle, the SFS can be used to discern certain evolutionary parameters of this continuous-time linear birth-death process. For example, the tails of 
$$\lim_{t\to\infty}S_j(t)/S_{\geq 1}(t)=\lim_{n\to\infty}S_j(\tau_n)/S_{\geq 1}(\tau_n)
$$ and 
$$\lim_{t\to\infty}R_f(t)=\lim_{n\to\infty}R_f(\tau_n)$$
on $\Omega_0^\infty$ follow power-laws $j^{-\lambda_0/\lambda_1-1}$ and $f^{-\lambda_0/\lambda_1}$ as $j\to\infty$ and $f\to 0^+$ respectively, which can be used to estimate the selection effect (i.e. relative increase in fitness between type-0 and type-1 cells).  

We can also construct a consistent estimator for this fitness increase $b$ using the small $j=1$ region of the SFS, assuming all other parameters are known. For example, the parameters of the wild-type population may be known, but the fitness increase $b$ associated with mutation may not be. Define $F(b)=\frac{\int_0^\infty e^{-\lambda_0 s}P_1(s)\,ds}{\int_0^{\infty}e^{-\lambda_0 s}(1-P_0(s))}$. Using the explicit form of $P_j$ (for example provided in \cite{durrettbranching2015}), $P_1$ and $P_0$ are decreasing in $b$, so in total $F(b)$ is monotonically decreasing and is continuous. Hence, $F^{-1}$ is defined and continuous\footnote{That is if $F^{-1}$ is taken as a function over domain $(0,F(0)]$. We can extend $F^{-1}$ continuously onto $[0,1]$ by defining it over $(F(0),1]$ to be $0$}. We may then consider the estimator $\widehat{b}(t)=F^{-1}(f_1(t))$ where $f_1(t):=S_1(t)/S_{\geq 1}(t)$ is the proportion of type-1 lineages at time $t$ of size $1$. By Proposition \ref{prop:fraction_as} on $\Omega_0^\infty$, $f_1(t)\to F(b)$ a.s. as $t\to\infty$ and $f_1(\tau_n)\to F(b)$ as $n\to\infty$, so by continuous mapping $\widehat{b}(t)\to b$ a.s. and $\widehat{b}(\tau_n)\to b$ a.s.

\textbf{A comment on the semi-deterministic model.} We note that Durrett and Moseley \cite{durrett2010} studied an approximation to our model where the type-0 population grows semi-deterministically; this model was later used to investigate tumor heterogeneity, allow for random fitness increases, and quantify the relative SFS with selection \cite{durrett2010,durrettfoo2010,durrett2013,tung}. The justification for the semi-deterministic model is that for small mutation rates, the typical time of forming mutations (conditional on $\Omega_0^\infty$) occurs when the type-0 population is large and thus approximately growing according to $Ye^{\lambda_0 t}$. Our results here recover some quantities (like $R_f(t)$ and the size of the largest clone) computed in prior semi-deterministic models as small-mutation limits of the fully stochastic model (Proposition \ref{prop:relative_SFS_small_mut}).

\textbf{Future work.} The model considered in this work considers a single driver mutation, and does not allow for sequential drivers. In addition, we consider only the SFS of the driver mutation population here, not the SFS of the full tumor population.  

In upcoming work, we extend our analyses to study different components of the site frequency spectrum, including the hitchhiker SFS $S_j^h(t)$, in a more general multi-type population. Additionally, we will explore second-order asymptotics and rare-event deviations in the driver SFS.

\section{Proofs}\label{sec:proofs}
\subsection{Useful tools in our analysis}
\subsubsection{Filtrations used in our analyses}

Here, we define some important $\sigma$-filtrations

\begin{itemize}
    \item Let $\mathcal{F}_t^{\leq i}$ for $i\in \{0,1\}$ denote the smallest $\sigma$-algebra generated either (if $i=0$) by the type-0 process in the time interval $[0,t]$ or (if $i=1$) by the type-0 and type-1 process in the time interval $[0,t]$. The $i=0$ case is most often used, so that we also write $\mathcal{F}_t:=\mathcal{F}_t^{\leq 0}$. We also let $\mathcal{F}_\infty^{\leq i}=\sigma\left(\bigcup_{t\geq 0}\mathcal{F}_t^{\leq i}\right)$. Once again, for shorthand we write $\mathcal{F}_\infty:=\mathcal{F}_\infty^{\leq 0}$.
    \item Let $\mathcal{B}_t^{\leq i}$ for $i\in \{0,1\}$ denote the smallest $\sigma$-algebra containing $\mathcal{F}_t^{\leq i}$ as well as the timings of the type-1 mutations that occur within the interval $[0,t]$. We also write $\mathcal{B}_t:=\mathcal{B}_t^{\leq0}$. We also let $\mathcal{B}_\infty^{\leq i}=\sigma\left(\bigcup_{t\geq 0}\mathcal{B}_t^{\leq i}\right)$. Once again, for shorthand we write $\mathcal{B}_\infty:=\mathcal{B}_\infty^{\leq 0}$.
    \item Let $\mathcal{I}_t$ be the smallest $\sigma$-algebra containing $\mathcal{B}_t^{\leq 1}$ as well as the identities of the clones to which each type-1 cell belongs to. Finally, define $\mathcal{I}_\infty=\sigma\left(\bigcup_{t\geq 0}\mathcal{I}_t\right)$.
\end{itemize}

Clearly, $\mathcal{F}_t^{\leq i}\subseteq \mathcal{B}_t^{\leq i}\subseteq \mathcal{I}_t$. $(Z_0(t),Z_1(t))$ is adapted to $\mathcal{F}_t^{\leq 1}$. $(Z_0(t),Z_1^B(t))$ and $S_j(t)$ are adapted to $\mathcal{I}_t$

\subsubsection{An alternate representation of the model}\label{app:alt_representation_model}

There is an alternate construction of the model we refer to in some places of the proofs, which is equal in distribution to the model introduced in Section \ref{sec:model_notation}. Start with a one-type linear birth-death process with birth rate $a(1-v)$ and death rate $d+\mu$. Equivalently, each type-0 cell has an independent exponential clock with rate $a(1-v)+d+\mu$. Once the clock rings, the cell divides with probability $a(1-v)/(a(1-v)+d+\mu)$ and dies otherwise. We let $\{D_i\}$ denote the timings of all death events in this type-0 process.

Now, we describe the initiation times of type-1 clones. These type-1 clones, conditional on their initiation, grow as branching processes with birth rate $a+b$ and death rate $d$, independently from the other type-1 clones and from the type-0 process. To every cell, attach an independent homogeneous Poisson point process of rate $av$ over its lifetime. Conditional on $\mathcal{F}_\infty$, the summation of these points processes over all cells yields $\{T_i'\}$, a Poisson point process with rate $av Z_0(s)$. These mark the initiation times of the type-1 clones produced in a replication-dependent manner. Let $\{T_i''\}$ be an independent $\mu/(d+\mu)$-thinning of $\{D_i\}$. Then, $\{T_i''\}$ marks the initiation times of the type-1 clones produced in a replication-free manner. Then, $\{T_i\}=\{T_i'\}\cup \{T_i''\}$ are the initiation times of all type-1 clones.

\subsubsection{A comparison criterion for almost sure convergence}

Here, we reformulate Proposition 4 of \cite{gunnarsson2025} (which extends analysis used in Theorem 21 of Harris \cite{harris1963}) to work with more general classes of stochastic processes.

\ConvergenceTool*

\begin{proof}
It suffices to show that $\limsup_{t\to\infty} e^{-\lambda t}(X_t-Y_t)>0$ and $\liminf_{t\to\infty} e^{-\lambda t}(X_t-Y_t)<0$ both occur on sets of probability 0. Note that the integrability condition implies that a.s. $\int_0^\infty e^{-\lambda t}|X_t-Y_t|<\infty$. 

First consider sample paths such that $\limsup_{t\to\infty} e^{-\lambda t}(X_t-Y_t)>0$. For each such path there exists $\delta>0$ and a sequence of times $\{t_k;k\geq 1\}$ with $t_{k+1}-t_k\geq 1$ for which $e^{-\lambda t_k}(X_{t_k}-Y_{t_k})>4\delta$. Due to the almost sure convergence of $\{e^{-\lambda t}Y_t;t\geq 0\}$ there is a $T$ where $t>T$ implies $|e^{-\lambda t}Y_t-\widetilde{Y}_{\infty}|\leq \delta$. For any $k$ sufficiently large (so that $t_k\geq T$):

  \[e^{-\lambda t_k}X_{t_k}>4\delta+e^{-\lambda t_k}Y_{t_k}\geq 3\delta +\widetilde{Y}_{\infty}.\]
Next since $\{X_t;t\geq 0\}$ is monotone we know that for any $0<\Delta<1$, if $t\in [t_k,t_k+\Delta]$:
  \[e^{-\lambda t}X_t\geq e^{-\lambda (t-t_k)}e^{-\lambda t_k}X_{t_k}\geq e^{-\lambda \Delta}(3\delta+\widetilde{Y}_\infty).\]
We can then choose $\Delta$ sufficiently small such that:
  \[e^{-\lambda \Delta}(3\delta+\widetilde{Y}_\infty)\geq (2\delta + \widetilde{Y}_\infty).\]
Then combining with the estimate $\widetilde{Y}_\infty\geq e^{-\lambda t}Y_t-\delta$, we finally obtain:
  \[e^{-\lambda t}(X_t-Y_t)\geq \delta,\]
for $t$ in $\bigcup_{k=K}^\infty[t_k,t_k+\Delta]$. Therefore on the set $\{\omega\in \Omega; \limsup_{t\to\infty}e^{-\lambda t}(X_t-Y_t)>0\}$
  \[\int_0^\infty e^{-\lambda t}|X_t-Y_t|\,dt\geq \int_{\bigcup_{k=K}^\infty[t_k,t_k+\Delta]}\delta=\infty.\]
In particular, $\limsup_{t\to\infty}e^{-\lambda t}(X_t-Y_t)\leq 0$ a.s.

We now consider sample paths such that $\liminf_{t\to\infty}e^{-\lambda t}(X_t-Y_t)<0$. For each such path there exists  $\delta>0$ and $t_k$ a sequence with $t_{k+1}-t_k\geq 1$ (and $t_1\geq 1$) such that $e^{-\lambda t_k}(X_{t_k}-Y_{t_k})<-5\delta$. Fix any $0<\Delta<1$, (and take $t\in [t_k-\Delta,t_k]$) then:

  \[-5\delta >e^{-\lambda t_k}(e^{\lambda\Delta}X_t-e^{\lambda \Delta} Y_{t_k})\geq e^{-\lambda t_k}(e^{\lambda (t_k-t)}X_t-e^{\lambda\Delta}Y_{t_k})=e^{-\lambda t}X_t-e^{\lambda\Delta}e^{-\lambda t_k}Y_{t_k}.\]
Then for $t$ sufficiently large we can again use the a.s. convergence of $\{e^{-\lambda t} Y_t;t\geq 0\}$ to find $T<\infty$ such that $t\geq T$ implies $|\widetilde{Y}_\infty-e^{-\lambda t}Y_t|<\delta$ and thus
  \[-5\delta +\delta e^{\lambda\Delta}\geq e^{-\lambda t}X_t - e^{\lambda\Delta}\widetilde{Y}_\infty.\]
Take $\Delta$ sufficiently small such that $|e^{\lambda\Delta}\widetilde{Y}_\infty-\widetilde{Y}_\infty|<\delta$ and $\delta e^{\lambda\Delta}<2\delta$. Combining these bounds yields 

  \[-2\delta\geq e^{-\lambda t}X_t-\widetilde{Y}_\infty.\]
We can then replace $\widetilde{Y}_\infty$ with $e^{-\lambda t}Y_t$ at the cost of one more $\delta$ to arrive at 
  \[-\delta\geq e^{-\lambda t}X_t-e^{-\lambda t}Y_t\]
for $t\geq T$. In particular, for almost all sample paths such that $\liminf_{t\to\infty}e^{-\lambda t}(X_t-Y_t)<0$ $\int_0^\infty e^{-\lambda t}|X_t-Y_t|\,dt=\infty$, thus the theorem follows.

\end{proof}

\subsubsection{Representation of $W$}\label{app:W_representation}

The following is a useful representation of $W=\lim_{t\to\infty} e^{-\lambda_1 t}Z_1(t)$.

\begin{lemma}
  $W=\sum_{i\in \mathbb{N}} e^{-\lambda_1 T_i}Y_1^{(i)}$, where $Y_1^{(i)}$ are i.i.d. deflated exponentials where the mass at $0$ has probability $p_1$ and the rate of the exponential part is $q_1$.
\end{lemma}

\begin{proof}
  We have:

  \begin{align}
    e^{-\lambda_1 t}Z_1(t)=\sum_{T_i\leq t} e^{-\lambda_1 t}Z_1^{(i)}(t-T_i)=\sum_{i\in \mathbb{N}}\mathbf{1}_{T_i\leq t}e^{-\lambda_1 t}Z_1^{(i)}(t-T_i),\label{eq:W_limit}
  \end{align}
where $Z_1^{(i)}$ are i.i.d. branching processes of each type $1$ clone.  Once again, for the $i$ such that $T_i=\infty$, $Z_1^{(i)}(s)\equiv 0$ for all $s$. We will take the limit as $t\to\infty$ in \eqref{eq:W_limit}, and the conclusion will follow by passing the limit inside the summation.  We have pointwise limits: $\lim_{t\to\infty}\mathbf{1}_{T_i\leq t}e^{-\lambda_1 T_i}e^{-\lambda_1(t-T_i)}Z_1^{(i)}(t-T_i)=e^{-\lambda_1 T_i}Y_1^{(i)}$. Furthermore, we can bound Eq. \ref{eq:W_limit} by:

  \begin{align}
    \sum_{i\in \mathbb{N}}e^{-\lambda_1 T_i}\sup_{0\leq s< \infty}e^{-\lambda_1s}Z_1^{(i)}(s).\label{eq:W_summable}
  \end{align}

  To apply dominated convergence, it suffices to check the right hand side of Eq. \ref{eq:W_summable} has finite expectation. Now, $e^{-\lambda_1s}Z_1^{(i)}(s)$ is a martingale with respect to the filtration generated by $Z_1^{(i)}$. Noting that $e^{-\lambda_1 s}Z_1^{(i)}(s)$ is cadlag and uniformly integrable, we obtain by Doob's maximal inequality that:
  \begin{align*}
      \mathbb{E}[\sup_{0\leq s\leq t}e^{-\lambda_1 s}Z_1^{(i)}(s)]\leq C\mathbb{E}[e^{-\lambda_1 s}Z_1^{(i)}(t)\log(e^{-\lambda_1 t}Z_1^{(i)}(t))+1]
  \end{align*}
  For all $t\geq 0$ and a constant $C$ not depending on $t$ ($x\log x$ evaluated at $0$ is taken to be $0$). Taking $t\to\infty$, by monotone convergence, the left hand side converges to $\mathbb{E}[\sup_{0\leq s<\infty}e^{-\lambda_1 s}Z_1^{(i)}(s)]$. It is further known that $\sup_{0\leq s<\infty}\mathbb{E}[(e^{-\lambda_1 s}Z_1^{(i)}(s))^2]<\infty$ so $\sup_{0\leq s<\infty}\mathbb{E}\left[\left(e^{-\lambda_1 s}Z_1^{(i)}(s)\log(e^{-\lambda_1 s}Z_1^{(i)}(s))\right)^{1+\varepsilon}\right]<\infty$ for any $0<\varepsilon<1$. Hence since the integrand of the upper bound converges to $Y\log Y+1$, by Vitali's convergence theorem we obtain:
  
  \[\mathbb{E}\left[\sup_{0\leq s< \infty}e^{-\lambda_1s}Z_1^{(i)}(s)\right]\leq C\mathbb{E}[Y\log Y+1]<\infty.\]

  Since $Z_1^{(i)}$ is independent from $T_i$, it then suffices to check that $\sum_{i\in \mathbb{N}}e^{-\lambda_1 T_i}$ has finite expectation, but this directly follows from Campbell's Theorem.
\end{proof}

\subsubsection{Approximating hitting times}

The following approximation lemmas will be useful in our analyses.

\begin{lemma}\label{lem:large_size_time}
  Let $X_t$ be a jointly measurable stochastic process such that there exists a $\lambda>0$ such that for $\omega\in A$, a probability $1$ event, $t\to X_t(\omega)$ is bounded over compact sets and $\lim_{t\to\infty}e^{-\lambda t}X_t(\omega)=\theta(\omega)>0$. Then defining $\tau_n=\inf\{t:X(t)\geq n\}$ and $t_n=\frac{1}{\lambda}\log(n/\theta)$, $|\tau_n-t_n|\to 0$.
\end{lemma}

\begin{proof}
  This follows the same proof as Proposition 1 of Gunnarsson 2025 \cite{gunnarsson2025}
\end{proof}

\begin{lemma}\label{lem:point_process_times}
  Let $\Phi=\{T_k\}_{k=1}^\infty$ be a simple point process over $[0,\infty)$ with $\Phi([0,\infty))=\infty$. Denote $M_t=\Phi([0,t])$ for $t\geq 0$. Suppose there exists a $\lambda$ such that $\lim_{t\to\infty}e^{-\lambda t}M_t(\omega)=\theta(\omega)>0$. Then, defining $\widehat{T}_k$ by $k=\theta e^{\lambda \widehat{T}_k}$, $|\widehat{T}_k-T_k|\to 0$.
\end{lemma}

\begin{proof}
  Note that for all $\varepsilon>0$, there is $T(\omega)$ sufficiently large such that $t>T$ implies:

    \[(1-\varepsilon)\theta e^{\lambda t}\leq M_t\leq (1+\varepsilon)\theta e^{\lambda t}.\]
Fix $0<\varepsilon<1/2$ and define $\delta=-\frac{1}{\lambda}\log(1-\varepsilon)$ and $\delta_2=\frac{2}{\lambda}\log(1+\varepsilon)$. Choose $K$ sufficiently large such that $k>K$ implies $\widehat{T}_k+\delta>T$ and $\widehat{T}_k-\delta_2>T$. Note then that
\begin{align*}
    &M_{\widehat{T}_k+\delta}\geq(1-\varepsilon) \theta e^{\lambda \widehat{T}_k}e^{\lambda\delta}=k(1-\varepsilon)e^{\lambda \delta}= k,\\
    &M_{\widehat{T}_k-\delta_2}\leq (1+\varepsilon)\theta e^{\lambda \widehat{T}_k}e^{-\lambda \delta_2}=\frac{k}{1+\varepsilon}<k.
\end{align*}
Hence $\widehat{T}_k+\delta\geq T_k$ and $\widehat{T}_k-\delta_2\leq T_k$ and we obtain $|T_k-\widehat{T}_k|\leq \max\{\delta,\delta_2\}$. Hence, $\lim_{k\to\infty} |T_k-\widehat{T}_k|=0$
\end{proof}

\subsection{Law of large numbers in the small frequency region}\label{sec:SFS_convergence}
We prove Theorem \ref{thm:convergence_results} and Proposition \ref{prop:fraction_as} by proving several lemmas. We use a similar proof structure as in Gunnarsson et al. \cite{gunnarsson2025} where the SFS was decomposed into a sum of monotone processes. However, our $L^2$ approximation scheme is simpler.

\subsubsection{Mean SFS results}\label{sec:SFS_convergence_mean}

Here, we prove Proposition \ref{prop:SFS_mean_asymp}.

\SFSMeanAsymp*
\begin{proof}
  We first establish the results \ref{eq:mean_SFS} and \ref{eq:mean_SFS_random_fitness}. Let $\mathcal{M}_{[s,s+dt]}$ denote the event of having precisely one event - being mutation - in the specified interval and $\mathcal{B}_j(s,t)$ the event that the clone created in $(s,s+dt]$ grew to size $j$ at time $t$; with the assumption that $\mathcal{B}_j(s,t)=\emptyset$ if no mutations are created in the interval $(s,s+dt]$. Then, the contribution of mutations arising within $(s,s+dt]$ to $E[S_j(t)]$ is:
  \begin{align*}
    \mathbb{P}(\mathcal{M}_{[s,s+dt]}\cap \mathcal{B}_j(s,t))+o(dt)&=\mathbb{P}(\mathcal{M}_{[s,s+dt]})\mathbb{P}(\mathcal{B}_j(s,t)|\mathcal{M}_{(s,s+dt]})+o(dt)\\
    &=(av+\mu)\mathbb{E}[Z_0(s)]P_j(t-s)\,dt+o(dt)
  \end{align*}
  where $o(dt)$ accounts for the possibility of multiple events occurring and that $P_j(t-s)$ ought be $P_j(t-s')$ for $s'\in (s,s+dt]$. Summing such contributions recovers Eq. \ref{eq:mean_SFS}. The formula for $\mathbb{E}[S_j^B(t)]$ analogously follows by conditioning $\mathbb{P}(\mathcal{B}_j(s,t)|\mathcal{M}_{[s,s+dt]})$ on the selection amount, distributed as $B$, of that mutant clone.

  Thus,
  \begin{align}
    &e^{-\lambda_0t}\mathbb{E}[S_j(t)]= e^{-\lambda_0t}(av+\mu)\int_0^t\mathbb{E}[Z_0(s)]P_j(t-s)ds=(av+\mu)\int_0^t e^{-\lambda_0 s}P_j(s)ds,\label{eq:exact_mean_driver_sfs}\\
    &e^{-\lambda_0t}\mathbb{E}[S_j^B(t)]= e^{-\lambda_0t}(av+\mu)\int_0^t\mathbb{E}[Z_0(s)]\mathbb{E}[P_j(t-s|B)]ds=(av+\mu)\int_0^t e^{-\lambda_0 s}\mathbb{E}[P_j(s|B)]ds.\label{eq:exact_mean_driver_sfs_random}
\end{align}
We can then easily see
\begin{align}
  \lim_{t\to\infty}e^{-\lambda_0t}\mathbb{E}[S_j(t)]=(av+\mu)\int_0^\infty e^{-\lambda_0 s}P_j(s)\,ds,\label{eq:asymp_mean}\\
  \lim_{t\to\infty}e^{-\lambda_0t}\mathbb{E}[S_j^B(t)]=(av+\mu)\int_0^\infty e^{-\lambda_0 s}\mathbb{E}[P_j(s|B)]\,ds.\label{eq:asymp_mean_random}
\end{align}

  The explicit expansion of $\lim_{t\to\infty}\exp(-\lambda_0 t)\mathbb{E}[S_j(t)]$, its large $j$-scaling behavior, and the large $j$-scaling under the random fitness advance model are computational and left for Appendix Section \ref{app:SFS_mean_j_scaling}.
\end{proof}

\subsubsection{Setup and decomposition for Theorem \ref{thm:convergence_results}}\label{sec:SFS_convergence_notation}
First, we introduce the notation as in \cite{gunnarsson2025}. Define $\tau_{j,-}^i(0)=T_i$ (the initiation time of the $i$-th driver clone produced) and recursively for $k\geq 1$, define
\begin{align*}
  &\tau_{j,+}^i(k)=\inf\{s>\tau_{j,-}^i(k-1):Z_1^{(i)}(s-T_i)=j\}\\
  &\tau_{j,-}^i(k)=\inf\{s>\tau_{j,+}^i(k):Z_1^{(i)}(s-T_i)\neq j\},
\end{align*}
with the convention $\inf\emptyset=\infty$. Then $\tau_{j,+}^i(k)$ is the time that the $i$th mutant clone hits size $j$ for the $k$th time, and $\tau_{j,-}^i(k)$ is the time that the clone exits out of size $j$ for the $k$th time. Then define:
\begin{align*}
  I_{j,\pm}^{(i)}(t)=\sum_{k\in \mathbb{N}}\mathbf{1}_{\tau_{j,\pm}^i(k)\leq t}
\end{align*}
which are the number of times the $i$th clone has entered into or exited size $j$ by time $t$. Finally, we define:
\begin{align}
  S_{j,\pm}^{k}(t)=\sum_{i\in\mathbb{N}} \mathbf{1}_{I^{(i)}_{j,\pm}(t)\geq k}\label{eq:def_S_j_pm}
\end{align}
which are the number of mutant clones that have reached size $j$ at least $k$ times (or exited size $j$ at least $k$ times) by time $t$. Note that $t\to S_{j,\pm}^{k}(t)$ is monotone increasing and:
\begin{align}
  S_j(t)=\sum_{k\in \mathbb{N}} S_{j,+}^{k}(t)-\sum_{k\in \mathbb{N}} S_{j,-}^{k}(t).\label{eq:S_j_decomposition}
\end{align}
We also define:
\[P^k_{j,\pm}(t)=\mathbb{P}\left(I_{j,\pm}^{(1)}(t)\geq k|Z_1^{(1)}(0)=1,Z_0(0)=0\right)\]
which is the probability that a type-1 clone has entered into (or exited from) size $j$ at least $k$ times by time $t$. Note that
\begin{align*}
\mathbf{1}_{\{Z_1(t)=j\}}
= \sum_{k\ge1}\big(\mathbf{1}_{\{I^{(1)}_{j,+}(t)\ge k\}}
-
\mathbf{1}_{\{I^{(1)}_{j,-}(t)\ge k\}}\big),
\end{align*}
since being at size $j$ at time $t$ means the clone has entered size $j$ exactly one more time than it has exited it; taking expectations gives 
\begin{align}
    P_{j}(t)=\sum_{k\in \mathbb{N}}\left(P_{j,+}^k(t)-P_{j,-}^k(t)\right).\label{eq:P_j_decomposition}
\end{align}

We also use the following approximating processes:
\begin{align*}
  &\overline{S}_{j,\pm}^{k}(t)=(av+\mu)\int_0^t Z_0(s)P_{j,\pm}^k(t-s)\,ds\\
  &\widehat{S}_{j,\pm}^{k}(t)=(av+\mu)\int_0^t Ye^{\lambda_0 s}P_{j,\pm}^k(t-s)\,ds\\
  &\overline{S}_j(t)=(av+\mu)\int_0^t Z_0(s)P_{j}(t-s)\,ds\\
  &\widehat{S}_j(t)=(av+\mu)\int_0^t Ye^{\lambda_0 s}P_{j}(t-s)\,ds
\end{align*}

The following uniform bound will be useful in what follows.
\begin{lemma}\label{lem:uniform_bound_kth_entrance}
  There exists $\theta \in (0,1)$, independent of $j,k,b,t$, such that $P_{j,\pm}^k(t)\lesssim \theta^k$ as $k\to\infty$ uniformly in $j,k,b$.\label{lem:P_j_pm_summable}
\end{lemma}
\begin{proof}
  Let $\mathcal{Z}(n)$ denote a discrete time-embedding of the continuous-time linear birth-death process associated with a type-1 clone. Let $\mathcal{Z}(0)=j$. If the first step is a birth, then $\mathbb{P}(\text{exists some time }n\geq 2\text{ with }\mathcal{Z}(n)=j|\mathcal{Z}(1)=j+1)=\frac{d}{a+b}$, since this the extinction probability. Thus 
\[  
\mathbb{P}(\text{exists some time }n\geq 1\text{ with }\mathcal{Z}(n)=j|\mathcal{Z}(0)=j)\leq \frac{d}{a+b}\frac{a+b}{a+b+d}+1\cdot \frac{d}{a+b+d}=2\frac{d}{a+b+d}\leq \frac{2d}{a+d}<1,
\]
and we can thus take $\theta=\frac{2d}{a+d}$. Notice that $P_{j,-}^k(t)\leq P_{j,+}^k(t)$ since $\tau_{j,-}^{(1)}(k)\geq \tau_{j,+}^{(1)}(k)$ for all $k\geq 1$, which implies $I_{j,+}^{(1)}(t)\geq I_{j,-}^{(1)}(t)$ for all $t\geq 0$. The result then follows by the strong Markov property.
\end{proof}

\begin{lemma}\label{lem:CenteredMuts_Martingale}
  For $t>0$ define $R_t=M_t-(av+\mu)\int_0^t Z_0(s)\,ds$ then $\{R_t;t\geq 0\}$ is a 0-mean $\mathcal{B}_t$-martingale. 
\end{lemma}

\begin{proof}
By forming a partition $\Pi_t$ over $[0,t]$ of intervals of length $dt$, let $0<s<t$ and over one such increment $\mathbb{E}[M_s-M_{s-dt}]=(av+\mu)\mathbb{E}[Z_0(s)]dt+o(dt^2)$. Summing over all contributions $\mathbb{E}[M_t]=(av+\mu)\mathbb{E}\left[\int_0^t Z_0(s)\,ds\right].$

  Now fix $t>s\geq 0$, notice:

  \[\mathbb{E}\left[M_t-(av+\mu)\int_0^tZ_0(r)\,dr|\mathcal{B}_s \right]=\mathbb{E}\left[(M_t-M_s)-(av+\mu)\int_s^tZ_0(r)\,dr|\mathcal{B}_s \right]+M_s-(av+\mu)\int_0^sZ_0(r)\,dr\]

  However, using the strong Markov property and independence of lineages:
  \[\mathbb{E}\left[(M_t-M_s)-(av+\mu)\int_s^tZ_0(r)\,dr|\mathcal{B}_s \right]=Z_0(s)\mathbb{E}[M_{t-s}-(av+\mu)\int_0^{t-s}Z_0(r)\,ds]=0\]
\end{proof}

We next show an $L^2$ approximation result for the monotonic processes $\{S_{j,\pm}^k(t);t\geq 0\}$:

\begin{proposition}\label{prop:L2_driver}
  There exists $A<\infty$ such that $\mathbb{E}\left[(S_{j,\pm}^{k}(t)-\overline{S}_{j,\pm}^{k}(t))^2\right]\leq A\theta^k e^{\lambda_0 t}$
\end{proposition}

\begin{proof}
Let $C_t=\mathbb{E}[S_{j,\pm}^k(t)|\mathcal{B}_t]$ and then note that

\begin{align}\label{eq:2ndMomentDecomposition}
    &\mathbb{E}\left[\left(S_{j,\pm}^{k}(t)-\overline{S}_{j,\pm}^{k}(t)\right)^2\right]\nonumber\\
    &=\mathbb{E}\left[\left(S_{j,\pm}^{k}(t)-C_t+C_t-\overline{S}_{j,\pm}^{k}(t)\right)^2\right]\nonumber\\
    &=\mathbb{E}\left[\left(S_{j,\pm}^{k}(t)-C_t\right)^2\right]+2\mathbb{E}\left[\left(S_{j,\pm}^{k}(t)-C_t\right)\left(C_t-\overline{S}_{j,\pm}^{k}(t)\right)\right]+\mathbb{E}\left[\left(C_t-\overline{S}_{j,\pm}^{k}(t)\right)^2\right].
\end{align}

Note that $C_t,\overline{S}_{j,\pm}^{k}(t)$ are both $\mathcal{B}_t$-measurable. Thus using towering the middle term vanishes:

\[\mathbb{E}\left[\left(S_{j,\pm}^{k}(t)-C_t\right)\left(C_t-\overline{S}_{j,\pm}^{k}(t)\right)|\mathcal{B}_t\right]=0.\]
Now note that the first term in \eqref{eq:2ndMomentDecomposition} is the mean of the conditional variance: $\mathbb{E}\left[\mathbb{V}\left(S_{j,\pm}^{k}(t)|\mathcal{B}_t\right)\right].$
Recall our definition of $S_{j,\pm}^{k}(t)$ (Eq. \ref{eq:def_S_j_pm}) which is a sum of indicators of functions of $Z_1^{(i)}$ (which are i.i.d.) and $T_i$. We thus have

\begin{align}
  \mathbb{E}\left[\mathbb{V}\left(S_{j,\pm}^{k}(t)|\mathcal{B}_t\right)\right]&=\mathbb{E}\left[\sum_{T_i\leq t} P_{j,\pm}^{k}(t-T_i)(1-P_{j,\pm}^{k}(t-T_i))\right]\nonumber\\
  &=\int_0^t (av+\mu)\mathbb{E}[Z_0(s)]P_{j,\pm}^{k}(t-s)(1-P_{j,\pm}^{k}(t-s))\,ds\leq (A/2)\theta^k e^{\lambda_0 t}.\label{eq:L2_S_to_C}
\end{align}
With the second to last equality due to Campbell's theorem, and the inequality an application of Lemma \eqref{lem:uniform_bound_kth_entrance}.

For the third term in \eqref{eq:2ndMomentDecomposition}, observing that $C_t=\sum_{T_i\leq t}P_{j,\pm}^{k}(t-T_i)=\int_0^t P_{j,\pm}^k(t-s,1)\,dM_s$, and recalling the definition of $\overline{S}_{j,\pm}^{k}(t)$, we can conclude that $C_t-\overline{S}_{j,\pm}^{k}(t)=\int_0^t P_{j,\pm}^{k}(t-s)\,dR_s$, where $R_t=M_t-(av+\mu)\int_0^t Z_0(s)\,ds$ is a $0$-mean martingale via Lemma \ref{lem:CenteredMuts_Martingale}. The quadratic variation of $R_t$, denoted $[R]_t$ is:
\begin{align*}
    [R]_t=M_t.
\end{align*}
So by Ito's isometry and Campbell's theorem:
\begin{align}
  \mathbb{E}\left[(C_t-\overline{S}_{j,\pm}^{k}(t))^2\right]=\mathbb{E}\left[\int_0^t P_{j,\pm}^{k}(t-s)^2\,dM_s\right]=\int_0^t (av+\mu)e^{\lambda_0 s}P_{j,\pm}^{k}(t-s)^2\,ds\leq (A/2)\theta^k e^{\lambda_0t}\label{eq:L2_C_int_Z}.
\end{align}
Combining \ref{eq:L2_S_to_C} and \ref{eq:L2_C_int_Z} gives the result.
\end{proof}

We will need an estimate of the error between $Z_0$ and $Ye^{\lambda_0 s}$ in $L^2$. This can be calculated explicitly:
\begin{lemma}\label{lem:L2_distance}
    Let $Z(s)$ be a one-type continuous-time linear birth-death process with birth rate $\alpha$ and death rate $\beta$ with $\alpha>\beta$. Denote $\lambda=\alpha-\beta$. Let $Q_t=e^{-\lambda t}Z(t)$ and $Y=\lim_{t\to\infty} Q_t$, almost surely. Then:
    \[\mathbb{E}\left[\left(e^{-\lambda t}Z(t)-Y\right)^2\right]=\frac{\alpha+\beta}{\lambda}e^{-\lambda t}\]
\end{lemma}
\begin{proof}
    This can be computed using the Kolmogorov backwards equation and using various properties of branching processes. There is however an easier way. Note that the quadratic variation:
    \[[Q]_t=\sum_{T\in\Phi_{\text{event}}} e^{-2\lambda T},\]
    where $\Phi_{\text{event}}$ is the point process of event timings, having mean intensity $(\alpha+\beta)e^{\lambda t}\,dt$. $Q_t$ is a $\sigma(\{Z(s)\}_{0\leq s\leq t})$-martingale. Then, we can write $Y-e^{-\lambda t}Z(t)$ as a stochastic integral and use Ito's isometry followed by Campbell's theorem:
    \begin{align*}
        \mathbb{E}\left[\left(Y-e^{-\lambda t}Z(t)\right)^2\right]=\mathbb{E}\left[\left(\int_t^\infty\,dQ_s\right)^2\right]&=\mathbb{E}\left[\int_t^\infty d[Q]_s\right]\\
        &=(\alpha+\beta)\int_t^\infty e^{-\lambda s}\,ds\\
        &=\frac{\alpha+\beta}{\lambda}e^{-\lambda t}
    \end{align*}
\end{proof}

\begin{lemma}\label{lem:pm_approx_via_Y}
There exists a positive constant $A$ such that for all $t>0$
\[
\mathbb{E}\left[(\overline{S}_{j,\pm}^{k}(t)-\widehat{S}_{j,\pm}^{k}(t))^2\right]\leq At\theta^ke^{\lambda_0 t}.
\]
\end{lemma}

\begin{proof}
  By Cauchy-Schwarz:
  \[\left(\int_0^t (av+\mu)P_{j,\pm}^{k}(t-s)(Z_0(s)-Ye^{\lambda_0 s})\,ds\right)^2\leq (av+\mu)^2\int_0^t P_{j,\pm}^{k}(t-s)^2\,ds\int_0^t (Z_0(s)-Ye^{\lambda_0 s})^2\,ds\]
  Note the bounds:
  \[\int_0^t P_{j,\pm}^{k}(t-s)^2\,ds\leq C\theta^k t\]
  and:
  \[\mathbb{E}\left[\left(Z_0(s)-Ye^{\lambda_0 s}\right)^2\right]\lsimas{s}{\infty}e^{\lambda_0 s},\]
  from the prior lemma. The proof is complete combining the above.
\end{proof}

We can now combine our previous lemmas and propositions to establish the following a.s. convergence.

\begin{proposition}
  $\lim_{t\to\infty} e^{-\lambda_0 t}S_j(t)=(av+\mu)Y\int_0^\infty e^{-\lambda_0 s}P_j(s)\,ds$ on $\Omega_0^\infty$.\label{prop:almost_sure_conv}
\end{proposition}

\begin{proof}
  Using Proposition \ref{prop:L2_driver}, we have:

    \[\mathbb{E}\left|S_{j,\pm}^{k}(t)-\overline{S}_{j,\pm}^{k}(t)\right|\leq \sqrt{\mathbb{E}\left|S_{j,\pm}^{k}(t)-\overline{S}_{j,\pm}^{k}(t)\right|^2}\leq C_2\theta^{k/2}e^{(\lambda_0/2) t}.\]
From Lemma \ref{lem:pm_approx_via_Y} we obtain:

\[\mathbb{E}\left|\overline{S}_{j,\pm}^{k}(t)-\widehat{S}_{j,\pm}^{k}(t)\right|\leq C_3\sqrt{t}\theta^{k/2}e^{\lambda_0 t/2},\]
and using the triangle inequality:
\[\mathbb{E}\left|S_{j,\pm}^{k}(t)-\widehat{S}_{j,\pm}^{k}(t)\right|\leq C_4 \theta^{k/2}\sqrt{t}e^{(\lambda_0t)/2}.\]

Define $\widehat{S}_{j,\pm}(t):=\sum_{k\geq 1}\widehat{S}_{j,\pm}^{k}$, which is summable from Lemma \ref{lem:P_j_pm_summable}, and then via the triangle inequality:

\[\mathbb{E}\left|\sum_{k=1}^\infty S_{j,\pm}^{k}(t)-\widehat{S}_{j,\pm}(t)\right|\leq \sum_{k=1}^\infty \mathbb{E}\left|S_{j,\pm}^{k}(t)-\widehat{S}_{j,\pm}^{k}(t)\right|\lesssim\sqrt{t}e^{\lambda_0/2\cdot t}\]
Thus, $\int_0^\infty e^{-\lambda_0 t}\mathbb{E}\left|\sum_{k\geq 1}S_{j,\pm}^{k}(t)-\widehat{S}_{j,\pm}(t)\right|<\infty$. We get a.s. convergence of the sum by noting $e^{-\lambda_0 t}\widehat{S}_{j,\pm}(t)$ converges a.s. and by invoking Proposition \ref{prop:convergence_tool}. By writing $S_j(t)$ as in Eq. \ref{eq:S_j_decomposition} and using the identity Eq. \ref{eq:P_j_decomposition}, the statement is proven.
\end{proof}

We can also extend the above analysis to get a.s. convergence in the case where the selection amount is random:

\begin{corollary}
    Suppose that each clone has random fitness $B$, where $B$ has support within a compact interval $[0,b_{\text{max}}]$. Then:

    \[\lim_{t\to\infty}e^{-\lambda_0t}S_j^B(t)=(av+\mu)Y\int_0^\infty e^{-\lambda_0 s}\mathbb{E}\left[P_j(s|B)\right]\,ds\]
 almost surely.\label{cor:L2_as_random_fitness}
\end{corollary}

\begin{proof}
  Define analogs of $P_{j,\pm}^k$ and the approximations $\overline{S}_{j,\pm}^k,\widehat{S}_{j,\pm}^k$ by replacing $P_{j,\pm}^k$ with the expectation of the probability conditional on selection amount $B$:
  \begin{align*}
      P_{j,\pm}^{k,B}(t)=\mathbb{E}[\mathbb{P}(I_{j,\pm}^{(i)}(t)\geq k|Z_1^{(1)}(0)=1,Z_0(0)=0,B)],
  \end{align*}
  and replacing $P_{j,\pm}^k$ with $P_{j,\pm}^{k,B}$ in the definitions of $\overline{S}_{j,\pm}^k,\widehat{S}_{j,\pm}^k$. Remark still that:
  \[P_{j}^B(t)=\sum_{k\in \mathbb{N}}\left(P_{j,+}^{k,B}(t)-P_{j,-}^{k,B}(t)\right)\]
  by linearity of expectation. Importantly, Lemma \ref{lem:uniform_bound_kth_entrance} follows since $\theta$ is independent of the selection amount. Thus, performing the same analysis, the $L^2$ errors from Proposition \ref{prop:L2_driver} and Lemma \ref{lem:pm_approx_via_Y} are still of the form $\theta^k e^{\lambda_0 t}$, up to a constant multiple. Therefore, the analysis of Proposition \ref{prop:almost_sure_conv} will still show that $S_j^B$ converges almost surely, as desired.

\end{proof}

Then, the large detection-size behavior of $S_j$ follows by approximating $\tau_n$ by $ (1/\lambda_1)\log(n/W)$.

\begin{proposition}
    $\lim_{n\to\infty} \frac{1}{n^{\lambda_0/\lambda_1}}S_j(\tau_n)=(av+\mu)YW^{-\lambda_0/\lambda_1}\int_0^\infty e^{-\lambda_0 s}P_j(s)\,ds$ (conditional on $\Omega_0^\infty$)\label{prop:convergence_detection_size}
\end{proposition}

\begin{proof}
Since $\tau_n\to\infty$ a.s., we can now apply Proposition \ref{prop:almost_sure_conv} to conclude that conditional on $\Omega_\infty$:
  \begin{align}
    \lim_{n\to\infty} e^{-\lambda_0\tau_n}S_j(\tau_n)=(av+\mu)Y\int_0^\infty e^{-\lambda_0 s}P_j(s)\,ds\label{eq:hitting_times_large}.
  \end{align}
    Set $t_n=\frac{1}{\lambda_1}\log\left(\frac{n}{W}\right)$ and note that $|\tau_n-t_n|\to 0$ as $n\to\infty$ a.s. by Lemma \ref{lem:large_size_time}. This can then be combined with \eqref{eq:hitting_times_large} to conclude
    \begin{align}
        \lim_{n\to\infty} e^{-\lambda_0\tau_n}S_j(\tau_n)&=\lim_{n\to\infty} \left(e^{-\lambda_1(\tau_n-t_n)}e^{-\lambda_1t_n}\right)^{\lambda_0/\lambda_1}S_j(\tau_n)\nonumber\\
        &=\lim_{n\to\infty}\frac{W^{\lambda_0/\lambda_1}}{n^{\lambda_0/\lambda_1}}S_j(\tau_n).\label{eq:size_large}
    \end{align}
The desired result then follows by combining Eqs. \ref{eq:hitting_times_large} and \ref{eq:size_large}.
\end{proof}

Finally, Theorem \ref{thm:convergence_results} can easily be shown.

\ConvergenceResults*

\begin{proof}
  $L^2$ convergence follows by adapting Proposition \ref{prop:L2_driver} and Lemma \ref{lem:pm_approx_via_Y}. Specifically, Proposition \ref{prop:L2_driver} can be adapted to show $\mathbb{E}\left[\left(S_j(t)-\overline{S}_j(t)\right)^2\right]\lesssim e^{\lambda_0 t}$ by replacing $S_{j,\pm}^{k}$ with $S_j$ and $P_{j,\pm}^{k}$ with $P_j$. Similarly, adapting Lemma \ref{lem:pm_approx_via_Y}, we get $\mathbb{E}\left[(\overline{S}_j(t)-\widehat{S}_j(t))^2\right]\lesssim te^{\lambda_0 t}$. Thus $e^{-\lambda_0 t}S_j(t)$ converges in $L^2$.
  
  Almost sure convergence for $e^{-\lambda_0 t}S_j(t)$ and $n^{-\lambda_0/\lambda_1}S_j(\tau_n)$ were established in Propositions \ref{prop:almost_sure_conv} and \ref{prop:convergence_detection_size} respectively. $L^2$ and a.s. convergence in the random fitness advance setting follows from Corollary \ref{cor:L2_as_random_fitness}. 
\end{proof}

Finally, we can easily prove convergence results for $S_{\geq j}(t)$, counting the number of mutations present in $\geq j$ cells at time $t$.

\RelativeTypeAS*

\begin{proof}
  Conduct the same analysis as to show convergence for $S_j$, but take $\tau_{j,+}^i(k),\tau_{j,-}^i(k)$ to be the times the $i$th mutant clone becomes \textit{at least} size $j$ $k$ times and the times the $i$th mutant clone exits from the region $\{j,j+1,\dots\}$ $k$ times.

  For the series of equalities in \ref{eq:normalized_SFS_limit}, the first three expressions are clearly equal by the almost sure limits. Therefore, it remains to check the last equality. From Proposition \ref{prop:SFS_mean_asymp}:
  \begin{align*}
      \lim_{t\to\infty}e^{-\lambda_0 t}\mathbb{E}[S_j(t)]&=(av+\mu)\frac{q_1}{\lambda_1}\frac{\Gamma(j)\Gamma(\lambda_0/\lambda_1+1)}{\Gamma(j+\lambda_0/\lambda_1+1)}\sum_{n=0}^\infty \frac{(\lambda_0/\lambda_1)_n(j)_n}{(j+1+\lambda_0/\lambda_1)_n n!}p_1^n\\
      &=(av+\mu)\frac{q_1}{\lambda_1}\frac{\Gamma(j)\Gamma(\lambda_0/\lambda_1+1)}{\Gamma(j+\lambda_0/\lambda_1+1)}{}_2F_1(\lambda_0/\lambda_1,j;j+\lambda_0/\lambda_1+1;p_1).
  \end{align*}
  By monotone convergence:
  \begin{align*}
      \lim_{t\to\infty}e^{-\lambda_0 t}\mathbb{E}[S_{\geq 1}(t)]=(av+\mu)\sum_{j\geq 1}\int_0^\infty P_j(s)e^{-\lambda_0 s}\,ds
  \end{align*}
  and by using \ref{eq:SFS_mean_integral}
  \begin{align*}
      &=(av+\mu)\frac{q_1}{\lambda_1}\int_0^1 (1-y)^{\lambda_0/\lambda_1-1}(1-p_1 y)^{-\lambda_0/\lambda_1}\,dy\\
      &=(av+\mu)\frac{q_1}{\lambda_0}{}_2F_1(\lambda_0/\lambda_1,1;\lambda_0/\lambda_1+1;p_1).
  \end{align*}
  Dividing both yields the last equality.
\end{proof}

\subsection{Sparsity of mutations in the intermediate frequency region}\label{sec:intermediate_SFS}
Here we prove Proposition \ref{prop:SFSSparsity_quant}, Theorem \ref{thm:SFSSparsityAS}, and Corollary \ref{cor:SFSSparsity_qual}.

It will prove useful to obtain a simpler expression of $\mathbb{E}[S_{j}(t)]$ which decouples $j$ and $t$:
\begin{lemma}\label{lem:changing_j_mean}
  Define:
  \begin{align}
      I_t(n)=\int_{ne^{-\lambda_1t}}^{n}\frac{\lambda_1u^{\frac{\lambda_0}{\lambda_1}}}{((a+b)-d(u/n))^2}\left(\frac{(a+b)n-(a+b)u}{(a+b)n-du}\right)^{n-1}\,du\label{eq:I_t_n_def}
  \end{align}
  Then:
  \begin{align}
    \mathbb{E}[S_{j}(t)]=(av+\mu)e^{\lambda_0 t}j^{-\frac{\lambda_1+\lambda_0}{\lambda_1}}I_t(j)\label{eq:meanSFS_varying}
  \end{align}
\end{lemma}

Furthermore, $\sup_{n\in \mathbb{N},t\geq 0}I_t(n)$ is bounded; if $j(t)\ll e^{\lambda_1 t}$, then $I_t(j(t))$ converges to a constant if $j(t)$ is non-decreasing.
\begin{proof}
  By Proposition \ref{prop:SFS_mean_asymp}
    \[\mathbb{E}[S_{j}(t)]=(av+\mu)e^{\lambda_0 t}\int_0^t e^{-\lambda_0 s}P_{j}(s)\,ds.\]
    Explicitly write out $P_j(s)$ using \ref{eq:P_j} and use the change of variables $u=je^{-\lambda_1 s}$:
    \begin{align}
        &=\frac{(av+\mu)}{\lambda_1}e^{\lambda_0 t}\int_{je^{-\lambda_1 t}}^{j}u^{\frac{\lambda_0}{\lambda_1}}j^{-\frac{\lambda_0}{\lambda_1}}\frac{\lambda_1^2(j/u)}{((a+b)(j/u)-d)^2}\left(\frac{(a+b)(j/u)-(a+b)}{(a+b)(j/u)-d}\right)^{j-1}\,\frac{du}{u}\nonumber\\
        &=(av+\mu)e^{\lambda_0 t}j^{-\left(\frac{\lambda_1+\lambda_0}{\lambda_1}\right)}I_t(j)\nonumber
    \end{align}
    proving the equation \ref{eq:meanSFS_varying}. Denote $I_\infty(n)$ by the same equation as \ref{eq:I_t_n_def} but with the lower integration bound set to $0$. Remark that by dominated convergence:
    \begin{align}
        I_\infty(n)\overset{n\to\infty}{\to} \int_0^\infty \frac{\lambda_1}{(a+b)^2}u^{\frac{\lambda_0}{\lambda_1}}e^{-\frac{a+b-d}{a+b}u}\,du\label{eq:limit_I_inf}
    \end{align}
    which is strictly positive. Using that $t\to I_t(n)$ is increasing:
    \[\sup_{n\in \mathbb{N},t\geq 0}I_t(n)=\sup_{n\in \mathbb{N}}I_\infty(n)<\infty.\]

    Finally, under the condition that $j(t)\ll e^{\lambda_1 t}$ and $j(t)\to\infty$ as $t\to\infty$, we may write:
    \begin{align}
        I_t(j(t))=I_\infty(j(t))+(I_t(j(t))-I_\infty(j(t))),\label{eq:I_t_rewrite}
    \end{align}
    where the second term converges to $0$ by dominated convergence. So, $I_t(j(t))$ converges to the same limit as on the right-hand-side of \ref{eq:limit_I_inf}. If $j(t)\to N$, then examining \ref{eq:I_t_rewrite} once again, $I_t(j(t))$ converges to $I_\infty(N)$. \end{proof}

\subsubsection{Quantitative convergence theorems}\label{sec:quant_conv}
Here, we prove Proposition \ref{prop:SFSSparsity_quant}.
\SFSSparsityQuant*
We first prove the easiest parts of this proposition, namely point (1) and that the limits $L,L_{j(\cdot)}$ exist, as a direct corollary of Lemma \ref{lem:changing_j_mean}.
\begin{corollary}\label{cor:sparsity_part1_2_3}
The following hold:
\begin{itemize}
    \item $j(t)\gg\exp(t\lambda_0\lambda_1/(\lambda_0+\lambda_1))$ implies $S_{j(t)}(t)\to 0$ in $L^1$.
    \item $j(t)\sim C\exp(t\lambda_0\lambda_1/(\lambda_0+\lambda_1))$ implies:
    \[\lim_{t\to\infty}\mathbb{E}[S_{j(t)}(t)]=\frac{\lambda_1(av+\mu)}{(a+b)^2}C^{-\frac{\lambda_1+\lambda_0}{\lambda_1}}\left(\frac{a+b}{a+b-d}\right)^{\frac{\lambda_1+\lambda_0}{\lambda_1}}\Gamma(1+\lambda_0/\lambda_1).\]
    \item $j(t)\ll \exp(t\lambda_0\lambda_1/(\lambda_0+\lambda_1))$ implies
    \begin{align*}
        \lim_{t\to\infty}\mathbb{E}[e^{-\lambda_0 t}j(t)^{(\lambda_1+\lambda_0)/\lambda_1}S_{j(t)}(t)]=\begin{cases}
            \frac{\lambda_1(av+\mu)}{(a+b)^2}\left(\frac{a+b}{a+b-d}\right)^{\frac{\lambda_1+\lambda_0}{\lambda_1}}\Gamma(1+\lambda_0/\lambda_1) & \text{if }\lim_{t\to\infty}j(t)=\infty\\
            (av+\mu)I_\infty(N) & \text{if }\lim_{t\to\infty}j(t)=N,\,N<\infty
        \end{cases}
    \end{align*}
\end{itemize}
\end{corollary}
\begin{proof}
    Let $j(t)\gg \exp(t(\lambda_0\lambda_1/(\lambda_0+\lambda_1)))$.  So, $e^{\lambda_0 t}j(t)^{-(\lambda_1+\lambda_0)/\lambda_1}\ll 1$, so that using \ref{eq:meanSFS_varying} and that $I_t(n)$ is uniformly bounded over all $t\geq 0$ and $n$:
    \[\mathbb{E}[|S_{j(t)}(t)|]=\mathbb{E}[S_{j(t)}(t)]\to 0.\]
    The limits on mean in the second two cases once again follow from using the expression \ref{eq:meanSFS_varying} with $j(t)\sim C\exp(t\lambda_0\lambda_1/(\lambda_0+\lambda_1))$ and $j(t)\ll \exp(t(\lambda_0\lambda_1/(\lambda_0+\lambda_1)-\varepsilon))$.
\end{proof}

For the proof of point (2) of the theorem, it will be useful to control the term $P_j(t)$:
\begin{lemma}\label{lem:P_j_bound}
    $P_j(t)\leq 1/j$.
\end{lemma}
\begin{proof}
    Let
    \[\beta(t)=\frac{(a+b)e^{\lambda_1 t}-(a+b)}{(a+b)e^{\lambda_1 t}-d}.\]
    Then using \ref{eq:P_j}, we can bound $P_j(s)$ for any time $t$ by writing:
    \begin{align*}
        P_j(t)\leq (1-\beta(t))\beta(t)^{j-1}.
    \end{align*}
    Note that $\beta(z)\in [0,1]$, so that
    \begin{align*}
        \leq \max_{\beta\in [0,1]}(1-\beta)\beta^{j-1}\leq \frac{1}{j}.
    \end{align*}
\end{proof}

Now, we examine characteristic functions and then use the Levy continuity theorem to prove part 2 of Proposition \ref{prop:SFSSparsity_quant}.

\begin{lemma}\label{lem:sparsity_part2}
    Suppose $j(t)\sim C\exp(t\lambda_0\lambda_1/(\lambda_0+\lambda_1))$ and $L=\lim_{t\to\infty}\mathbb{E}[S_{j(t)}(t)]$. Then:
    \[\left|\mathbb{E}\left[e^{i\theta S_{j(t)}(t)}\mathbf{1}_{\Omega_0^\infty}\right]-\mathbb{E}\left[e^{YL(e^{i\theta}-1)}\mathbf{1}_{\Omega_0^\infty}\right]\right|\overset{t\to\infty}{\to} 0\]
\end{lemma}
\begin{proof}
    We examine the difference of the characteristic functions of $S_{j(t)}(t)$ and the proposed limit in distribution:
    \begin{align}
        \left|\mathbb{E}\left[e^{i\theta S_{j(t)}(t)}\mathbf{1}_{\Omega_0^\infty}\right]-\mathbb{E}\left[e^{YL(e^{i\theta}-1)}\mathbf{1}_{\Omega_0^\infty}\right]\right|&\leq \mathbb{E}\left[\left|\mathbb{E}[e^{i\theta S_{j(t)}(t)}|\mathcal{B}_\infty] -e^{\mathbb{E}[S_{j(t)}|\mathcal{B}_\infty](e^{i\theta}-1)}\right|\right]\label{eq:boundary_term_1}\\
        &+\mathbb{E}\left[\left|e^{\mathbb{E}[S_{j(t)}|\mathcal{B}_\infty](e^{i\theta}-1)}-e^{\overline{S}_{j(t)}(t)(e^{i\theta}-1)}\right|\right]\label{eq:boundary_term_2}\\
        &+\mathbb{E}\left[\left|e^{\overline{S}_{j(t)}(t)(e^{i\theta}-1)}-e^{\widehat{S}_{j(t)}(t)(e^{i\theta}-1)}\right|\right]\label{eq:boundary_term_3}\\
        &+\mathbb{E}\left[\left|e^{\widehat{S}_{j(t)}(t)(e^{i\theta}-1)}-e^{YL(e^{i\theta}-1)}\right|\right]\label{eq:boundary_term_4}
    \end{align}
    We now show that each term above tends to $0$ as $t\to\infty$. We start with Term \ref{eq:boundary_term_1}. Remark that:
    \begin{align*}
        \mathbb{E}[e^{i\theta S_{j(t)}(t)}|\mathcal{B}_\infty]=\prod_{T_i\leq t}\left(1+P_{j(t)}(t-T_i)(e^{i\theta}-1)\right).
    \end{align*}
    Furthermore:
    \begin{align*}
        e^{\mathbb{E}[S_{j(t)}(t)|\mathcal{B}_\infty](e^{i\theta}-1)}=\prod_{T_i\leq t}e^{P_j(t-T_i)(e^{i\theta}-1)}.
    \end{align*}
    Now, notice that for complex numbers $a_k,b_k$ where $1\leq k\leq n$, if $|a_k|,|b_k|\leq 1$ then:
    \begin{align*}
        \left|\prod a_k-\prod b_k\right|\leq \sum_{k} |a_k-b_k|.
    \end{align*}
    Remark further that complex number $z$ with $\text{Re}(z)\leq 0$ and any $k\geq 0$:
    \begin{align}
        \left|e^{z}-\sum_{\ell=0}^k z^\ell/\ell!\right|\leq \frac{|z|^{k+1}}{(k+1)!}.\label{eq:exponential_bound}
    \end{align}
    Combining these facts:
    \begin{align*}
        \left|\prod_{T_i\leq t}\left(1+P_{j(t)}(t-T_i)(e^{i\theta}-1)\right)-\prod_{T_i\leq t}e^{P_j(t-T_i)(e^{i\theta}-1)}\right|&\leq \sum_{T_i\leq t}\left|1+P_{j(t)}(t-T_i)(e^{i\theta}-1)-e^{P_{j(t)}(t-T_i)(e^{i\theta}-1)}\right|\\
        &\leq 2 \sum_{T_i\leq t}P_{j(t)}(t-T_i)^2.
    \end{align*}
    We then take expectation over the times $\{T_i\}_i$ apply Lemma \ref{lem:P_j_bound} on a single factor $P_{j(t)}(t-T_i)$ and use Lemma \ref{lem:changing_j_mean}:
    \begin{align*}
        \ref{eq:boundary_term_1}\leq \frac{1}{2j(t)}\mathbb{E}[S_{j(t)}(t)]\to 0.
    \end{align*}

    Now the terms \ref{eq:boundary_term_2}, \ref{eq:boundary_term_3}, \ref{eq:boundary_term_4} will tend to $0$ if it is shown that $\mathbb{E}[S_{j(t)}|\mathcal{B}_\infty]-\overline{S}_{j(t)}(t),\overline{S}_{j(t)}(t)-\widehat{S}_{j(t)}(t),\widehat{S}_{j(t)}(t)-YL$ converge to $0$ in $L^1$, respectively. This is because for $A,B\geq 0$:
    \[\left|e^{A\left(e^{i\theta}-1\right)}-e^{B\left(e^{i\theta}-1\right)}\right|=\left|e^{(A-B)\left(e^{i\theta}-1\right)}-1\right|\leq |A-B|\left|e^{i\theta}-1\right|\leq 2|A-B|\]
    with the inequality coming from \ref{eq:exponential_bound}. This is what we will show using the ideas from Proposition \ref{prop:L2_driver} and Lemma \ref{lem:pm_approx_via_Y}. Firstly:
    \begin{align*}
      \ref{eq:boundary_term_2}&\leq 2\mathbb{E}\left[\left|\mathbb{E}[S_{j(t)}(t)|\mathcal{B}_\infty]-\overline{S}_{j(t)}(t)\right|\right]\\
      &\leq 2\sqrt{\mathbb{E}\left[\left(\int_0^t P_{j(t)}(t-s)\,dR_s\right)^2\right]}\\
      &= 2\sqrt{\mathbb{E}\left[\int_0^t P_{j(t)}(t-s)^2\,d[R]_s\right]}\\
      &=2\frac{1}{\sqrt{j(t)}}\sqrt{\mathbb{E}[S_{j(t)}(t)]}\to 0.
    \end{align*}

    For term \ref{eq:boundary_term_3}, we have:
    \begin{align*}
      \ref{eq:boundary_term_3}&\leq 2\mathbb{E}\left[\left|\overline{S}_{j(t)}(t)-\widehat{S}_{j(t)}(t)\right|\right]\\
      &\leq 2(av+\mu)\int_0^t P_{j(t)}(t-s)\sqrt{\mathbb{E}[(Z_0(s)-Ye^{\lambda_0 s})^2]}\,ds.\
    \end{align*}

    From Lemma \ref{lem:L2_distance}, the $L^2$ distance between $Z_0(s)$ and $Ye^{\lambda_0 s}$ is $\mathcal{O}(e^{\lambda_0 s})$. We can also use Lemma \ref{lem:P_j_bound} and that $\frac{\lambda_0\lambda_1}{\lambda_0+\lambda_1}>\frac{1}{2}\lambda_0$ to see:
    \begin{align*}
      &\leq C\int_0^t P_{j(t)}(t-s)e^{s\lambda_0/2}\,ds\\
      &\leq C\frac{te^{t\lambda_0/2}}{j(t)}\to 0.
    \end{align*}

    Remark that $\widehat{S}_{j(t)}(t)=Y\mathbb{E}[S_{j(t)}(t)]$ so that trivially:
    \begin{align*}
      \ref{eq:boundary_term_4}\leq 2\mathbb{E}[Y]|\mathbb{E}[S_{j(t)}(t)]-L|\to 0.
    \end{align*}
\end{proof}

To conclude the proof, we show point (3)
\begin{lemma}\label{lem:sparsity_part3}
Suppose $j(t)\ll \exp(t\lambda_0\lambda_1/(\lambda_0+\lambda_1))$. Let $L=\lim_{t\to\infty}\mathbb{E}[e^{-\lambda_0 t}j(t)^{\frac{\lambda_1+\lambda_0}{\lambda_1}}S_{j(t)}(t)]$. Then:
\[\mathbb{E}[|e^{-\lambda_0 t}j(t)^{\frac{\lambda_1+\lambda_0}{\lambda_1}}S_{j(t)}(t)-YL|]\overset{t\to\infty}{\to}0.\]
\end{lemma}

\begin{proof}
  As in Section \ref{sec:SFS_convergence}, we may decompose
  \begin{align}
    |S_{j(t)}(t)-e^{\lambda_0 t}j(t)^{-\frac{\lambda_1+\lambda_0}{\lambda_1}}YL|&\leq |S_{j(t)}(t)-\mathbb{E}[S_{j(t)}(t)|\mathcal{B}_\infty]|\label{eq:beta_term_1}\\
    &+|\mathbb{E}[S_{j(t)}(t)|\mathcal{B}_\infty]-\overline{S}_{j(t)}(t)|\label{eq:beta_term_2}\\
    &+|\overline{S}_{j(t)}(t)-\widehat{S}_{j(t)}(t)|\label{eq:beta_term_3}\\
    &+|\widehat{S}_{j(t)}(t)-Ye^{\lambda_0 t}j(t)^{-\frac{\lambda_1+\lambda_0}{\lambda_1}}L|.\label{eq:beta_term_4}
  \end{align}

  Then using analogs of Proposition \ref{prop:L2_driver} and Lemma \ref{lem:pm_approx_via_Y} but replacing $P_{j,\pm}^k$ with $P_{j(t)}$ and using expression \ref{eq:meanSFS_varying}:
  \begin{align*}
    \max\{\mathbb{E}[(S_{j(t)}(t)-\mathbb{E}[S_{j(t)}(t)|\mathcal{B}_\infty])^2],\mathbb{E}[(\mathbb{E}[S_{j(t)}(t)|\mathcal{B}_\infty]-\overline{S}_{j(t)}(t))^2],\mathbb{E}[(\overline{S}_{j(t)}(t)-\widehat{S}_{j(t)}(t))^2]\}\lesssim e^{\lambda_0 t}j(t)^{-\frac{\lambda_1+\lambda_0}{\lambda_1}}.
  \end{align*}
Therefore:
  \begin{align*}
    e^{-\lambda_0 t}j(t)^{\frac{\lambda_1+\lambda_0}{\lambda_1}}\mathbb{E}[|S_{j(t)}(t)-\widehat{S}_{j(t)}(t)|]\lesssim e^{-\frac{\lambda_0}{2}t}j(t)^{\frac{\lambda_1+\lambda_0}{2\lambda_1}}\to 0.
  \end{align*}
Finally, term \ref{eq:beta_term_4} is controlled in a similar manner to how \ref{eq:boundary_term_4} was controlled in the previous lemma. Namely:
\begin{align*}
    \mathbb{E}[\ref{eq:beta_term_4}]=|\mathbb{E}[S_{j(t)}(t)]-e^{\lambda_0 t}j(t)^{-\frac{\lambda_1+\lambda_0}{\lambda_1}}L|\ll e^{\lambda_0 t}j(t)^{-\frac{\lambda_1+\lambda_0}{\lambda_1}}
\end{align*}
Combining everything:
\[\mathbb{E}\left|e^{-\lambda_0 t}j(t)^{\frac{\lambda_1+\lambda_0}{\lambda_1}}S_{j(t)}(t)-YL\right|\to 0.\]
\end{proof}

\subsubsection{Cutoff in the intermediate frequency region}
Here, we prove Theorem \ref{thm:SFSSparsityAS} and end by proving Corollary \ref{cor:SFSSparsity_qual}.

\SFSSparsityAS*

We first investigate the detection time series of equalities (Eq. \ref{eq:cutoff_time}). To prove the limits regarding $j_k^-(t),j_k^+(t)$, it suffices to prove the following on $\Omega_0^\infty$:

\[\frac{\lambda_0\lambda_1}{\lambda_0+\lambda_1}\leq \lim_{k\to\infty}\liminf_{t\to\infty}\frac{\log(j_k^-(t))}{t},\lim_{k\to\infty}\limsup_{t\to\infty}\frac{\log(j_k^+(t))}{t}\leq \frac{\lambda_0\lambda_1}{\lambda_0+\lambda_1}.\]
This is because we trivially have $j_k^-(t)\leq j_{k+1}^+(t)+1$ for every $t\geq 0$ and $k$. Before proving these two inequalities, we require probability bounds on the tails of $S_j(t)$.

We first bound the left tail of $S_j(t)$ on the event that the type-0 population never becomes too small.

\begin{lemma}\label{lem:upper_bound_Sj_strong}
  Fix $\Delta>0$ and $\xi=\log j/\lambda_1$. If:
  \[t>\xi+\Delta,\]
  then:
  \begin{align*}
    \mathbb{P}(S_j(t)\leq k,B_x)\leq \delta^{-k}e^{-c(1-\delta)xe^{\lambda_0 t}j^{-1-\lambda_0/\lambda_1}}
  \end{align*}
  for any $\delta\in (0,1)$ and where $c$ is some constant depending on $\Delta$.
\end{lemma}
\begin{proof}
  By considering the contributions of type-0 cells alive at least $\xi$ time units before $t$ to the SFS. See Appendix Section \ref{app:upper_bound_Sj_strong}.
\end{proof}

And indeed, event $B_x$ occurs quite frequently conditional on $\Omega_0^\infty$ as $x\to 0^+$.
\begin{lemma}\label{lem:large_deviation_bound}
    We have the following large-deviation bound:
    \begin{align*}
        \mathbb{P}(\inf_{t\geq 0} e^{-\lambda_0 t}Z_0(t)\leq x|\Omega_0^\infty)\lsimas{x}{0^+}\sqrt{x}.
    \end{align*}
\end{lemma}
\begin{proof}
    By Doob's submartingale inequality and the Chernoff bound. See Appendix Section \ref{app:large_deviation_bound}.
\end{proof}

The right tail requires some more preparation. Firstly, we require a lemma that says that mutations timings do not abnormally cluster.
\begin{lemma}\label{lem:V_moments_finite}
    Define:
    \[V:=\sum_{i}e^{-2\lambda_0 T_i}.\]
    For all $k$, $\mathbb{E}[V^k]<\infty$.
\end{lemma}
\begin{proof}
    By a first-step analysis. See Appendix Section \ref{app:V_moments_finite}
\end{proof}

We use this to prove finite moments of the supremum of the scaled count of mutations.
\begin{lemma}\label{lem:sup_M_finite}
Define:
  \[M^\ast:=\sup_{z\in [0,\infty)}\left(M_z e^{-\lambda_0 z}\right)\]
  For all $k\geq 1$, $\mathbb{E}[(M^\ast)^k]<\infty$.
\end{lemma}
\begin{proof}
  By a martingale decomposition and Lemma \ref{lem:V_moments_finite}. See Appendix Section \ref{app:sup_M_finite}.
\end{proof}

The right tail can now be controlled.
\begin{lemma}\label{lem:lower_bound_Sj_strong}
  For every $k\geq 1$, there is a $C>0$ such that:
  \begin{align*}
    \mathbb{P}(S_j(t)\geq k)\leq Ce^{k\lambda_0 t}j^{-\left(1+\lambda_0/\lambda_1\right)k}.
  \end{align*}
\end{lemma}
\begin{proof}
  Uses a representation of $S_j(t)$ as a sum of i.i.d. indicators conditional on $\mathcal{B}_\infty$ and the finite moments of $M^*$ (Lemma \ref{lem:sup_M_finite}). See Appendix Section \ref{app:lower_bound_Sj_strong}.
\end{proof}

Finally, we prove the main results.

\IntermediateJKPlus*
\begin{proof}
    Fix any $\varepsilon>0$. Define $j(t)\asymp \exp\left(\left(\lambda_0\lambda_1/(\lambda_0+\lambda_1)+\varepsilon\right)t\right)$ such that for each $n$, $j(t)=j_n$ is constant for $t\in [n,n+1)$. Suppose we show that:
  \begin{align}
    \sum_{n=1}^\infty \mathbb{P}\left(\exists j\geq j_n,\sup_{t\in [n,n+1)}S_{j}(t)\geq k+1\right)<\infty\label{eq:to_show_upper}
  \end{align}
  for sufficiently large $k$ as in the proposition statement. Then by Borel-Cantelli, there is a $T=T(\omega)$ (depending on the sample point) such that $t\geq T$ implies:
  \[\frac{\log(j_k^+(t))}{t}\leq \frac{\lambda_0\lambda_1}{\lambda_0+\lambda_1}+\varepsilon\Longrightarrow \limsup_{t\to\infty}\frac{\log(j_k^+(t))}{t}\leq \frac{\lambda_0\lambda_1}{\lambda_0+\lambda_1}+\varepsilon,\]
  almost surely, as desired. It remains to show \ref{eq:to_show_upper}. Perform a union bound:
  \begin{align}
      \mathbb{P}\left(\exists j\geq j_n,\,\sup_{t\in [n,n+1)}S_{j}(t)\geq k+1\right)\leq \sum_{j\geq j_n}\mathbb{P}\left(\sup_{t\in [n,n+1)}S_{j}(t)\geq k+1\right).\label{eq:union_bound_jn}
  \end{align}
  Firstly, we can break up each of the events in \ref{eq:union_bound_jn} by whether $S_{j_n}(n)$ is larger than $k+1$ or not:
  \begin{align}
    \sum_{j\geq j_n}\mathbb{P}\left(\sup_{t\in [n,n+1)}S_{j}(t)\geq k+1\right)\leq \sum_{j\geq j_n}\mathbb{P}\left(S_{j}(n)\geq k+1\right)+\sum_{j\geq j_n}\mathbb{P}\left(S_{j}(n)\leq k,\sup_{t\in [n,n+1)}S_{j}(t)\geq k+1\right).\label{eq:sup_S_j_n_trivial_bound}
  \end{align}
  Define event $E_n^j:=\{S_{j}(n)\leq k,\sup_{t\in [n,n+1)}S_{j}(t)\geq k+1\}$. Letting $m$ denote the Lebesgue measure on $\mathbb{R}$, we define the local time $L_{B,j}^{\geq k}:=m\{t\in B:S_j(t)\geq k\}$ for Borel set $B$. To control the second term, note that on event $E_n^j$, $\zeta_n:=\inf\{t\geq n:S_{j}(t)\geq k+1\}$ is an almost surely finite stopping-time (with respect to the filtration by $\mathcal{I}_t$) and stops within $[n,n+1)$. Remark that since $(Z_0,Z_1)$ is a linear birth-death process, $S_j$ changes by increments of $1$ (almost surely), so that by right-continuity $S_{j}(\zeta_n)=k+1$. Starting from time $\zeta_n$, a necessary condition for $S_j(t)\leq k$ for $t\geq \zeta_n$ is that \textit{some} event occur to one of the original $k+1$ clones of size $j_n$ present at time $\zeta_n$ by time $t$. Thus:
  \[L_{[n,n+2],j}^{\geq k+1}\geq \mathbf{1}_{E_n^j}\min\{\eta,1\},\]
  where -- by the Strong Markov property -- conditional on $\mathcal{I}_{\zeta_n}$, $\eta$ is the minimum of $k+1$ independent exponentials each of rate $(a+b+d)j$. Using that $E_n^j$ is $\mathcal{I}_{\zeta_n}$-measurable and taking expectations:
  \begin{align}
    \mathbb{P}(E_n^j)\leq \frac{j(k+1)}{1-e^{-j(k+1)(a+b+d)}}\mathbb{E}[L_{[n,n+2]}^{\geq k+1}].\label{eq:sup_bound}
  \end{align}

  Remark that:
  \begin{align}
    L_{[n,n+2]}^{\geq k+1}=\int_n^{n+2}\mathbf{1}_{S_{j}(s)\geq k+1}\,ds\label{eq:local_time_decomposition}
  \end{align}
  Therefore, to bound \ref{eq:sup_bound}, it suffices to bound $\mathbb{P}(S_{j}(s)\geq k+1)$. Using Lemma \ref{lem:lower_bound_Sj_strong} then gives for $s\in [n,n+2]$:
    \begin{align}
      \mathbb{P}(S_{j}(s)\geq k+1)\leq C_2 e^{(k+1)\lambda_0t}j^{-(1+\frac{\lambda_0}{\lambda_1})(k+1)}.\label{eq:S_j_n_tight_bound}
    \end{align}
    Summing over all $j\geq j_n$ gives
    \begin{align}
        \sum_{j\geq j_n}\mathbb{P}(S_j(s)\geq k+1)\leq C_3e^{(k+1)\lambda_0 t}j_n^{-(1+\frac{\lambda_0}{\lambda_1})(k+1)+1}\leq C_4e^{-(c_1^\varepsilon k-c_2)n},\label{eq:S_j_n_sum_tight_bound}
    \end{align}
    where $c_1^\varepsilon,c_2>0$, $c_1^\varepsilon$ depends on $\varepsilon$ (scales linearly with $\varepsilon$), and $s\in [n,n+2]$. Therefore, using \ref{eq:local_time_decomposition} along with the union bound \ref{eq:union_bound_jn} implies that for $k$ such that $c_1^\varepsilon k-c_2>0$ (or equivalently for some $C_1>0$, $k>C_1/\varepsilon$), the second term on the right-hand-side of \ref{eq:sup_S_j_n_trivial_bound} is summable (over $n$). \ref{eq:S_j_n_sum_tight_bound} implies that the first term on the right-hand-side of \ref{eq:sup_S_j_n_trivial_bound} is summable (over $n$) as well. Thus, \ref{eq:to_show_upper} holds.
\end{proof}

Because $\varepsilon>0$ is arbitrary above and $k\to j_k^+(t)$ is non-increasing, we trivially obtain as a Corollary

\begin{corollary}\label{cor:limsup_j_k_plus}
  \begin{align}
      \lim_{k\to\infty}\limsup_{t\to\infty}\frac{\log(j_k^+(t))}{t}\leq \frac{\lambda_0\lambda_1}{\lambda_0+\lambda_1}\label{eq:limsup_j_k_plus}
  \end{align}
  on $\Omega_0^\infty$
\end{corollary}

Now, we examine the lower bound. This requires controlling the left tail instead of the right tail of $S_j(t)$.
\begin{proposition}\label{prop:liminf_j_k_mimus}
  \[\frac{\lambda_0\lambda_1}{\lambda_0+\lambda_1}\leq \lim_{k\to\infty}\liminf_{t\to\infty}\frac{\log(j_k^-(t))}{t}\]
  on $\Omega_0^\infty$
\end{proposition}
\begin{proof}
Fix any $\varepsilon>0$. Take $\lambda_1>\beta>0$. Define $j(t)\asymp \exp\left((\lambda_0\lambda_1/(\lambda_0+\lambda_1)-\varepsilon)t\right)$ such that every $n$, $j(t)=j_n$ is constant over $t\in [n,n+1)$. Let $E_n=\{\sup_{t\in [n,n+1]}Z_1(t)\leq e^{(\lambda_1+\beta)n}\}$, where $\beta\in (0,\lambda_1)$. Notice that since $e^{-\lambda_1 t}Z_1(t)$ converges almost surely to a finite random variable, $\mathbb{P}(E_n\text{ eventually})=1$. Then, to prove the claim it suffices to show:
  \begin{align}
    \sum_{n=1}^\infty \mathbb{P}\left(\exists j\leq j_n,\inf_{t\in [n,n+1)}S_{j}(t)\leq k,\Omega_0^\infty,E_n\right)<\infty\label{eq:to_show_lower}
  \end{align}

  Letting $x_n=n^{-4}$, the probability of the event $B_n^c=\{\inf_{t}Z_0e^{-\lambda_0 t}\leq x_n\}$ conditional on $\Omega_0^\infty$ is summable by Lemma \ref{lem:large_deviation_bound} and in particular $P(B_n^c\text{ i.o.})=0$. So, to prove \ref{eq:to_show_lower} it suffices to show:
\begin{align*}
    \sum_{n=1}^\infty \mathbb{P}\left(\exists j\leq j_n,\inf_{t\in [n,n+1)}S_{j}(t)\leq k,\Omega_0^\infty,E_n,B_n\right)<\infty.
  \end{align*}

  Now, partition $[n,n+1)$ into $M_n\ordas{n}{\infty}\exp(2\lambda_1 n)$ equally-spaced half-open intervals. Denote the intervals $\{[t_i,t_{i+1})\}_{i=1}^{M_n}$. By applications of unions bounds, it suffices to show:
  \begin{align*}
      \sum_{n=1}^\infty\sum_{j\leq j_n}\sum_{i=1}^{M_n} \mathbb{P}\left(\inf_{t\in [t_i,t_{i+1})}S_{j}(t)\leq k,E_n,B_n\right)<\infty.
  \end{align*}

  Using Lemma \ref{lem:upper_bound_Sj_strong} with say $\delta=0.5$, we obtain:
  \begin{align*}
    \max_{j\leq j_n,1\leq i\leq M_n}\mathbb{P}(B_n\cap S_j(t_i)\leq 2k)\lsimas{n}{\infty}\exp\left(-0.5cn^{-4}\exp((\lambda_0\lambda_1/(\lambda_0+\lambda_1))\varepsilon n)\right).
  \end{align*}
  This decays doubly-exponentially so that summing $\mathbb{P}(B_n\cap S_j(t_i)\leq 2k)$ over $(n,j,i)$ is finite. Therefore, it suffices to prove:
  \begin{align*}
      \sum_{n=1}^\infty\sum_{j\leq j_n}\sum_{i=1}^{M_n} \mathbb{P}\left(S_j(t_i)\geq 2k,\inf_{t\in [t_i,t_{i+1})}S_{j}(t)\leq k,Z_1(t_i)\leq e^{(\lambda_1+\beta)n}\right)<\infty.
  \end{align*}

  For the event to occur, it must be that at least $k$ events occur to the type-1 members present at the beginning of the interval (at time $t_i$) by the end of the interval (by time $t_{i+1}$). Hence, the above probability is dominated by the tail of a binomial $B$ with $\lceil e^{(\lambda_1+\beta)n}\rceil$ trials with success probability $p=1-\exp(-(a+b+d)/M_n)$. Using $M_n\gsimas{n}{\infty}\exp(2\lambda_1 n)$ and by a factorial tail bound:
  \[\mathbb{P}\left(S_j(t_i)\geq 2k,\inf_{t\in [t_i,t_{i+1})}S_{j}(t)\leq k,Z_1(t_i)\leq e^{(\lambda_1+\beta)n}\right)\leq \mathbb{P}(B\geq k)\leq \frac{(\lceil e^{(\lambda_1+\beta)n}\rceil p)^k}{k!}\lsimas{n}{\infty}e^{(\beta-\lambda_1)kn}\]
Further using that $M_n\lsimas{n}{\infty}\exp(2\lambda_1 n)$ and that $j_n\lsimas{n}{\infty}\exp(\frac{\lambda_0\lambda_1}{\lambda_0+\lambda_1}n)$:
  \[\sum_{j\leq j_n}\sum_{i=1}^{M_n} \mathbb{P}\left(S_j(t_i)\geq 2k,\inf_{t\in [t_i,t_{i+1})}S_{j}(t)\leq k,Z_1(t_i)\leq e^{(\lambda_1+\beta)n}\right)\lsimas{n}{\infty}e^{\frac{\lambda_0\lambda_1}{\lambda_0+\lambda_1}n+2\lambda_1n+(\beta-\lambda_1)kn}.\]
  Taking $k$ large enough ensures that \ref{eq:to_show_lower} holds.
  
\end{proof}

The detection size results follow quite easily.
\begin{corollary}
  Eq. \ref{eq:cutoff_size} holds.
\end{corollary}
\begin{proof}
  Note that $\tau_n$ is a discrete subsequence of time so that on $\Omega_0^\infty$:
    \begin{align*}
        \frac{\lambda_0\lambda_1}{\lambda_0+\lambda_1}\leq \lim_{k\to\infty}\liminf_{n\to\infty}\frac{\log j_k^-(\tau_n)}{\tau_n},\, \lim_{k\to\infty}\limsup_{n\to\infty}\frac{\log j_k^+(\tau_n)}{\tau_n}\leq  \frac{\lambda_0\lambda_1}{\lambda_0+\lambda_1}.
    \end{align*}
    Once again, since $j_k^-(\tau_n)\leq j_k^+(\tau_n)+1$ and $\tau_n\to\infty$, we immediately conclude
    \begin{align*}
        &\lim_{k\to\infty}\liminf_{n\to\infty}\frac{\log j_k^-(\tau_n)}{\tau_n}=\lim_{k\to\infty}\limsup_{n\to\infty}\frac{\log j_k^-(\tau_n)}{\tau_n}\\
        =&\lim_{k\to\infty}\liminf_{n\to\infty}\frac{\log j_k^+(\tau_n)}{\tau_n}=\lim_{k\to\infty}\limsup_{n\to\infty}\frac{\log j_k^+(\tau_n)}{\tau_n}=\frac{\lambda_0\lambda_1}{\lambda_0+\lambda_1}.
    \end{align*}
Combining this with Lemma \ref{lem:large_size_time}, which implies that $\tau_n\sim \frac{1}{\lambda_1}\log n$, proves the result.
\end{proof}

We now prove Corollary \ref{cor:SFSSparsity_qual}.
\SFSSparsityQual*
\begin{proof}
    Remark that $\liminf_{t\to\infty}S_{j(t)}(t)\leq \liminf_{n\to\infty}S_{j(\tau_n)}(\tau_n)$, so to show the first line, it suffices to prove $\liminf_{t\to\infty}S_{j(t)}(t)=\infty$. Note that $\liminf_{t\to\infty}\frac{\log j_k^-(t)}{t}$ is decreasing in $k$ so that Theorem \ref{thm:SFSSparsityAS} implies that for all $k\geq 1$, $\liminf_{t\to\infty}\frac{\log j_k^-(t)}{t}\geq \frac{\lambda_0\lambda_1}{\lambda_0+\lambda_1}$. Since $\limsup_{t\to\infty}\frac{\log j(t)}{t}\leq \frac{\lambda_0\lambda_1}{\lambda_0+\lambda_1}-\varepsilon$, there is some time $T$ such that for all $t\geq T$, $j(t)<j_k^-(t)$. Then by definition $S_{j(t)}(t)>k$ for $t\geq T$, i.e., $\liminf_{t\to\infty}S_{j(t)}(t)\geq k$ and since $k$ is arbitrary, the first line is proven.

    For the second line, let $k_\varepsilon=\lceil c_2/c_1^{\varepsilon/2}+2\rceil$. Then as per Proposition \ref{prop:intermediate_j_k_plus}:
    \[\limsup_{t\to\infty}\frac{\log(j_k^+(t))}{t}\leq \frac{\lambda_0\lambda_1}{\lambda_0+\lambda_1}+\frac{\varepsilon}{2}.\]
    Thus, there is a $T(\omega)$ such that $t\geq T(\omega)$ implies $J(t)>j_k^+(t)$. For such $t$ by definition of $j_k^+(t)$, $S_{J(t)}(t)\leq k_\varepsilon\lesssim\varepsilon^{-1}$ as $\varepsilon\to 0^+$. And, once again the fixed detection-size limit follows from the detection-time limit.

    The last line follows from (1) of Proposition \ref{prop:SFSSparsity_quant}. In particular, it implies that there is a subsequence of times $t_n\to \infty$ such that $S_{j(t_n)}(t_n)\to 0$.
\end{proof}

\subsection{Frequencies of large clones}\label{sec:proof_frequencies_large_clones}
Here, we provide proofs for the statements in Section \ref{sec:large_frequencies}. We begin with Theorem \ref{thm:relative_SFS} and Corollary \ref{cor:scale_sum_SFS}. We then the prove Proposition \ref{prop:relative_SFS_small_mut}.

\subsubsection{Proofs for Eqs. \ref{eq:frequency_limit} and \ref{eq:support_frequency_limit} and Corollary \ref{cor:scale_sum_SFS}}\label{sec:proof_thm_corollary_relative_SFS}

Prior to showing Eq. \ref{eq:frequency_limit}, we require an important lemma regarding the minuscule contribution of type-1 clones formed late.

\UniformCloneControl*

\begin{proof}
    It suffices to show that for all $\varepsilon>0$ there is an $I$ for which $i\geq I$ implies $e^{-\lambda_1 T_i}Y_*^{(i)}\leq \varepsilon$ (a.s.).

    Define $A_{i,\varepsilon}=\{e^{-\lambda_1 T_i}Y_*^{(i)}>\varepsilon\}$. By Borel-Cantelli, it suffices to show that for each $\varepsilon>0$, $\sum_{i=1}^\infty \mathbb{P}(A_{i,\varepsilon})<\infty$. By Doob's Martingale inequality:
\[
\mathbb{P}(A_{i,\varepsilon})=\mathbb{E}[\mathbb{P}(Y_*^{(i)}>\varepsilon e^{\lambda_1 T_i}|T_i)]\leq \frac{1}{\varepsilon}\mathbb{E}[e^{-\lambda_1 T_i}].
\]
It thus suffices to show that $\sum_{i=1}^\infty \mathbb{E}\left[e^{-\lambda_1 T_i}\right]<\infty$, which is immediate from Campbell's theorem.
\end{proof}

And therefore:
\begin{proposition}
  Eq. \ref{eq:frequency_limit}, which states:
  \[\lim_{t\to\infty}R_f(t)=\sum_{i\in \mathbb{N}}\mathbf{1}_{e^{-\lambda_1 T_i}Y_1^{(i)}>fW}\]
  holds on $\Omega_0^\infty$.
\end{proposition}
\begin{proof}
  Conditional on $\Omega_0^\infty$, there a.s exists a $\varepsilon>0$ such that $Wf>\varepsilon>0$. But from Lemma \ref{lem:Y_small}, for $i>I$ sufficiently large, $e^{-\lambda_1 T_i}Y_1^{(i)}<\varepsilon $, thus the proposed limit is a.s. finite.
  
  Recall that:
  \[
  R_f(t)=\sum_{1\leq i\leq M_t} \mathbf{1}_{Z_1^{(i)}(t-T_i)>Z_1(t)f}=\sum_{1\leq i\leq M_t} \mathbf{1}_{e^{-\lambda_1 T_i}e^{-\lambda_1(t-T_i)}Z_1^{(i)}(t-T_i)>e^{-\lambda_1 t}Z_1(t)f}.
  \]

  Let $\delta=\varepsilon/2$ and note there exists a time $T$ such that for all $t>T$, $|e^{-\lambda_1 t}Z_1(t)f-Wf|\leq \delta$. Without loss of generality, consider $T$ even larger such that $i>M_T$ implies that $|e^{-\lambda_1 T_i}Y_*^{(i)}|\leq \delta$, so:
  \begin{align*}
    \lim_{t\to\infty} R_f(t)&=\lim_{t\to\infty} \sum_{1\leq i\leq M_T} \mathbf{1}_{Z_1^{(i)}(t-T_i)>Z_1(t)f}+\sum_{M_T< i\leq M_t} \mathbf{1}_{Z_1^{(i)}(t-T_i)>Z_1(t)f}\\
    &=\sum_{1\leq i\leq M_T}\mathbf{1}_{e^{-\lambda_1 T_i}Y_1^{(i)}>Wf}+0\\
    &=\sum_{i\in \mathbb{N}}\mathbf{1}_{e^{-\lambda_1 T_i}Y_1^{(i)}>Wf}.
  \end{align*}
In the second equality, we are justified in taking the limit within the indicator in the first sum as long as $e^{-\lambda_1 T_i}Y_1^{(i)}/W\neq f$ (which is a probability 1 event), from continuous mapping.
\end{proof}

In fact, a uniform in $f$ analog of Eq. \ref{eq:frequency_limit} holds.

\begin{corollary}\label{cor:Rf_conv_uniform}
  Let $\{F_i\}_i$ be an enumeration (sorted in descending order) of the jump discontinuities of the limit of $R_f(\infty)$, where:
  \[\lim_{f\to F_i^-}\sum_{i\in \mathbb{N}}\mathbf{1}_{e^{-\lambda_1 T_i}Y_1^{(i)}>Wf}\neq \lim_{f\to F_i^+}\sum_{i\in \mathbb{N}}\mathbf{1}_{e^{-\lambda_1 T_i}Y_1^{(i)}>Wf}.\]
  Then,
  \[\mathbb{P}\left(\forall f\in [0,1]\backslash \{F_i\}_i,\,\lim_{t\to\infty}R_f(t)=\sum_{i\in \mathbb{N}}\mathbf{1}_{e^{-\lambda_1 T_i}Y_1^{(i)}>Wf}|\Omega_0^\infty\right)=1\]
\end{corollary}
\begin{proof}
  Obviously by countable subadditivity and Eq. \ref{eq:frequency_limit}:
  \begin{align}
    \mathbb{P}\left(\forall f \in [0,1]\cap \mathbb{Q},\,\lim_{t\to\infty}R_f(t)=\sum_{i\in \mathbb{N}}\mathbf{1}_{e^{-\lambda_1 T_i}Y_1^{(i)}>Wf}|\Omega_0^\infty\right)=1.\label{eq:uniform_conv_countable}
  \end{align}
  Now, consider the interval $A_i=(F_{i+1},F_i)$. Obviously since $f\to R_f(t)$ is decreasing, for $f_1<f_2,f_1,f_2\in A_i$, if $\lim_{t\to\infty}R_{f_1}(t)=\lim_{t\to\infty}R_{f_2}(t)=L$, then by the squeeze theorem $f\in (f_1,f_2)$ implies $\lim_{t\to\infty}R_f(t)=L$. Therefore by Eq. \ref{eq:uniform_conv_countable}:
  \begin{align*}
    \mathbb{P}\left(\forall f \in A_i,\,\lim_{t\to\infty}R_f(t)=\sum_{i\in \mathbb{N}}\mathbf{1}_{e^{-\lambda_1 T_i}Y_1^{(i)}>Wf}|\Omega_0^\infty\right)=1.
  \end{align*}
  Writing $[0,1]\backslash \{F_i\}_i=\bigcup_i A_i$ gives the result
\end{proof}

This result gives a simple proof of Corollary \ref{cor:scale_sum_SFS}.

\ScaleSumSFS*

\begin{proof}
  On $\Omega_0^\infty$, $W>0$. Furthermore, infinitely many established type-1 clones will be formed, say the first one is $Z_1^{(n_1)},Z_1^{(n_2)}$. Therefore:
  \[
  \frac{Z_1^{(n_1)}(t-T_{n_1})}{Z_1(t)}=e^{-\lambda_1 T_{n_1}}\frac{e^{-\lambda_1(t-T_{n_1})}Z_1^{(n_r)}(t-T_{n_1})}{e^{-\lambda_1 t}Z_1(t)}\overset{t\to\infty}{\to} L_1>0.
  \]
So, the largest clone makes up a fraction $<1-\varepsilon$ of the population for some $\varepsilon>0$ and all sufficiently large times $t>T$ (for example, with $\varepsilon=L_1/2$ say). Take $n>N$ such that $\tau_n>T$ and such that $k(n)/n<\varepsilon/2$ and thus $\sum_{j=n-k(n)}^n S_j(\tau_n)=0$. For the right upper limit, eventually $G^+(t)>Z_1(t)$ for all $t>T$.

For the middle row, right column, take $F(t)=x\exp(\lambda_1 t)/Z_1(t)$. Note that $\lim_{t\to\infty}F(t):=F_\infty=x/W$ on $\Omega_0^\infty$. Note that $F_\infty\not\in \{F_i\}_i$ a.s. Thus, define $F_+,F_-$ such that there exists an $i$ where $F_\infty\in (F_-,F_+)\subseteq (F_{i+1},F_i)$. By Corollary \ref{cor:Rf_conv_uniform} and squeezing $F_-\leq F(t)\leq F_+$ for large $t$:
\[\lim_{t\to\infty}R_{F_+}(t)=\lim_{t\to\infty}R_{F_-}(t)=\lim_{t\to\infty}R_{F(t)}(t)=\sum_{i\in \mathbb{N}}\mathbf{1}_{Y_1^{(i)}>xe^{\lambda_1 T_i}}.\]
For the middle row, left column, the result follows from noting that $Z_1(\tau_n)+Z_0(\tau_n)=n$, using that $Z_0(\tau_n)=o_{n\to\infty}(n)$, and using a similar set of comparisons as used for the middle row, right column.

  The lower left and right limits follow by similar arguments. For example, for every $i$, there is a $T$ such that $t\geq T$ implies $G^-(t)/Z_1(t)<F_i$.
\end{proof}

We now prove Eq. \ref{eq:support_frequency_limit}
\begin{proposition}
  Eq. \ref{eq:support_frequency_limit}, which states:
  \[\sup\supp R_f(t)=\max_{i:T_i\leq t}(Z_1^{(i)}(t-T_i))/Z_1(t)\overset{t\to\infty}{\to}\sup_i\left(e^{-\lambda_1 T_i}Y_1^{(i)}\right)/W\]
  holds on $\Omega_0^\infty$.
\end{proposition}

\begin{proof}
  The first equality in Eq. \ref{eq:support_frequency_limit} is immediate. For the limit, we show that the $\liminf$ is larger than the limit and the $\limsup$ is smaller than it. 
  For any $K$ and sufficiently large $t$, we have that:
\[
\max_{T_i\leq t}\left( e^{-\lambda_1 T_i}Y_1^{(i)}(t)\right)\geq \max_{1\leq i\leq K}\left( e^{-\lambda_1 T_i}Y_1^{(i)}(t)\right).
\]
Taking $\liminf_{t\to\infty}$ -- noting that $\lim_{t\to\infty}Y^{(i)}(t)=Y_1^{(i)}$ -- and sending $K\to\infty$ reveals:
\[
\liminf_{t\to\infty}\max_{1\leq i\leq M(t)}\left( e^{-\lambda_1 T_i}Y_1^{(i)}(t)\right)\geq \sup_{i}\left( e^{-\lambda_1 T_i}Y_1^{(i)}\right).
\]
Now, we use that $\lim_{i\to\infty}e^{-\lambda_1 T_i}Y_*^{(i)}=0$ (Lemma \ref{lem:Y_small}). For any $\varepsilon>0$, there is sufficiently large $I$ such that for all $i\geq I$, $e^{-\lambda_1 T_i}Y_*^{(i)}\leq \varepsilon$. Hence:
\begin{align*}
        \limsup_{t\to\infty}\max_{1\leq i\leq M(t)}\left( e^{-\lambda_1 T_i}Y_1^{(i)}(t)\right)&\leq \limsup_{t\to\infty}\max\{\max_{1\leq i\leq I}\left(e^{-\lambda_1 T_i}Y_1^{(i)}(t)\right),\varepsilon\}\\
        &=\max\{\max_{1\leq i\leq I}\left(e^{-\lambda_1 T_i}Y_1^{(i)}\right),\varepsilon\}\\
        &\leq \max\{\sup_{i}\left(e^{-\lambda_1 T_i}Y_1^{(i)}\right),\varepsilon\}.
\end{align*}
Taking $\varepsilon\to0^+$ gives the result.
\end{proof}

\subsubsection{Investigating the $f\to 0^+$ limit.}

Here, we finish proving Theorem \ref{thm:relative_SFS} by showing Eq. \ref{eq:frequency_small_limit}.

Define $\widetilde{R}_x(\infty)=\sum_{i}\mathbf{1}_{e^{-\lambda_1 T_i}Y_1^{(i)}>x}$. As mentioned, since $\widetilde{R}_{fW}(\infty)=R_f(\infty)$, it suffices to show Eq. \ref{eq:x_limit_semi}. With the change of variables $x(r)=\exp(-r)$, we can thus show for $H(r)=\widetilde{R}_{x(r)}(\infty)$:
\begin{align}
  \lim_{r\to\infty}e^{-\lambda_0/\lambda_1 r}H(r)=\frac{(av+\mu) q_1^{1-\lambda_0/\lambda_1}}{\lambda_1}\Gamma(\frac{\lambda_0}{\lambda_1})Y.\label{eq:small_freq_to_show}
\end{align}

We proceed via Proposition \ref{prop:convergence_tool}. Define the approximations:
\begin{align*}
  &\overline{H}(r)=(av+\mu)\int_0^\infty \mathbb{P}(e^{-\lambda_1 s}Y_1^{(1)}>e^{-r})Z_0(s)\,ds\\
  &\widehat{H}(r)=(av+\mu)Y\int_0^\infty \mathbb{P}(e^{-\lambda_1 s}Y_1^{(1)}>e^{-r})e^{\lambda_0 s}\,ds.
\end{align*}
The integral is finite by Markov's inequality. Remark that:
\begin{align}
  \mathbb{E}[H(r)]=\mathbb{E}\left[\sum_{T_i}\mathbb{P}(e^{-\lambda_1 T_i}Y_1^{(1)}>e^{-r})\right]=(av+\mu)\int_0^\infty\mathbb{P}(e^{-\lambda_1 s}Y_1^{(1)}>e^{-r})e^{\lambda_0 s}\,ds.\label{eq:H_mean}
\end{align}
Therefore:
\[\widehat{H}(r)=Y\mathbb{E}[H(r)].\]
Now, we obtain the following approximation results:
\begin{proposition}\label{prop:approx_relative}
  \[\max\{\mathbb{E}[|H(r)-\overline{H}(r)|],\mathbb{E}[|\overline{H}(r)-\widehat{H}(r)|]\}\lesssim e^{\frac{\lambda_0}{2\lambda_1}r}\]
\end{proposition}
\begin{proof}
  Analogous to Proposition \ref{prop:L2_driver} and Lemma \ref{lem:pm_approx_via_Y}
\end{proof}

Finally note the following large $r$ behavior of $\mathbb{E}[H(r)]$:
\begin{lemma}\label{lem:approx_H_conv}
  \[e^{-\frac{\lambda_0}{\lambda_1}r}\mathbb{E}[H(r)]= \frac{(av+\mu)q_1^{1-\lambda_0/\lambda_1}}{\lambda_1}\Gamma(\frac{\lambda_0}{\lambda_1},q_1e^{-r}),\]
  where $\Gamma(\cdot,\cdot)$ denotes the upper-incomplete Gamma function.
\end{lemma}
\begin{proof}
  Note that $Y_1^{(1)}$ is $0$ with probability $p_1$ and $\text{Exp}(q_1)$ with probability $q_1$. Therefore using \ref{eq:H_mean} and the change of variables $z=q_1\exp(\lambda_1 s-r)$:
  \begin{align*}
    \mathbb{E}[H(r)]&=(av+\mu)\int_0^\infty q_1e^{-q_1\exp(\lambda_1 s-r)}e^{\lambda_0 s}\,ds\\
    &=\frac{(av+\mu)q_1^{1-\frac{\lambda_0}{\lambda_1}}}{\lambda_1}e^{\frac{\lambda_0}{\lambda_1}r}\int_{q_1e^{-r}}^\infty e^{-z}z^{\lambda_0/\lambda_1-1}\,dz,
  \end{align*}
  proving the result.
\end{proof}
Take $X_r=H(r),Y_r=\overline{H}(r)$ as in Proposition \ref{prop:convergence_tool}. By Proposition \ref{prop:approx_relative}, condition \ref{eq:convergence_tool_condition} is satisfied. Therefore, by the convergence result for $Y_r$ (Lemma \ref{lem:approx_H_conv}), we obtain Eq. \ref{eq:frequency_small_limit}.

\subsubsection{Investigating the small mutation limit.}\label{sec:small_mutation}
Here, we prove Proposition \ref{prop:relative_SFS_small_mut}. Note in the setting where $\mu=0$, mutation times are distributed as a PPP with rate $avZ_0(s)\,ds$ conditional on $\mathcal{F}_\infty$.

\begin{proposition}\label{prop:joint_weak_conv}
  As $v\to 0^+$, $(Y_v,v^{-\lambda_1/\lambda_0}W_v)|\Omega_{0,v}^\infty$ converges in distribution to a random vector with Laplace Transform:
  \begin{align}
    M_0(x,y)=\frac{q_{0,0}}{q_{0,0}+x+y^{\lambda_{0,0}/\lambda_1}\frac{aq_1^{1-\lambda_{0,0}/\lambda_1}}{\lambda_1}\frac{\pi}{\sin(\pi\lambda_{0,0}/\lambda_1)}}\label{eq:laplace_transform_limit}
  \end{align}
  The $v$-subscript in $Y_v,W_v,\Omega_{0,v}^\infty,q_{0,v},\lambda_{0,v}$ denotes the dependence on the mutation rate $v$.\label{prop:conv_dist}
\end{proposition}
\begin{proof}
  We also explicitly write $Z_{0,v}(t)=Z_0(t)$ to denote the type-0's dependence on $v$. Conditional on $\mathcal{F}_\infty$, $\widetilde{\Phi}=\sum_{i}\delta_{(T_i,Y_1^{(i)})}$ is a Poisson point process on $\mathbb{R}_{\geq 0}\times \mathbb{R}_{\geq 0}$. Let $f_y(t,\ell)=ye^{-\lambda_1 t}\ell$ and denote the distribution of $Y_1^{(i)}$ by $\mathbb{Y}$. Then:

\begin{align*}
  M_v(x,y)=\mathbb{E}\left[e^{-xY_v-yW_v}\right]&=\mathbb{E}[e^{-xY_v}\mathbb{E}[e^{-\int f_y(t,\ell)\widetilde{\Phi}(dt\times d\ell)}|\mathcal{F}_{\infty}]]\\
  &=\mathbb{E}[e^{-xY}e^{-\int_0^\infty \int_0^\infty (1-e^{-ye^{-\lambda_1 s}\ell})\,\mathbb{Y}(d\ell)avZ_{0,v}(s)\,ds}]\\
  &=\mathbb{E}[e^{-xY}e^{-\int_0^\infty q_1av (1-q_1/(q_1+ye^{-\lambda_1 s}))Z_{0,v}(s)\,ds}].
\end{align*}

First note that the exponent term in the expectation of $M_v(0,v^{-\lambda_1/\lambda_{0,v}}y)$ can be simplified as follows:

  \[
  \int_0^\infty q_1 av\left(1-\frac{q_1}{q_1+yv^{-\lambda_1/\lambda_{0,v}}e^{-\lambda_1 s}}\right)Z_{0,v}(s)\,ds=\int_0^\infty q_1 av\left(\frac{yv^{-\lambda_1/\lambda_{0,v}}}{q_1e^{\lambda_1 s}+yv^{-\lambda_1/\lambda_{0,v}}}\right)Z_{0,v}(s)\,ds,
  \]
letting $u=e^{\lambda_1 s}$ we then get
\begin{align*}
  &=\frac{q_1a}{\lambda_1}\int_{1}^\infty v\frac{yu^{-1}}{q_1 uv^{\lambda_1/\lambda_{0,v}}+y}Z_{0,v}\left(\frac{1}{\lambda_1}\log u\right)\,du
\end{align*}
we then make the substitution $r=uv^{\lambda_1/\lambda_{0,v}}$ to see
\begin{align}
  &=\frac{q_1 a}{\lambda_1}\int_{0}^\infty \mathbf{1}_{r\geq v^{\lambda_1/\lambda_{0,v}}}\frac{yr^{-1}}{q_1 r+y}vZ_{0,v}\left(\frac{1}{\lambda_1}\log\left(rv^{-\lambda_1/\lambda_{0,v}}\right)\right)\,dr\nonumber\\
  &=\frac{q_1 a}{\lambda_1}\int_{0}^\infty \mathbf{1}_{r\geq v^{\lambda_1/\lambda_{0,v}}}\frac{yr^{\frac{\lambda_{0,v}}{\lambda_1}-1}}{q_1 r+y}e^{-\lambda_0 t_v(r)}Z_{0,v}\left(t_v(r)\right)\,dr\nonumber
\end{align}
where $t_v(r)=\frac{1}{\lambda_1}\log\left(rv^{-\lambda_1/\lambda_{0,v}}\right)$. Now, we can approximate $e^{-\lambda_{0,v}t_v(r)}Z_{0,v}(t_v(r))$ via $Y_v$ in the above integral as follows:
\begin{align*}
    &\mathbb{E}\left[\left|\frac{q_1 a}{\lambda_1}\int_{0}^\infty \mathbf{1}_{r\geq v^{\lambda_1/\lambda_{0,v}}}\frac{yr^{\frac{\lambda_{0,v}}{\lambda_1}-1}}{q_1 r+y}\left(e^{-\lambda_0 t_v(r)}Z_{0,v}\left(t_v(r)\right)-Y_v\right)\,dr\right|\right]\\
    &\leq \frac{q_1 a}{\lambda_1}\int_{0}^\infty \mathbf{1}_{r\geq v^{\lambda_1/\lambda_{0,v}}}\frac{yr^{\frac{\lambda_{0,v}}{\lambda_1}-1}}{q_1 r+y}\mathbb{E}\left[\left|e^{-\lambda_0 t_v(r)}Z_{0,v}\left(t_v(r)\right)-Y_v\right|\right]\,dr\\
    &\leq \frac{q_1 a}{\lambda_1}\int_{0}^\infty \mathbf{1}_{r\geq v^{\lambda_1/\lambda_{0,v}}}\frac{yr^{\frac{\lambda_{0,v}}{\lambda_1}-1}}{q_1 r+y}\sqrt{\mathbb{E}\left[\left|e^{-\lambda_0 t_v(r)}Z_{0,v}\left(t_v(r)\right)-Y_v\right|^2\right]}\,dr\\
    &= \frac{q_1 a}{\lambda_1}\sqrt{\frac{a(1-v)+d+\mu}{\lambda_0}}\int_{0}^\infty \mathbf{1}_{r\geq v^{\lambda_1/\lambda_{0,v}}}\frac{yr^{\frac{\lambda_{0,v}}{\lambda_1}-1}}{q_1 r+y}e^{-\frac{\lambda_0}{2}t_v(r)}\,dr\\
    &\leq  v^{1/2}\frac{q_1 a}{\lambda_1}\sqrt{\frac{a(1-v)+d+\mu}{\lambda_0}}\int_{0}^\infty\frac{yr^{\frac{\lambda_{0,v}}{2\lambda_1}-1}}{q_1 r+y}\,dr\overset{v\to 0^+}{\to} 0.
\end{align*}
The equality follows from Lemma \ref{lem:L2_distance}. Remark that $Y_v$ is $0$ with probability $p_{0,v}$ and exponential with rate $q_{0,v}$ with probability $q_{0,v}$, so it is easy to deduce that $Y_v\Rightarrow Y_0$ in distribution as $v\to 0^+$. Thus:
\begin{align}
    -xY_v+\frac{q_1 aY_v}{\lambda_1}\int_{0}^\infty \mathbf{1}_{r\geq v^{\lambda_1/\lambda_{0,v}}}\frac{yr^{\frac{\lambda_{0,v}}{\lambda_1}-1}}{q_1 r+y}\,dr\Rightarrow -xY_0+\frac{q_1 a Y_0}{\lambda_1}\int_{0}^\infty \frac{yr^{\frac{\lambda_{0,0}}{\lambda_1}-1}}{q_1 r+y}\,dr.\label{eq:limit_laplace}
\end{align}
Thus by an application of Slutsky's theorem and remarking that for $z\geq 0$ and $X$ a non-negative random variable $e^{-zX}\leq 1$ is bounded, we obtain:
\begin{align}
  M(x,v^{-\lambda_1/\lambda_{0,v}}y)&\to \mathbb{E}\left[\exp\left(-Y_0x-Y_0\frac{q_1 a}{\lambda_1}\int_0^\infty\frac{yr^{\frac{\lambda_0}{\lambda_1}-1}}{q_1r+y}\,dr\right)\right]=\mathbb{E}\left[\exp\left(-Y_0x-Y_0\frac{q_1 a}{\lambda_1}\frac{y^{\lambda_0/\lambda_1}}{q_1^{\lambda_0/\lambda_1}}\frac{\pi}{\sin\left(\pi\lambda_0/\lambda_1\right)}\right)\right]\nonumber\\
  &=p_{0,0}+\frac{q_{0,0}^2}{q_{0,0}+x+y^{\lambda_{0,0}/\lambda_1}\frac{ aq_1 ^{1-\lambda_{0,0}/\lambda_1}}{\lambda_1}\frac{\pi}{\sin(\pi\lambda_{0,0}/\lambda_1)}}:=M_0(x,y).\nonumber
\end{align}
Also note that on $(\Omega_1^{\infty})^c$, $(Y,W)=(0,0)$ a.s. Hence:
\begin{align}
  \mathbb{E}[e^{-xY_v-yW_v}|\Omega_1^\infty]=\frac{1}{\mathbb{P}(\Omega_1^\infty)}\left(M(x,y)-\mathbb{P}(\Omega_1^{\infty,c})\right).\label{eq:trivial_cond}
\end{align}

Now, by Eq. \ref{eq:trivial_cond}, $(Y_v,v^{-\lambda_1/\lambda_{0,v}}W_v)|\Omega_{1,v}^\infty$ converges in distribution. But then since $\Omega_{0,v}^\infty\subseteq \Omega_{1,v}^\infty$ (up to a probability 0 set) and $\mathbb{P}(\Omega_{1,v}^\infty\Delta \Omega_{0,v}^\infty)\to 0$ as $v\to 0^+$ (where $\Delta$ denotes the symmetric difference) we obtain $(Y_v,v^{-\lambda_1/\lambda_{0,v}}W_v)|\Omega_{0,v}^\infty$ converges with the Laplace Transform specified in the proposition statement.
\end{proof}
\begin{remark}
  The marginal laplace transforms (setting $x=0$ or $y=0$ in Eq. \ref{eq:laplace_transform_limit}) agrees with the Mittag-Leffler distribution derived in a semi-deterministic version of our model in \cite{nicholson2023}.
\end{remark}

\begin{corollary}
  Eq. \ref{eq:small_freq_small_mutation} holds.
\end{corollary}
\begin{proof}
  Denote $(Q_1,Q_2)\sim M_0(x,y)$. We characterize $C'\Gamma(\alpha)Q_1Q_2^{-\lambda_0/\lambda_1}$ where $C'=aq_1^{1-\lambda_{0,0}/\lambda_1}/\lambda_1$. First, we note that:
\begin{align*}
  \mathbb{E}[(Q_1Q_2^{-\lambda_0/\lambda_1})^n]&=\frac{(-1)^n}{\Gamma(n\lambda_0/\lambda_1)}\int_0^\infty \partial_x^nM_0(0,y)y^{n\lambda_0/\lambda_1-1}\,dy\\
  &=\frac{n!}{(C'\pi/\sin(\pi\alpha))^n\Gamma(1+\alpha n)}.
\end{align*}
Thus:
\begin{align*}
  \mathbb{E}[\exp(-u C'\Gamma(\alpha)Q_1Q_2^{-\lambda_0/\lambda_1})]=\sum_{n=0}^\infty(-1)^n \frac{u^n}{(\pi/(\Gamma(\alpha)\sin (\pi\alpha)))^n \Gamma(1+\alpha n)}=E_\alpha(-u/(\pi/(\Gamma(\alpha)\sin (\pi\alpha)))).
\end{align*}
Therefore, by continuous mapping and Proposition \ref{prop:joint_weak_conv}:
\begin{align*}
  \frac{avq_1^{1-\lambda_{0,v}/\lambda_1}}{\lambda_1}\Gamma(\frac{\lambda_{0,v}}{\lambda_1})Y_vW_v^{-\lambda_{0,v}/\lambda_1}|\Omega_0^\infty&=\frac{aq_1^{1-\lambda_{0,v}/\lambda_1}}{\lambda_1}\Gamma(\frac{\lambda_{0,v}}{\lambda_1})Y_v\left(v^{-\lambda_1/\lambda_{0,v}}W\right)^{-\lambda_{0,v}/\lambda_1}|\Omega_0^\infty\\
  &\overset{v\to 0^+}{\Longrightarrow}C'\Gamma(\alpha)Q_1Q_2^{-\lambda_0/\lambda_1},
\end{align*}
as desired.
\end{proof}

\begin{proposition}
  Eq. \ref{eq:large_freq_small_mutation} holds.
\end{proposition}
\begin{proof}
  Let $T_i'$ denote the time of forming the $i$th established mutant lineage. Let $Z_1^{(i)'}$ denote the corresponding established type-1 clone, and let $Y_1^{(i)'}=\lim_{t\to\infty}e^{-\lambda_1 t}Z_1^{(i)'}(t)$. Then $\sum \delta_{(T_i',Y_1^{(i)'})}$ conditional on $\mathcal{F}_\infty$ is a Poisson process with intensity $q_1 avZ_0(s)(q_1e^{-q_1 x}\,dx)\,ds$ by the marking theorem. Let $J_v^{(i)}=v^{-\lambda_1/\lambda_{0}}e^{-\lambda_1 T_i'}Y_1^{(i)'}$ and define $\mathcal{N}_v=\sum \delta_{J_v^{(i)}}$. We take this to be a point process over $((0,\infty],D)$ with the vague topology. This metric space is complete and separable using $D(x,y)=\left|\frac{1}{x}-\frac{1}{y}\right|$. In what follows we will first show that $\mathcal{N}_v$ converges weakly to a point process $\mathcal{N}_0$ as $v\to 0^+$, and then use the continuous mapping theorem to establish the convergence of the normalized maximum clone size. 
  
First note that $v^{-\lambda_1/\lambda_{0}}W=\sum J_v^{(i)}$. Let $f\in C_{0,b}((0,\infty])$ -- namely, $f:((0,\infty],r)\to \mathbb{R}_{\geq 0}$ is a continuous function which is bounded and has bounded support (that is, is $0$ close to $0$). We define further $h:(0,\infty]\to \mathbb{R}$ via $h(j)=\theta j+f(j)$. Call $L_v(\theta,f):[0,\infty)\times C_{0,b}((0,\infty])\to \mathbb{R}_{\geq 0}$ the joint laplace function of $(v^{-\lambda_1/\lambda_0}W,\mathcal{N}_v)|\Omega_{0,v}^\infty$:

\begin{align*}
  L_v(\theta,f)&=\mathbb{E}\left[\exp\left(-\theta v^{-\lambda_1/\lambda_{0}}W-\sum f(J_v^{(i)})\right)\middle| \Omega_{0,v}^\infty\right]=\mathbb{E}\left[\exp\left(-\sum h(J_v^{(i)})\right)\middle| \Omega_{0,v}^\infty\right]\\
  &=\mathbb{E}\left[\exp\left(-avq_1\int_0^\infty \int_0^\infty (1-e^{-h(v^{-\lambda_1/\lambda_0}e^{-\lambda_1 s}x)})q_1 e^{-q_1 x}\,dxZ_0(s)\,ds\right)\middle| \Omega_{0,v}^\infty\right]\\
  &=\mathbb{E}\left[\exp\left(-avq_1\int_0^\infty \int_0^\infty (1-e^{-h(v^{-\lambda_1/\lambda_0}e^{-\lambda_1 s}x)})Z_0(s)\,ds q_1 e^{-q_1 x}\,dx\right)\middle| \Omega_{0,v}^\infty\right].
\end{align*}
Let $r=v^{-\lambda_1/\lambda_0}e^{-\lambda_1 s}x$, hence $s(r)=\lambda_1^{-1}\log r^{-1}v^{-\lambda_1/\lambda_0}x$

\begin{align*}
  &=\mathbb{E}\left[\exp\left(-\lambda_1^{-1}avq_1\int_0^\infty \int_0^{v^{-\lambda_1/\lambda_0}x} (1-e^{-h(r)})e^{\lambda_0 s(r)}e^{-\lambda_0 s(r)}Z_0(s(r))\,\frac{dr}{r} q_1 e^{-q_1 x}\,dx\right)\middle| \Omega_{0,v}^\infty\right]\\
  &=\mathbb{E}\left[\exp\left(-\lambda_1^{-1}aq_1\int_0^\infty \int_0^{v^{-\lambda_1/\lambda_0}x} (1-e^{-h(r)})r^{-\lambda_0/\lambda_1}x^{\lambda_0/\lambda_1}e^{-\lambda_0 s(r)}Z_0(s(r))\,\frac{dr}{r} q_1 e^{-q_1 x}\,dx\right)\middle| \Omega_{0,v}^\infty\right]\\
  &\overset{v\to 0^+}{\to} \mathbb{E}\left[\exp\left(-\lambda_1^{-1}aq_1\int_0^\infty \int_0^{\infty} (1-e^{-h(r)})r^{-\lambda_0/\lambda_1}x^{\lambda_0/\lambda_1}e^{-\lambda_0 s(r)}Y\,\frac{dr}{r} q_1 e^{-q_1 x}\,dx\right)\middle| \Omega_{0,0}^\infty\right]\\
  &=\mathbb{E}\left[\exp\left(-C\cdot Y\int_0^\infty(1-e^{-h(r)})r^{-1-\lambda_0/\lambda_1}\,dr\right)\middle|\Omega_{0,0}^\infty\right],
\end{align*}
where the convergence is justified by similar arguments as in Proposition \ref{prop:conv_dist}. Hence $(v^{-\lambda_1/\lambda_0}W,\mathcal{N}_v)|\Omega_{0,v}^\infty$ converges in distribution. We thus conclude $\mathcal{N}_v$ converges to a point process $\mathcal{N}_0$. From the Laplace functional we conclude that $\mathcal{N}_0$ is doubly-stochastic, where conditional on $Y$, it is a Poisson point process with intensity of a $\lambda_{0,0}/\lambda_1$-subordinator. 

 We define the function $G:\mathbb{R}_{>0}\times N((0,\infty])\to \mathbb{R}$ via:
\[
G(s,\mu)=\frac{\sup\text{supp}(\mu)}{s}.
\]
 By continuous mapping (see Lemma \ref{lem:G_cont}):
\[
\frac{\sup\left(e^{-\lambda_1 T_i}Y_1^{(i)}\right)}{W}|\Omega_0^\infty=\frac{\sup\left(v^{-\lambda_1/\lambda_0}e^{-\lambda_1 T_i}Y_1^{(i)}\right)}{v^{-\lambda_1/\lambda_0}W}|\Omega_0^\infty\overset{v\to 0^+}{\Longrightarrow} \sup\left(\frac{\text{supp}(\mathcal{N}_0)}{\sum_{x\in \mathcal{N}_0}x}\right)=V.
\]
By Pitman and Yor \cite{pitman1997}, the above is the largest point of a $\text{PD}(\alpha,0)$ distribution.
\end{proof}

\begin{lemma}\label{lem:G_cont}
  Let $N((0,\infty])$ denote the space of boundedly-finite counting measures with the vague topology. The function $G:\mathbb{R}_{>0}\times N((0,\infty])\to (0,\infty]$ defined by:
  \begin{align*}
    G(s,\mu)=\frac{\sup\text{supp}(\mu)}{s},
  \end{align*}
  is continuous at points $s,\mu$ where $\mu$ is non-trivial. Here $(0,\infty]$ is equipped with the metric $D(x,y)=|1/x-1/y|$.
\end{lemma}
\begin{proof}
  See Appendix Section \ref{app:G_cont}.
\end{proof}

\section{Acknowledgements}
The work of K. Leder and J. Foo was partially supported by NSF DMS award 2526511. The work of J. Foo and A. Ahmed was partially supported by CMMI 2228034.  

\printbibliography

\appendix

\section{Appendix}

\subsection{Large-frequency scaling of the SFS} \label{app:SFS_mean_j_scaling}
Here, we prove Eqs. \ref{eq:power_series_representation_S_j}, \ref{eq:large_j_equivalent}, and \ref{eq:large_j_equivalent_random}. Our proofs depend on the explicit form of $P_j(s)$ (for $j\geq 1$), which is (see \cite{durrettbranching2015}):
\begin{align}
    P_j(t)=\left(1-\frac{de^{\lambda_1 t}-d}{(a+b)e^{\lambda_1 t}-d}\right)\left(1-\frac{(a+b)e^{\lambda_1 t}-(a+b)}{(a+b)e^{\lambda_1 t}-d}\right)\left(\frac{(a+b)e^{\lambda_1 t}-(a+b)}{(a+b)e^{\lambda_1 t}-d}\right)^{j-1}\label{eq:P_j}
\end{align}
\begin{proof}[Proof of Eqs. \ref{eq:power_series_representation_S_j}, \ref{eq:large_j_equivalent}, and \ref{eq:large_j_equivalent_random}]
  Using the definition of $P_j(s)$ we can write the integral (Eq. \ref{eq:asymp_mean}) as
$$
\int_0^\infty e^{-\lambda_0 s}P_j(s)\,ds=q_1^2\int_0^\infty \frac{e^{-s(\lambda_0+\lambda_1)}}{(1-p_1e^{-\lambda_1s})^2}\left(\frac{1-e^{-\lambda_1s}}{1-p_1e^{-\lambda_1s}}\right)^{j-1}ds.
$$
Using the changes of variable $x=e^{-\lambda_1 s}$ and $y=\frac{1-x}{1-p_1x}$, we can rewrite the previous display as
\begin{align}
&\frac{q_1^2}{\lambda_1}\int_0^1\frac{x^{\lambda_0/\lambda_1}}{(1-p_1x)^2}\left(\frac{1-x}{1-p_1x}\right)^{j-1}dx=\frac{q_1}{\lambda_1}\int_0^1\left(\frac{1-y}{1-p_1y}\right)^{\lambda_0/\lambda_1}y^{j-1}dy.\label{eq:SFS_mean_integral}
\end{align}
One way of analyzing this integral is to relate it to a hyper-geometric series. Noticing that $(1-y)^{-a}=\sum_{n=0}^\infty \frac{(a)_{n}}{n!}y^n$, where for positive number $(a)_{n}=\Gamma(a+n)/\Gamma(a)$ denotes the rising factorial, we can rewrite the integral on the right-hand-side of the previous display as:
\begin{align}
    \label{eq:ExpectedSFS_Series}
\frac{\Gamma(j)\Gamma(\lambda_0/\lambda_1+1)}{\Gamma(j+\lambda_0/\lambda_1+1)}\sum_{n=0}^\infty \frac{(\lambda_0/\lambda_1)_{n}(j)_{n}}{(j+1+\lambda_0/\lambda_1)_{n}n!}p_1^n.
\end{align}
This shows Eq. \ref{eq:power_series_representation_S_j}. We now show the leading order asymptotic in Eq. \ref{eq:large_j_equivalent}. By Watson's Lemma, noting the cofactor to the sum is the beta function ($B(j,\lambda_0/\lambda_1+1)$):

\[\frac{\Gamma(j)\Gamma(\lambda_0/\lambda_1+1)}{\Gamma(j+\lambda_0/\lambda_1+1)}=\frac{\Gamma(\lambda_0/\lambda_1+1)}{j^{\lambda_0/\lambda_1+1}}+\mathcal{O}\left(\frac{1}{j^{\lambda_0/\lambda_1+2}}\right).\]
Since the rising factorial function is increasing we know that
$$
\frac{(\lambda_0/\lambda_1)_{n}(j)_{n}}{(j+1+\lambda_0/\lambda_1)_{n}n!}p_1^n\leq \frac{(\lambda_0/\lambda_1)_n}{n!}p_1^n,
$$
and furthermore basic asymptotics of the Gamma function give us that as $j\to\infty$ the summands of \eqref{eq:ExpectedSFS_Series} converge to $\frac{(\lambda_0/\lambda_1)_n}{n!}p_1^n.$
Thus by dominated convergence we can conclude
\[\int_0^1 \left(\frac{1-y}{1-p_1y}\right)^{\lambda_0/\lambda_1}y^{j-1}\,dy\sim \frac{\Gamma(\lambda_0/\lambda_1+1)}{j^{\lambda_0/\lambda_1+1}}\sum_{n=0}^\infty \frac{(\lambda_0/\lambda_1)_n}{n!}(p_1)^n=\frac{\Gamma(\lambda_0/\lambda_1+1)}{j^{\lambda_0/\lambda_1+1}}(1-p_1)^{-\lambda_0/\lambda_1}.\]
For the deterministic fitness advance case, it remains to obtain the error term $\mathcal{O}(j^{-2-\lambda_0/\lambda_1})$. Now, to get a rate of convergence in terms of $j$ as $j\to\infty$ on:

\[\sum_{n=0}^\infty \frac{(\lambda_0/\lambda_1)_n(j)_n}{(j+1+\lambda_0/\lambda_1)_nn!}p_1^n\]
it suffices to examine how quickly

\[\sum_{n=0}^\infty \frac{(\lambda_0/\lambda_1)_n}{n!}p_1^n\left(1-\frac{(j)_n}{(j+1+\lambda_0/\lambda_1)_n}\right)\to 0\]
as $j\to\infty$. Notice that:

\begin{align*}
  j\left(1-\frac{(j)_n}{(j+1+\lambda_0/\lambda_1)_n}\right)&=j\cdot \left(\frac{(j+1+\lambda_0/\lambda_1)_n-(j)_n}{(j+1+\lambda_0/\lambda_1)_n}\right)\\
  &\overset{j\to\infty}{\sim} j\left(\frac{(j+1+\lambda_0/\lambda_1)^n-j^n}{(j+1+\lambda_0/\lambda_1)^n}\right)\\
  &\overset{j\to\infty}{\sim} j\left((1+\frac{1+\lambda_0/\lambda_1}{j})^n-1\right)\\
  &= j\left(n\frac{1+\lambda_0/\lambda_1}{j}+\mathcal{O}(j^{-2})\right)\\
  &\overset{j\to\infty}{\sim} n(1+\lambda_0/\lambda_1)
\end{align*}
Notably then $f_j(n)=j\frac{(\lambda_0/\lambda_1)_n}{n!}p_1^n\left(1-\frac{(j)_n}{(j+1+\lambda_0/\lambda_1)_n}\right)$ has a pointwise limit as $j\to\infty$. Note that:

\[|f_j(n)|=\left|\frac{(\lambda_0/\lambda_1)_n}{n!}p_1^n\right|\cdot j\cdot \left|\left(1-\prod_{k=0}^{n-1} \frac{(j+k)}{(j+1+\lambda_0/\lambda_1+k)}\right)\right|\]
Using the Weierstrass product inequality:

\begin{align*}
  &=|\frac{(\lambda_0/\lambda_1)_n}{n!}p_1^n|\cdot j\cdot |\left(1-\prod_{k=0}^{n-1} \left(1-\frac{(1+\lambda_0/\lambda_1)}{(j+1+\lambda_0/\lambda_1+k)}\right)\right)|\\
  &\leq |\frac{(\lambda_0/\lambda_1)_n}{n!}p_1^n|\cdot j\cdot \sum_{k=0}^{n-1}\frac{(1+\lambda_0/\lambda_1)}{(j+1+\lambda_0/\lambda_1+k)}\\
  &\leq |\frac{(\lambda_0/\lambda_1)_n}{n!}p_1^n|\cdot n(1+\lambda_0/\lambda_1)
\end{align*}
which is summable against $n$ (and upper bounds $f_j$ for all $j$), hence we can use dominated convergence to interchange $\lim_{j\to\infty}$ and the sum in $\lim_{j\to\infty}\sum_{n=0}^\infty j\frac{(\lambda_0/\lambda_1)_n}{n!}p_1^n\left(1-\frac{(j)_n}{(j+1+\lambda_0/\lambda_1)_n}\right)$. We further notice that:

\[\lim_{j\to\infty}\sum_{n=0}^\infty j\frac{(\lambda_0/\lambda_1)_n}{n!}p_1^n\left(1-\frac{(j)_n}{(j+1+\lambda_0/\lambda_1)_n}\right)=\sum_{n=0}^\infty \frac{(\lambda_0/\lambda_1)_n}{n!}p_1^nn(1+\lambda_0/\lambda_1)=L\in (0,\infty)\]

This then implies that the original sum of interest

\[\sum_{n=0}^\infty \frac{(\lambda_0/\lambda_1)_n(j)_n}{(j+1+\lambda_0/\lambda_1)_n(n)!}p_1^n\]

converges at rate $\mathcal{O}(j^{-1})$ to its limit as $j\to\infty$. We thus finally obtain:

\begin{align*}
  \lim_{t\to\infty}e^{-\lambda_0 t}\mathbb{E}[S_j(t)]&=(av+\mu)\frac{q_1}{\lambda_1}\int_0^1 \left(\frac{1-y}{1-p_1y}\right)^{\lambda_0/\lambda_1}y^{j-1}\,dy\\
  &=(av+\mu)\frac{q_1\Gamma(\lambda_0/\lambda_1+1)}{\lambda_1j^{\lambda_0/\lambda_1+1}}(1-p_1)^{-\lambda_0/\lambda_1}+\mathcal{O}\left(\frac{1}{j^{\lambda_0/\lambda_1+2}}\right),
\end{align*}
which proves Eq. \ref{eq:large_j_equivalent}.

In order to show the asymptotic result in the random fitness advance model (Eq. \ref{eq:large_j_equivalent_random}), we in fact use \ref{eq:large_j_equivalent}. Let $\widetilde{C}(b)=(av+\mu)\frac{q_1(b)\Gamma(\lambda_0/\lambda_1(b)+1)}{\lambda_1(b)}(1-p_1(b))^{-\lambda_0/\lambda_1(b)}h(b)$. From Proposition \ref{prop:SFS_mean_asymp} and Fubini-Tonelli:
  \begin{align}
      \lim_{t\to\infty}e^{-\lambda_0 t}\mathbb{E}[S_j^{B}(t)]&=\mathbb{E}\left[(av+\mu)\frac{q_1(B)\Gamma(\lambda_0/\lambda_1(B)+1)}{\lambda_1(B)j^{\lambda_0/\lambda_1(B)+1}}(1-p_1(B))^{-\lambda_0/\lambda_1(B)}+\mathcal{O}\left(\frac{1}{j^{\lambda_0/\lambda_1(B)+2}}\right)\right]\nonumber\\
      &\overset{j\to\infty}{\sim} \int_0^{b_{\text{max}}}\widetilde{C}(b)e^{-\left(\lambda_0/\lambda_1(b)+1\right)\log j}\,db.\label{eq:large_time_mean_random_goal}
  \end{align}
 Therefore, we examine the rate at which \ref{eq:large_time_mean_random_goal} goes to $0$ as $j\to\infty$. Let $\phi(b)=\lambda_0/\lambda_1(b)+1$ and perform the change of variables $\ell=\phi(b)-\phi(b_{\max})$. Let $C(\ell):=\frac{\tilde{C}(b(\ell))}{|\phi'(b(\ell))|}$ and let $C:=C(0)$, where $C$ is as in the proposition statement. Then:

\[\int_0^{b_{\text{max}}}\tilde{C}(b)e^{-\left(\lambda_0/\lambda_1(b)+1\right)\log j}\,db=e^{-\phi(b_{\max})\log j}\int_0^{\ell_{\text{max}}}C(\ell)e^{-\ell\log j}\,d\ell.\]
$C(\ell)$ remains continuous at $0$ with $C(0)>0$ and is bounded over $[0,\ell_{\max}]$ by quantity $M$ from Assumption \ref{assumption:random}. Fix $\varepsilon>0$ and choose $\delta>0$ such that the $\sup_{-\delta<y<\delta}|C(0)-C(y)|<\varepsilon$ (note $\delta$ is independent of $j$).
Then:

\[(C(0)-\varepsilon)(1-e^{-\delta\log j})\leq (\log j)\int_0^\delta C(\ell)e^{-\ell \log j}\,d\ell\leq (C(0)+\varepsilon)(1-e^{-\delta\log j}).\]
Also notice that $(\log j)\int_\delta^{\ell_{\max}} C(\ell)e^{-\ell \log j}\,d\ell\to 0$ as $j\to\infty$ for $\delta>0$ since $C$ is bounded on the interval $[0,\ell_{max}].$
We thus have that
\[C(0)-\varepsilon\leq \liminf_{j\to\infty}(\log j)\int_0^{\ell_{\max}} C(\ell)e^{-\ell \log j}\,d\ell \leq \limsup_{j\to\infty}(\log j)\int_0^{\ell_{\max}} C(\ell)e^{-\ell \log j}\,d\ell\leq C(0)+\varepsilon\]
and since $\varepsilon$ is arbitrary, we conclude
\[e^{-\phi(b_{\max})\log j}\int_0^{\ell_{\text{max}}}C(\ell)e^{-\ell\log j}\,d\ell\sim \frac{C(0)}{j^{\phi(b_{\max})}\log j}.\]
\end{proof}

\subsection{Results for mutation timings}\label{app:mutation_timings}
\subsubsection{Proof of Lemma \ref{lem:V_moments_finite}}\label{app:V_moments_finite}
\begin{proof}
    We prove this result by induction on $k$. In the $k=1$ case, the result follows by Campbell's theorem. Next assume the induction hypothesis $\mathbb{E}[V^{i}]<\infty$ for $1\leq i\leq k-1$. Let $V(t):=\sum_{i:T_i\leq t}e^{-2\lambda_0 T_i}$. By examining what occurs during the first event (occurring at time $\tau$, with $E_b,E_d,E_{v},E_{\mu}$ denoting birth without mutation, death, replication-dependent mutation, and replication-free mutation), we may write:
    \begin{align*}
        V(t)&\overset{d}{=}\mathbf{1}_{E_b}e^{-2\lambda_0 \tau}(V_1(t-\tau)+V_2(t-\tau))+\mathbf{1}_{E_v}(e^{-2\lambda_0 \tau}+e^{-2\lambda_0 \tau}V_1(t-\tau))+\mathbf{1}_{E_\mu}e^{-2\lambda_0 \tau}\\
        &\leq\mathbf{1}_{E_b}e^{-2\lambda_0 \tau}(V_1(t)+V_2(t))+\mathbf{1}_{E_v}(e^{-2\lambda_0 \tau}+e^{-2\lambda_0 \tau}V_1(t))+\mathbf{1}_{E_\mu}e^{-2\lambda_0 \tau}
    \end{align*}
    where $V_1,V_2$, conditional on the first event, are independent and distributed as $V$. Taking powers and expectation, we obtain:
    \begin{align*}
        \mathbb{E}[V(t)^k]\leq \frac{a(1-v)\mathbb{E}[(V_1(t)+V_2(t))^k]}{a+d+\mu+2\lambda_0 k}+\frac{av\mathbb{E}[(1+V(t))^k]}{a+d+\mu+2\lambda_0 k}+\frac{\mu}{a+d+\mu+2\lambda_0 k}.
    \end{align*}
    Let $C_{<k}$ denote some linear combination of terms of the form $\mathbb{E}[V^{i_1}]\mathbb{E}[V^{i_2}]$ where $0\leq i_1,i_2<k$. We can rewrite the above using the leading order terms:
    \begin{align*}
        &\mathbb{E}[V^k](t)\leq \frac{(a+a(1-v))\mathbb{E}[V(t)^k]}{a+d+\mu+2\lambda_0 k}+C_{<k}\\
        \Longrightarrow &\mathbb{E}[V(t)^k]\leq  \frac{a+d+\mu+2\lambda_0 k}{(2k-1)\lambda_0}C_{<k}<\infty.
    \end{align*}
    The right-hand-side is independent of time, so by monotone convergence $\mathbb{E}[V^k]=\lim_{t\to\infty}\mathbb{E}[V(t)^k]<\infty$, as desired.
\end{proof}
\subsubsection{Proof of Lemma \ref{lem:sup_M_finite}}\label{app:sup_M_finite}
\begin{proof}
  $\sup_{z\in [0,\infty)}\left(M_z e^{-\lambda_0 z}\right)$ has all finite moments. To see this, recall from Lemma \ref{lem:CenteredMuts_Martingale} the decomposition $M_z=R_z+(av+\mu)\int_0^z Z_0(s)\,ds$ where $R_z$ is a $\mathcal{B}_z$-martingale. It suffices to show then that $\sup_{z\in [0,\infty)}e^{-\lambda_0 z}|R_z|$ and $\sup_{z\in [0,\infty)}e^{-\lambda_0 z}\int_0^z Z_0(s)\,ds$ have finite moments. For the latter term:

    \begin{align*}
        \sup_{z\in [0,\infty)}e^{-\lambda_0 z}\int_0^z Z_0(s)\,ds\leq \sup_{s\in [0,\infty)}(e^{-\lambda_0 s}Z_0(s))\sup_{z\in [0,\infty)}\frac{1}{\lambda_0}(1-e^{-\lambda_0 z})=\frac{1}{\lambda_0}\sup_{z\in [0,\infty)}e^{-\lambda_0 s}Z_0(s). 
    \end{align*}
The $p$th moment of the upper bound is controlled by Doob's maximal inequality:
    \begin{align*}
        \mathbb{E}\left[\sup_{z\in [0,\infty)}\left(e^{-\lambda_0 s}Z_0(s)\right)^p\right]\leq C_p\sup_{t\in [0,\infty)}\mathbb{E}[\left(e^{-\lambda_0 t}Z_0(t)\right)^p]<\infty,
    \end{align*}
    where the final inequality follows from e.g. Lemma 5 of \cite{foo2014escape}.
To control the term $e^{-\lambda_0 z}|R_z|$, define the $\mathcal{B}_t$-martingale $Q_t=\int_0^t e^{-\lambda_0 s}dR_s$. Using integration by parts:
    \begin{align*}
        |R_t|=\left|\int_0^t e^{\lambda_0 s}\,dQ_s\right|\leq |Q_t|e^{\lambda_0 t}+\lambda_0 \int_0^t \left|Q_{s-}\right|e^{\lambda_0 s}\,ds.
    \end{align*}
    Multiplying by $e^{-\lambda_0 t}$ and taking a supremum over $z\in [0,\infty)$:
    \begin{align*}
        \sup_{z\in [0,\infty)} e^{-\lambda_0 z}|R_z|\leq 3\sup_{z\in [0,\infty)}|Q_z|.
    \end{align*}
    We can control the upper bound using the BDG-inequalities, noting that $[Q]_t=\int_0^t e^{-2\lambda_0 s}d[R]_s=\int_0^t e^{-2\lambda_0 s}\,dM_s$:
    \begin{align*}
        \mathbb{E}[\sup_{z\in [0,\infty)}|Q_z|^p]\leq C_p \mathbb{E}\left[\left(\sum_{T_i}e^{-2\lambda_0 T_i}\right)^{p/2}\right].
    \end{align*}
    The last term is finite by using Lemma \ref{lem:V_moments_finite} and applying Jensen's inequality.
\end{proof}

\subsection{Tail bounds on population sizes and the SFS}
\subsubsection{Proof of Lemma \ref{lem:large_deviation_bound}}\label{app:large_deviation_bound}
\begin{proof}
    Remark that on $(\Omega_0^\infty)^c$, $\inf_t Z_0(t)e^{-\lambda_0 t}\leq x$ occurs (for any $x>0$). Thus we may write:
    \begin{align}
        \mathbb{P}(\inf_t Z_0(t)e^{-\lambda_0t}\leq x|\Omega_0^\infty)=\frac{\mathbb{P}(\inf_t Z_0(t)e^{-\lambda_0t}\leq x)-p_0}{q_0}\label{eq:large_deviation_rewrite}.
    \end{align}
    For each $\alpha>0$, we have $e^{-\alpha Z_0(t)e^{-\lambda_0 t}}$ is a $\mathcal{F}_t$-submartingale. So using Doob's submartingale inequality:
    \[\mathbb{P}(\inf_t Z_0(t)e^{-\lambda_0t}\leq x)=\mathbb{P}(\sup_t e^{-\alpha Z_0(t)e^{-\lambda_0 t}}\geq e^{-\alpha x})\leq e^{\alpha x}\mathbb{E}[e^{-\alpha Y}]=e^{\alpha x}\left(p_0+\frac{q_0^2}{q_0+\alpha}\right)\]
    Notice we may optimize $\alpha$ in the upper bound by choosing $\alpha^*=\frac{\sqrt{q_0^4+\frac{4p_0q_0^2}{x}}}{2p_0}-\frac{q_0^2}{2p_0}-q_0$. In particular, there exists a positive constant $D$ such that:
    \[\inf_\alpha \left(e^{\alpha x}\left(p_0+\frac{q_0^2}{q_0+\alpha}\right)\right)=p_0+D\sqrt{x}+\mathcal{O}_{x\to 0^+}(x).\]
    The leading order term is constant, but by plugging the expression into \ref{eq:large_deviation_rewrite}, this term cancels which yields the result.
\end{proof}

\subsubsection{Proof of Lemma \ref{lem:upper_bound_Sj_strong}}\label{app:upper_bound_Sj_strong}
\begin{proof}
    $\xi$ represents the timescale at which a type-1 clone reaches size $j$. Define:
  \begin{align*}
    p_{j,\Delta,s\to t}=\int_0^\Delta (av+\mu)e^{-(a+d+\mu)z}P_j(t-s-z)\,dz
  \end{align*}
  as the probability that a single type-0 cell alive at time $s$ has a mutation event occur up to the first $\Delta$-time units of its lifetime and that the type-1 clone reaches size $j$ by time $t$. Here, we take $s=t-\xi-\Delta$. Let $S_{j,\Delta,s}(t)$ denote the number of type-1 clones of size $j$ at time $t$ produced by type-0 cells alive at time $s$ during $\Delta$-time units of their lifetime. Importantly:
  \begin{align*}
    S_{j,\Delta,s}(t)|\mathcal{F}_s\overset{d}{=}B(Z_0(s),p_{j,\Delta, s\to t}).
  \end{align*}
  We now lower bound both $Z_0(s),p_{j,\Delta,s\to t}$. On $B_x$, trivially $Z_0(s)\geq x\exp(\lambda_0 s)$. Note that for our chosen $s$ and $z\in[0,\Delta]$ and using Eq. \ref{eq:P_j}:
  \begin{align*}
    P_j(t-s-z)=P_j(\xi+\Delta-z)&\geq \frac{(a+b-d)j}{(a+b)je^{\lambda_1 \Delta}-d}\frac{a+b-d}{(a+b)je^{\lambda_1 \Delta}-d}\left(1-\frac{a+b-d}{(a+b)j-d}\right)^{j-1}\\
    &\geq \frac{\lambda_1^2}{(a+b)^2}e^{-2\lambda_1 \Delta-1}j^{-1}.
  \end{align*}
  Therefore:
  \begin{align*}
    p_{j,\Delta,s\to t}\geq C_\Delta j^{-1}.
  \end{align*}
  
   Observe that combining everything:
  \begin{align*}
    \mathbb{P}(S_j(t)\leq k,B_x)\leq \mathbb{P}(S_{j,\Delta,s}(t)\leq k,Z_0(s)\geq xe^{\lambda_0 s})&\leq \mathbb{P}(B(\lfloor xe^{\lambda_0 s}\rfloor,C_\Delta j^{-1})\leq k)
  \end{align*}
  and using a Chernoff bound:
  \begin{align*}
    \leq \delta^{-k}e^{-(1-\delta)C_\Delta xe^{\lambda_0 s}j^{-1}}.
  \end{align*}
  Finally, substituting:
  \begin{align*}
    e^{\lambda_0 s}j^{-1}=e^{-\lambda_0 \Delta}e^{\lambda_0 t}j^{-1-\lambda_0/\lambda_1},
  \end{align*}
  concludes the result.
\end{proof}

\subsubsection{Proof of Lemma \ref{lem:lower_bound_Sj_strong}}\label{app:lower_bound_Sj_strong}
\begin{proof}
  Recall that $S_{j}(t)$, conditional on $\mathcal{B}_\infty$, is a sum of independent indicators, so that upon taking a factorial bound:
    \begin{align}
        \mathbb{P}(S_{j}(s)\geq k+1)\leq \mathbb{E}\left[\frac{\mathbb{E}[S_{j}(s)|\mathcal{B}_\infty]^{k+1}}{(k+1)!}\right].\label{eq:factorial_bound}
    \end{align}
    Now, remark that:
    \begin{align}
        \mathbb{E}[S_j(t)|\mathcal{B}_\infty]=\int_0^t P_j(t-s)\,dM_s=\int_0^t M_{s-}P_j'(t-s)\,ds\leq \left(\sup_{z\in [0,\infty)}M_z e^{-\lambda_0 z}\right)e^{\lambda_0 t}\int_0^\infty e^{-\lambda_0 s}(P_j'(s))_+\,ds.\label{eq:bound_conditional_Sj}
    \end{align}
    Now, using Eq. \ref{eq:P_j}, we can obtain:
    \begin{align*}
      P_j'(s)\geq \frac{P_1'(s)}{P_1(s)}P_j(s)\geq (-a+b+d)P_j(s)
    \end{align*}
    from which we can show:
    \begin{align*}
      (P_j'(s))_+=P_j'(s)+(P_j'(s))_-\leq P_j'(s)+CP_j(s).
    \end{align*}
    Then:
    \begin{align}
        \ref{eq:bound_conditional_Sj}&\leq  C\left(\sup_{z\in [0,\infty)} M_ze^{-\lambda_0 z}\right)e^{\lambda_0 t}\int_0^\infty e^{-\lambda_0 s}P_j(s)\,ds\nonumber\\
        &\leq C\sup_{z\in [0,\infty)} \left(M_ze^{-\lambda_0 z}\right) e^{\lambda_0 t}j^{-1-\frac{\lambda_0}{\lambda_1}},\label{eq:S_j_expectation_bound}
    \end{align}
    with the last step coming from the asymptotic for $\mathbb{E}[e^{-\lambda_0 t}S_j(t)]$ (Proposition \ref{prop:SFS_mean_asymp}). Remark that $\sup_{z\in [0,\infty)}\left(M_z e^{-\lambda_0 z}\right)$ by Lemma \ref{lem:sup_M_finite}, proving the claim.
\end{proof}

\subsection{Point process continuity result: Proof of Lemma \ref{lem:G_cont}}\label{app:G_cont}
\begin{proof}
  It suffices to show that $\mu\to \sup\text{supp}(\mu)$ is continuous at $\mu$ described in the lemma statement. Let $\mu\in N((0,\infty])$ be non-trivial. It has a largest element of the support ($x_{\max}$). Remark it cannot be a limit point of $\text{supp}(\mu)$, as $\mu$ is boundedly finite. Let $\mu_n\to\mu$ vaguely, and denote $x_{\max}^n=\sup\text{supp}(\mu_n)$. If $x_{\max}=\infty$, take a sequence $k_m\to\infty$ such that $k_m\not\in \text{supp}(\mu)$ for any $m$. Notice that $\mu_n((k_m,\infty])\to \mu((k_m,\infty])>0$. Therefore $\liminf_{n\to\infty}x_{\max}^n\geq k_m$. Taking $m\to\infty$ shows $x_{\max}^n\to x_{\max}$ as desired.
  
  Suppose $x_{\max}<\infty$. Note that for some $\varepsilon'>0$ with $x_{\max}>\varepsilon'$, $[x_{\max}-\varepsilon',x_{\max}+\varepsilon')$ contains only one element in $\text{supp}(\mu)$. Then for any $0<\varepsilon<\varepsilon'$, and for large enough $N$ with $n\geq N$ implying $\mu_n([x_{\max}-\varepsilon,x_{\max}+\varepsilon'))=1$, we can conclude that $x_{\max}^n\geq x_{\max}-\varepsilon$. Also $[x_{\max}+\varepsilon,\infty]$ (which is bounded under $D$) contains no element in $\text{supp}(\mu)$, thus for sufficiently large $n$ $x_{\max}^n \leq x_{\max}+\varepsilon$. Hence $\mu$ is a continuity point of $\sup\text{supp}(\mu)$.
\end{proof}

\end{document}